\documentclass[a4paper]{article}
\usepackage{setspace} %
\usepackage[backref=page,hypertexnames=false]{hyperref}
\usepackage[margin=2.5cm]{geometry}
\usepackage{comment}
\usepackage{mathtools}
\usepackage{amssymb}
\usepackage{amsthm}
\usepackage{thmtools}
\usepackage{empheq}
\usepackage{thm-restate}
\usepackage{blkarray, bigstrut}
\usepackage{booktabs}

\usepackage{enumitem}
\usepackage[utf8]{inputenc}

\usepackage[dvipsnames]{xcolor}
\usepackage{authblk}
\usepackage{amsmath}

\usepackage[short]{optidef}

\usepackage[linesnumbered,commentsnumbered, ruled, scleft, vlined]{algorithm2e}

\usepackage{cleveref}

\def\eatdot.{}

\usepackage{tikz}
\usetikzlibrary{arrows.meta}
\usetikzlibrary{arrows,shapes,decorations.pathmorphing,calc}
\usetikzlibrary{decorations.pathreplacing,angles,quotes,calligraphy}

\tikzstyle{vertex}=[circle,draw]
\tikzstyle{arc} = [draw,-{Stealth}]

\usepackage{esvect}

\usepackage{transparent}

\SetInd{0.0em}{0.6em}

\SetCommentSty{mycommfont}

\SetKwComment{Komm}{\quad $\triangleright$\ }{}
\SetKw{KwGoTo}{\textbf{go to}}

\SetAlFnt{\small}
\usepackage{algorithmic}
\algsetup{linenosize=\tiny}

\usepackage{newpxtext}
\usepackage{newpxmath}
\usepackage{accents}

\usepackage{multicol}
\newcommand{\norm}[1]{\left\lVert#1\right\rVert}

\newcommand{\R}{\mathbb{R}}
\newcommand{\Q}{\mathbb{Q}}
\newcommand{\Z}{\mathbb{Z}}
\newcommand{\Rplus}{\mathbb{R}_{\ge0}}
\newcommand{\Rplusplus}{\mathbb{R}_{>0}}
\newcommand{\Qplusplus}{\mathbb{Q}_{>0}}

\newcommand{\Zplusplus}{\mathbb{Z}_{>0}}

\newcommand{\eQ}{\widehat{\Q}}
\newcommand{\eQplusplus}{\widehat{\Q}_{>0}}
\newcommand{\val}{\operatorname{val}}

\newcommand{\mcp}{^{\mathfrak{m}}}
\newcommand{\m}{\mathfrak{m}}

\newcommand{\pr}[2]{\left\langle #1, #2 \right\rangle}
\def\<#1>{\mathinner{\langle#1\rangle}}

\newcommand{\Primal}{\mathcal{P}}
\newcommand{\Primalp}{\mathcal{P}_{>0}}
\newcommand{\Dual}{\mathcal{D}}
\newcommand{\Dualp}{\mathcal{D}_{>0}}

\newcommand{\poly}{\operatorname{poly}}
\renewcommand{\cal}{\mathcal}

\newcommand{\rank}{\operatorname{rank}}

\newcommand{\size}{\operatorname{size}}

\newcommand{\adjt}{\operatorname{adj}}
\newcommand{\Tmul}{\mathcal{T}_{\textup{mul}}}
\newcommand{\Tdet}{\mathcal{T}_{\textup{det}}}
\newcommand{\Tadj}{\mathcal{T}_{\textup{adj}}}
\newcommand{\Tinv}{\mathcal{T}_{\textup{inv}}}
\newcommand{\Tbasis}{\mathcal{T}_{\textup{basis}}}
\newcommand{\Tcirc}{\mathcal{T}_{\textup{circ}}}
\newcommand{\Tcr}{\mathcal{T}_{\textup{cr}}}

\newcommand{\TMW}{\mathcal{T}_{\textup{MW}}}

\newcommand{\out}{^\mathrm{out}}

\newcommand{\slceta}{\eta}

\newcommand{\im}{\operatorname{im}}

\renewcommand{\rho}{\varrho}

\newcommand{\1}{\mathbf 1}
\newcommand{\0}{\mathbf 0}

\newcommand{\cp}{^{\mathrm{cp}}}
\newcommand{\pz}{z}
 \newcommand{\px}{x}
 \newcommand{\ps}{s}

\newcommand{\mz}{z\mcp}
\newcommand{\mx}{x\mcp}
\newcommand{\bmx}{\bar{x}\mcp}
\newcommand{\ms}{s\mcp}
\newcommand{\bms}{\bar{s}\mcp}

\renewcommand{\phi}{\varphi}

\newcommand{\cN}{\mathcal{N}^2}
\newcommand{\clN}{\overline{\mathcal{N}}^2}

\newcommand{\conic}{\mathrm{c}}

\newcommand{\feas}{\mathrm{f}}
\newcommand{\opt}{\mathrm{o}}

\newcommand{\gap}{\overline{\mu}}

\newcommand{\supp}{\operatorname{supp}}

\newcommand{\MSysName}{\textsc{MSys}}
\newcommand{\PMaxName}{\textsc{PrimalMaxS}}
\newcommand{\DMaxName}{\textsc{DualMaxS}}
\newcommand{\PMaxAName}{\textsc{PrimalMaxSA}}
\newcommand{\DMaxAName}{\textsc{DualMaxSA}}
\newcommand{\MSys}{\hyperref[eq:big-M-sys]{\MSysName}}
\newcommand{\PMax}{\hyperref[sub:PMax]{\PMaxName}}
\newcommand{\DMax}{\hyperref[sub:DMax]{\DMaxName}}
\newcommand{\PMaxA}{\hyperref[sub:PMaxA]{\PMaxAName}}
\newcommand{\DMaxA}{\hyperref[sub:DMaxA]{\DMaxAName}}
\newcommand{\LP}{\textsc{LP}}

\newcommand{\SLC}{\operatorname{SLC}}

\newcommand{\pf}{\mathrm{pf}}   %
\newcommand{\df}{\mathrm{df}}   %
\newcommand{\mPi}{\boldsymbol{\Pi}}  %

\newcommand{\Pc}{P^{\conic}}    %
\newcommand{\Dc}{D^{\conic}}    %
\newcommand{\Pf}{P^{\pf}}       %
\newcommand{\Df}{D^{\df}}       %

\newcommand\restr[2]{{%
  \left.\kern-\nulldelimiterspace %
  #1 %
  \vphantom{\big|} %
  \right|_{#2} %
  }}

\makeatletter
\let\original@algocf@latexcaption\algocf@latexcaption
\long\def\algocf@latexcaption#1[#2]{%
  \@ifundefined{NR@gettitle}{%
    \def\@currentlabelname{#2}%
  }{%
    \NR@gettitle{#2}%
  }%
  \original@algocf@latexcaption{#1}[{#2}]%
}
\makeatother

\makeatletter
\newcommand*{\rom}[1]{\expandafter\@slowromancap\romannumeral #1@}
\makeatother

\newlength{\leftstackrelawd}
\newlength{\leftstackrelbwd}
\def\leftstackrel#1#2{\settowidth{\leftstackrelawd}%
{${{}^{#1}}$}\settowidth{\leftstackrelbwd}{$#2$}%
\addtolength{\leftstackrelawd}{-\leftstackrelbwd}%
\leavevmode\ifthenelse{\lengthtest{\leftstackrelawd>0pt}}%
{\kern-.5\leftstackrelawd}{}\mathrel{\mathop{#2}\limits^{#1}}}

\newtheorem{theorem}{Theorem}[section]
\newtheorem{lemma}[theorem]{Lemma}
\newtheorem{fact}[theorem]{Fact}
\newtheorem{claim}[theorem]{Claim}
\newtheorem{proposition}[theorem]{Proposition}
\newtheorem{corollary}[theorem]{Corollary}

\theoremstyle{definition}
\newtheorem{remark}[theorem]{Remark}
\newtheorem{definition}[theorem]{Definition}

\newtheorem{subroutine}[theorem]{Subroutine}

\newenvironment{claimproof}[1][\proofname]
{
    \proof[#1]%
        
}
{
    \endproof
}

\makeatletter
\providecommand*{\cupdot}{%
  \mathbin{%
    \mathpalette\@cupdot{}%
  }%
}
\newcommand*{\@cupdot}[2]{%
  \ooalign{%
    $\m@th#1\cup$\cr
    \hidewidth$\m@th#1\cdot$\hidewidth
  }%
}
\makeatother

\foreach \x in {A,...,Z}{%
	\expandafter\xdef\csname m\x\endcsname{\noexpand\boldsymbol{\x}}
}

\DeclareMathOperator*{\argmax}{arg\,max}

\newlist{mycases}{enumerate}{1}
\setlist[mycases,1]{wide=0pt, widest=99,leftmargin=\parindent,label=\textbf{Case \Roman*.}, ref=Case \Roman*}

\newcommand\myshade{100}
\colorlet{mylinkcolor}{NavyBlue}
\colorlet{mycitecolor}{YellowOrange}
\colorlet{myurlcolor}{Aquamarine}

\hypersetup{
  linkcolor  = mylinkcolor!\myshade!black,
  citecolor  = mycitecolor!\myshade!black,
  urlcolor   = myurlcolor!\myshade!black,
  colorlinks = true,
}

\author[1]{Daniel Dadush}
\author[2]{Zhuan Khye Koh}
\author[3]{Bento Natura}
\author[4]{Neil Olver}
\author[5]{L{\'{a}}szl{\'{o}} A. V{\'{e}}gh}
\affil[1]{Centrum Wiskunde \& Informatica, Amsterdam, The Netherlands}
\affil[2]{University of British Columbia, Vancouver, Canada}
\affil[3]{Columbia University, New York, USA}
\affil[4]{London School of Economics and Political Science, UK}
\affil[5]{University of Bonn, Germany}

\title{Solving Classes of Linear Programs with Strongly Polynomial Straight Line Complexity in Strongly Polynomial Time}

\begin{document}

\maketitle

\begin{abstract}
 Recently, Allamigeon, Dadush, Loho, Natura, and V\'egh \cite{ADLNV2022} introduced an interior point method for linear programs whose number of iterations is bounded by the so-called straight line complexity (SLC) of the instance.
This yields a strongly polynomial bound on the number of iterations of the corresponding interior point algorithm for broad classes{} of linear programs. In particular, in the companion paper \cite{DKNOV_mincost} we give a strongly polynomial straight line complexity bound for two-variable-per-inequality linear programs.\footnote{This result was announced in the STOC 2024 proceedings paper ``A strongly polynomial algorithm for linear programs with at most two nonzero entries per row or column'' \cite{DKNOV_STOC}. We present the full argument in two journal papers, the manuscript in preparation \cite{DKNOV_mincost}
containing the straight line complexity results, and the current manuscript including the initialization framework and the bit complexity arguments.}
However, in order to obtain a strongly polynomial algorithm for these classes, one also needs to address two important technical issues: initialization and bit-complexity bounds. Standard IPM initialization techniques either require numerical estimates that cannot be obtained in the strongly polynomial model,
or involve modifying the instance in ways that do not preserve the geometry of the original LP, potentially affecting the SLC.

In this work, we present an initialization framework for interior point methods for linear programming, along with a rounding procedure that ensures the continued feasibility of the IPM. We develop a multistage scheme in which each stage can be initialized in strongly polynomial time given the output of the previous stage. Additionally, we address the challenge of controlling the bit complexity of each iterate throughout the algorithm. To this end, we show that it is possible to represent iterates as low-complexity convex combinations of vertices and extreme rays. Our combinatorial rounding step is independent of the algorithm in \cite{ADLNV2022}, and can be applied more broadly  in the context of path-following methods.
\end{abstract}
\thispagestyle{empty}
\newpage

\tableofcontents
\thispagestyle{empty}

\clearpage
\pagenumbering{arabic}

\section{Introduction}\label{sec:introduction}
We consider linear programming (LP) in the following primal-dual form:
\begin{equation}
\label{LP_primal_dual} \tag{$\LP(\mA,b,c)$}
\begin{aligned}
\min \; & \pr{c}{x} \\
\mA x & = b  \\
x & \geq \0_n \, ,
\end{aligned}
\qquad \qquad
\begin{aligned}
\max \; & \pr{b}{y}\\
\mA^\top y + s & = c \\
s & \geq \0_n \, ,
\end{aligned}
\end{equation}
where $\mA \in \R^{m \times n}$, $b \in \R^m$, $c \in \R^n$, and $\mA$ has rank $m$.
  Our focus is on LP algorithms that find exact primal and dual optimal solutions, or conclude infeasibility or unboundedness.

The first polynomial-time LP algorithms were  the ellipsoid method by Khachiyan in 1979 \cite{Khachiyan79} and interior point methods,  introduced by Karmarkar in 1984 \cite{Karmarkar84}. However, it remains an outstanding open question to find a \emph{strongly polynomial algorithm} for linear programming. The question was listed by the Fields medalist Smale as one of the most prominent mathematical challenges for the 21st century \cite{smale1998}. In such an algorithm, only $\poly(n)$ basic arithmetic operations  and comparisons  are allowed, and the algorithm uses polynomial space (see Section~\ref{sec:comp-mod} for formal definitions).

The notion of strongly polynomial algorithms was first formally introduced by Megiddo \cite{Megiddo83}, under the term `genuinely polynomial'; he gave such an algorithm for finding feasible solutions to  \emph{two-variable-per-inequality (2VPI)} linear programs.
In a seminal, Fulkerson-prize winning paper \cite{Tardos1985}, Tardos obtained the first strongly polynomial algorithm for minimum-cost circulations.  A particularly important technique in this paper was \emph{variable fixing}: by solving an approximate version of the LP with rounded costs, one can deduce that a certain variable is at the lower or upper capacity bound in an optimal solution.
Towards general LP, Tardos \cite{Tardos86} extended this approach  to obtain a strongly polynomial algorithm for `combinatorial LPs'. More precisely, for \ref{LP_primal_dual} with an integer constraint matrix $\mA\in \mathbb{Z}^{m\times n}$, this algorithm runs in $\poly(n,\log\Delta_{\mA})$ iterations, where $\Delta_{\mA}$ is the maximum absolute value of a subdeterminant of $\mA$. The running time is independent of $b$ and $c$. In particular, this bound is strongly polynomial if all entries of $\mA$ are at most $\poly(n)$, such as for multicommodity flows and other combinatorial problems.

\paragraph{Strongly polynomial algorithms from interior point methods.}
Interior point methods (IPMs) give the fastest current weakly polynomial algorithms for general LPs, see \cite{CLS19,vdb20,br2020solving,Jiang2021}, and for special classes such as minimum-cost circulations \cite{chen2022maximum,brand2023deterministic} and multicommodity flows \cite{brand2023faster}. Somewhat surprisingly, they are also a potent approach in the context of strongly polynomial computability. In recent years, IPMs have emerged as a unifying tool in this context.

The algorithms discussed next  fall into the class of primal-dual path-following algorithms. A key concept here is the \emph{central path}, the algebraic curve formed by  minimizers of $\pr{c}{x}-\mu \sum_{i=1}^n \log(x_i)$ over the domain $\{x : \mA x = b\}$ for $\mu>0$ (see Section~\ref{sec:CP} for more details).
As $\mu\to 0$, the limit of the central path is an optimal solution.
 Path-following methods maintain iterates in a certain neighborhood of the central path while geometrically decreasing $\mu$ and thus the optimality gap. The logarithmic barrier function above can be replaced by more general barrier functions. %

 The main progress measure in IPMs is the decrease in optimality gap and thus these methods are inherently weakly polynomial; finding an exact optimal solution typically requires an additional rounding step. In the influential paper \cite{Vavasis1996}, Vavasis and Ye introduced
 \emph{Layered Least Squares (LLS) IPMs} that terminate with an exact optimal solution in
 $\poly(n,\log\bar\chi_{\mA})$ arithmetic operations, where $\bar\chi_{\mA}$ is the Dikin--Stewart--Todd condition number of the  matrix $\mA$. Remarkably, the number of arithmetic operations is independent of the vectors $b$ and $c$.
  For integer matrices, $\bar\chi_{\mA}=O(n\Delta_{\mA})$, so this strengthens Tardos's result. We note that a similar dependence, obtained via a black-box variable-fixing approach extending Tardos's work, was established in~\cite{DadushNV20}. %

This led to a line of research on improved combinatorial IPMs \cite{MMT98,Monteiro2003,MonteiroT05,LMT09}. The paper by Dadush, Huiberts, Natura, and V\'egh \cite{DadushHNV20} revealed that $\bar\chi_{\mA}$ is closely related to the \emph{circuit imbalance measure} $\kappa_{\mA}$ that bounds the maximum absolute value of ratios of two nonzero entries of an elementary vector in the kernel of $\mA$.
Moreover, they obtained an LLS algorithm invariant under column rescaling, thus improving the dependence to the best $\kappa^*_{\mA}$ value achievable under column rescalings.

The above results may raise hopes of finding a strongly polynomial IPM. However, the  papers by Allamigeon, Benchimol, Gaubert, and Joswig~\cite{allamigeon2018}, see also \cite{AGV22,Zong2023},
yield a surprising negative answer. By analyzing the tropical limits of linear programs, these papers exhibit parametric families of LPs such that for suitably large parameter values, no path-following method can be strongly polynomial. This was first shown for the standard logarithmic barrier \cite{allamigeon2018}, and later for arbitrary self-concordant barriers \cite{AGV22}.

\paragraph{IPMs matching straight line complexity.}
The lower bounds in \cite{allamigeon2018,AGV22} are ultimately based on the following insight. The trajectory of any path-following IPM is a piecewise linear curve in the neighborhood of the central path; the number of pieces corresponds to the number of iterations. Thus, a lower bound on the number of linear segments of \emph{any} piecewise linear curve in the neighborhood provides a lower bound on the number of iterations. The above-mentioned papers exhibit families of LPs with exponential lower bounds.

The recent result by Allamigeon, Dadush, Loho, Natura, and V\'egh~\cite{ADLNV2022} complements these negative results by a positive algorithmic bound. Namely, their \emph{Subspace Layered Least Squares (SLLS) IPM} matches the straight line complexity lower bound within a strongly polynomial $\tilde O(n^{1.5})$ factor; their  result is stated formally in \Cref{thm:SLC-bound}.
At a high level, the algorithm exploits the structural property that whenever there is a long straight line segment in the neighborhood of the central path, the variables exhibit a `polarized'
behavior: There exists a partition $[n]=(B,N)$ such that primal variables barely change on $B$ and scale roughly with the gap on $N$, while the reverse holds for the dual variables. The SLLS IPM can detect such a partition and uses a step that exhibits this behavior to reach a point near the end of such a segment in $\tilde O(n^{1.5})$ steps. Very recently, \cite{Dadush2026} improved the iteration complexity of the above algorithm, by using a strongly polynomial implementation of the trust region steps introduced
in \cite{LMT09}. In particular, they showed that the iteration complexity of the Trust Region IPM is optimal within a constant factor in the $\ell_2$-neighborhood.

The SLLS IPM in \cite{ADLNV2022} provides a unifying framework for strongly polynomially solvable classes of LPs. One can show that the straight line complexity is bounded polynomially by $n$ and $\log(\kappa_{\mA}^*)$, the optimized circuit imbalance measure, which was the most general previously known parametrization. Another example is the class of discounted Markov Decision Processes \cite{ye2005new}, and more generally, LPs where the nonzero entries of basic solutions fall into strongly polynomially bounded intervals \cite{chubanov2015polynomial}.
It is also the key tool for settling the strongly polynomial complexity of two-variable-per-inequality LPs, as discussed next.

\paragraph{Two-variable-per-inequality LPs.}
 We say that the dual program in \ref{LP_primal_dual} is a \emph{two-variable-per-inequality (2VPI)} linear program if every row of $\mA^\top$ includes at most two nonzero entries. In such a case, we refer to the pair of LPs as a 2VPI primal-dual pair.

 As already mentioned, one of the first strongly polynomial algorithms in the literature was given by Megiddo \cite{Megiddo83} for the 2VPI feasibility problem. As shown in \cite{Hochbaum2004}, 2VPI primal-dual pairs can be strongly polynomially reduced to solving the \emph{minimum-cost generalized flow} problem, the classical network flow model where the flow may change while traversing the arcs according to gain factors. The strongly polynomial algorithms for generalized flow maximization \cite{Vegh17,OV17} thus imply a strongly polynomial algorithm for finding primal-feasible solutions of LPs with at most two nonzeros per column. However, for the minimum-cost generalized flow problem, and for 2VPI primal-dual pairs, this was only settled recently in the conference paper by the authors \cite{DKNOV_STOC}:

 \begin{theorem}\label{thm:main-main}
    There is a deterministic strongly polynomial algorithm for the minimum-cost generalized flow problem, and for two-variable-per-inequality primal-dual pairs.
\end{theorem}
A key ingredient in showing \Cref{thm:main-main} is proving strongly polynomial straight-line complexity bounds for minimum-cost generalized flow problems; these bounds are described in the companion paper \cite{DKNOV_mincost}. However, to obtain the claimed strongly polynomial bound, there are  significant additional
challenges regarding  initialization and implementation in the Turing model. The current paper exhibits general techniques to address these challenges.

\paragraph{Summary of contributions.}
Our contributions, presented in \Cref{sec:our_contributions}, can be informally summarized as follows:
\begin{enumerate}[label=(\arabic*), ref=(\arabic*)]
\item\label{task:init} \emph{Initialization:} IPMs require a pair of strictly feasible initial primal and dual solutions near the central path, even though a strictly feasible (or even a feasible) solution may not exist. We present a careful initialization scheme that solves linear programs in three stages, and preserves the straight line complexity bounds; see \Cref{thm:main-1}.
\item\label{task:turing}  \emph{Implementation in the Turing model:}
We show that the bit-length of the computations can be controlled in a model using only basic arithmetic operations and comparisons; see \Cref{lem:main-round}.
\end{enumerate}
Initialization is always a significant issue when applying IPMs. It is usually done by embedding the LP into an auxiliary system that is strictly feasible for which a suitable starting solution can be easily found and the optimal solutions of the auxiliary LP map to the optimal solutions of the original LP. Such constructions, however, may lose the combinatorial structure of the LP, or rely on numerical estimates based on bit-complexity (see \Cref{sec:intro-initialization}). In contrast, our initialization scheme does not require any a priori numerical estimates. The overall running time in \Cref{thm:main-1} is bounded by the maximum straight line complexity of the original LP and its three simple modifications. For the min-cost generalized flow problem, each of these LPs also describes a generalized flow instance.

Our second contribution addresses a fundamental issue regarding the computational models; we formally introduce the real RAM and Turing models in \Cref{sec:comp-mod} below. While `combinatorial' algorithms such as Tardos's results  \cite{Tardos1985,Tardos86} or the 2VPI/generalized flow papers \cite{Megiddo83,Vegh17,OV17} are strongly polynomial in the Turing model of computation, the Vavasis--Ye algorithm \cite{Vavasis1996} and all subsequent work on combinatorial IPMs, including the SLLS IPM \cite{ADLNV2022}, are presented in the real RAM model. Whenever the straight line complexity is strongly polynomially bounded, the number of arithmetic operations is strongly polynomial.
 However, in the Turing model one needs to ensure that the algorithm only uses polynomial space. The above algorithms are `stepwise strongly polynomial', i.e., each iteration can be implemented in strongly polynomial time. Still, the bit-length could increase exponentially over the sequence of iterations. This notion is defined in \Cref{sec:comp-mod}, where we also discuss why simple rounding techniques are not applicable in the strongly polynomial model.

 In \Cref{lem:main-round}, we introduce a `combinatorial' strongly polynomial rounding step. This is applicable broadly for the combinatorial IPMs above, and can thus turn them, under certain conditions explained below, into strongly polynomial algorithms in the Turing model for the classes where their iteration complexities are strongly polynomial.

\subsection{Computational models}\label{sec:comp-mod}
We now turn to the discussion of the computational models.
There are multiple related, yet distinct notions of a \emph{strongly polynomial} computational model. Smale's question was posed in the \emph{Blum--Shub--Smale (BSS) real model of computation} \cite{blum1989theory}. In this model, the input can be given by arbitrary real numbers, and one step may compute a rational  function
of the previously computed quantities with real coefficients, or make comparisons between two quantities. In the  more restrictive
\emph{real RAM} model, one can perform a sequence of elementary arithmetic operations ($+,-,\times,/$) and comparisons ($\ge$) on real numbers. We say that an algorithm is \emph{polynomial in the real RAM model} if the number of elementary arithmetic operations and comparisons  is bounded polynomially in the dimension of the input; in the case of LP, this is $K=n\times m+n+m$.

We now turn to the Turing model. Consider a problem where the input is given by $K$ integers; for LP, the input $(\mA,b,c)$ is described by $K=2(n\times m+n+m)$ integers representing the rational entries. An algorithm
is  \emph{strongly polynomial in the Turing model} (see \cite{gls}), if it only performs $\mathrm{poly}(K)$  elementary arithmetic operations and comparisons as in the real RAM model; in the LP case, this means  $\mathrm{poly}({m}, n)$ arithmetic operations. Additionally,  the bit-complexity of the numbers during the computations must remain polynomially bounded in the encoding length of the input. Equivalently, the algorithm must be polynomial in the real RAM model and also be in PSPACE.

There is an ambiguity in the models regarding how divisions can be implemented; see discussion of variants in \cite[Section 1.3]{gls}. The results of this paper work with the most restricted setting: all numbers are maintained as numerator--denominator pairs, which are never reduced to lowest terms. To make this precise, we distinguish a rational number from the pair that represents it, and define sizes of numbers following~\cite{gls}.
\begin{definition}[Encoded rationals and sizes of numbers]\label{def:size}
For an integer $k\in\Z$, we define the \emph{encoding length of $k$} as $\size(k)=1+\lceil \log_2(|k|+1)\rceil$.
We represent rational numbers as ordered pairs of integers, and accordingly set
\[
  \eQ\coloneqq\Z\times\Zplusplus\, ,\qquad
  \val(p,q)\coloneqq\frac{p}{q}\, ,\qquad
  \size(p,q)\coloneqq\size(p)+\size(q)\, ,
\]
calling the elements of $\eQ$ \emph{encoded rationals}. The canonical \emph{value map} $\val\colon\eQ\to\Q$ is surjective, and $\Q$ can be seen as the quotient of $\eQ$ by the equivalence $(p,q)\sim(p',q')$ if and only if $pq'=p'q$. We regard an integer $k$ as the encoded rational $(k,1)$, note that its two sizes then differ by the additive constant $2$, and we set $\eQplusplus\coloneqq \{(p,q)\in\eQ\mid p >0\}$.
Further, for a rational vector or matrix, the size is defined as the sum of the sizes of the entries. For matrices and vectors $a_1,a_2,\ldots,a_k$ of possibly different dimensions, we let $\size(a_1,a_2,\ldots,a_k)=\sum_{i=1}^k\size(a_i)$.
Finally, for a \emph{value} $\alpha\in\Q$ we set
\[
\size(\alpha)\coloneqq\min\{\size(a)\mid a\in\eQ,\ \val(a)=\alpha\}\, ,
\]
the minimum being attained by the lowest-terms representation, and likewise entrywise for rational vectors and matrices, matching the classical notion of encoding length of a rational number.
\end{definition}

Every number produced by our algorithms is an encoded rational. For $(p,q),(p',q') \in \eQ$ an addition returns $(pq'+p'q,qq')$, a multiplication returns $(pp',qq')$, and a division $\frac{p}{q} / \frac{p'}{q'}$ returns $(pq',\,qp')$, with the sign normalized so that the second coordinate stays positive.

We identify an encoded rational with its value wherever only the value matters: for $a\in\eQ^n$ and $S\subseteq\Q^n$ we write $a\in S$ for $\val(a)\in S$, and equations, inequalities and norms of encoded objects are those of their values.

While a strongly polynomial algorithm in the Turing model implies a  polynomial algorithm in the real {RAM} model,
the converse is not necessarily the case: enforcing PSPACE may be challenging. {For example,  Gaussian elimination needs to be done carefully to keep the sizes of numbers under control, see \cite{edmonds1967systems} and \cite[Section 1.4]{gls}.}

The LLS interior point methods~\cite{Vavasis1996,MMT98,Monteiro2003,MonteiroT05,DadushHNV20,ADLNV2022} are polynomial in the real RAM model\footnote{Some of these IPMs make use of square-root computations, and hence rely on the extended real model $(+,-,\times,/,\sqrt{.})$. However, square roots can be plausibly avoided in these works. E.g.,
one can approximate $\sqrt{n}$ to high accuracy in strongly polynomial time. No square-roots are used in \cite{ADLNV2022} and in \cite{Dadush2026}.} whenever $\log(\bar\chi_{\mA})=\mathrm{poly}(n)$.
However, we are not aware of any strongly polynomial implementation of such algorithms in the Turing model. The next definition highlights the key difficulty in developing such implementations.

\paragraph{Stepwise strongly polynomial algorithms.}
By an \emph{iterative algorithm} we mean an algorithm that computes a sequence of points $z^{(0)},z^{(1)},\ldots,z^{(T)}\in\R^k$ for some $k\in\mathbb{N}$; interior point methods all have this form. By a \emph{stepwise strongly polynomial algorithm} we mean an iterative algorithm that, given a rational input $z^{(t)}\in\Q^k$, computes the next iterate $z^{(t+1)}\in\Q^k$ using a strongly polynomial algorithm in the Turing model.

Note that stepwise strong polynomiality does not require that the total number of iterations $T$ is strongly polynomially bounded. Most standard interior point methods, such as the predictor-corrector method \cite{MTY}, are stepwise strongly polynomial: computing each step direction amounts to solving a system of linear equations, where the coefficients are constructed from the input $(\mA,b,c)$ and the current iterate. Accordingly, the bit complexity of the iterates might grow by a polynomial factor in each iteration. Thus, after a superconstant number of iterations, we would not be able to guarantee that the bit complexity of the iterates remains polynomially bounded in the input length.

Consequently, even if an iterative algorithm terminates in a strongly polynomial number of iterations, and is stepwise strongly polynomial, it does \emph{not} follow that the algorithm is strongly polynomial in the Turing model.

\paragraph{The difficulty of rounding in the strongly polynomial model.}
To implement an IPM in PSPACE, a natural idea is to use additional rounding steps.
In the weakly polynomial model, i.e.,~when running time dependence on the total encoding length is allowed, the  bit complexity of the algorithms can be controlled by approximate linear system solves and rounding. Recent work by Ghadiri, Peng, and Vempala \cite{ghadiri2023bit} developed general tools that enable one to keep the bit complexity of recent fast IPMs under control.

However, these techniques are not applicable in the strongly polynomial Turing model, where the number of arithmetic operations and comparisons must be bounded independently of the encoding length. In this model one cannot even approximate $\log(x)$ to within an additive constant for large $x$.

As already mentioned, \Cref{lem:main-round}  in this paper presents a general, strongly polynomial rounding technique that can be used to convert---under certain technical conditions---any stepwise strongly polynomial interior point method that stays in the $\ell_2$-neighborhood to a PSPACE algorithm.

\subsection{The central path and the analytic center}\label{sec:CP}
We next formally introduce some concepts from the theory of interior point methods needed to define straight-line complexity and to state our results.
Given an instance of \ref{LP_primal_dual}, we let
\[
\begin{aligned}
  \Primal\coloneqq\{x \in \R^n \mid \mA x=b\, ,\, x \geq \0\}\, , \quad&\mbox{and}&\quad \, \Primalp\coloneqq\Primal\cap \Rplusplus^n\, ,\\
  \Dual\coloneqq\{s \in \R^n \mid \exists y\in\R^m\, , \, \mA^\top y+s=c\, ,\, s \geq \0\}
  \, , \quad&\mbox{and}&\quad \Dualp\coloneqq\Dual\cap \Rplusplus^n\, .
  \end{aligned}
\]
denote the primal and dual feasible regions, and the sets of strictly feasible primal and dual solutions, respectively.
For  $(x,s)\in \Primal\times \Dual$, the \emph{normalized duality gap} is
\[
\gap(x,s)\coloneqq\frac{\pr{x}{s}}{n}\, .
\]
By complementary slackness, $(x,s)\in \Primal\times \Dual$ is a pair of primal and dual optimal solutions if and only if $\gap(x,s)=0$.

Given an instance of \ref{LP_primal_dual}, if $\Primalp,\Dualp\neq \emptyset$, then for every $\mu>0$ there exists a unique point $(x(\mu),s(\mu))\in \Primal\times\Dual$ with $x_i(\mu)s_i(\mu)=\mu$ for every $i\in [n]$; this is called the \emph{central path point at $\mu$} and is denoted by $(x\cp(\mu),s\cp(\mu))$. The \emph{central path} is the algebraic curve formed by the points $(x(\mu),s(\mu))$ for $\mu\in (0,\infty)$. The limit of the central path is a primal-dual optimal solution pair $(x^\star,s^\star)\in \Primal\times\Dual$, such that $B^\star=\supp(x^\star)$ and $N^\star=\supp(s^\star)$ form a partition of $[n]$.

Interior point methods maintain iterates in some neighborhood of the central path. We define the
 $\ell_2$-\emph{neighborhood} of the central path for a value $\beta \in [0,1)$ as
\begin{equation} \label{eq:l2_neighborhood}
\cN(\beta) \coloneqq \left\{(x,s) \in \Primal_{>0} \times \Dual_{>0} : \norm{\frac{x \circ s}{\gap(x,s)} - \1_n}_2 \leq \beta \right\}\, ,
\end{equation}
where $x\circ s=(x_1s_1,x_2s_2,\ldots,x_ns_n)$ denotes the Hadamard product of the vectors $x$ and $s$.
We let $\clN(\beta)$ denote the topological closure of the $\ell_2$ neighborhood. This
corresponds to the part of the neighborhood containing the limit optimal
solutions on the central path.
For simplicity, throughout the paper, $\beta$ is a parameter fixed by the algorithm, with $1/\poly(n)\le\beta\le 1/16$ and $\size(\beta)=O(\log n)$ in the sense of \Cref{def:size}. In particular, $\log(1/\beta)=O(\log n)$, and $\size(\beta)$ is absorbed in all $\poly(n)\size(\mA,b,c)$ bounds below. The lower bound on $\beta$ comes from the interior point method, whose per-iteration progress depends on $\beta$.

If $\Primalp$ is nonempty and bounded, its
\emph{analytic center}\footnote{We are slightly abusing notation here, in that the analytic center depends on the system $\mA x=b, x \geq 0$ and not just the feasible set.} is
\[
\tilde x\coloneqq \argmax_{x\in \Primalp}\sum_{i=1}^n \log x_i\, .
\]
By Lagrangian duality, there exists a unique $\tilde v\in \im(\mA^\top)$ such that $\tilde x\circ \tilde v=\1_n$. We define the analytic center of $\Dualp$ analogously as the maximizer of $\sum_{i=1}^n \log s_i$ over $\Dualp$. This is equivalent to the existence of $\tilde w\in \ker(\mA)$ with $\tilde s\circ\tilde w=\1_n$. We note that if both $\Primalp$ and $\Dualp$ are nonempty, one of them   is always unbounded; hence, the analytic center may only exist for at most one of them.

It is well known that the central path limit point $(x^\star,s^\star)$ coincides with the analytic centers of the primal and dual optimal facets, see, e.g., \cite[Theorem I.30]{roos2005interior}.
Namely, $x^\star$ is the analytic center of $\{x_{B^\star}\mid x\in \Primal\, , \supp(x)\subseteq B^\star\}$, and $s^\star$ is the analytic center of $\{s_{N^\star}\mid s\in \Dual\, , \supp(s)\subseteq N^\star\}$.

\subsection{Straight line complexity and the Subspace Layered Least Squares Interior Point Method}\label{sec:LLS-intro}
In this section, we introduce the concept of straight-line complexity, and state the main result on the SLLS IPM from \cite{ADLNV2022}.
Assume \ref{LP_primal_dual} is feasible and bounded with optimum value $v^\star$.
Given $g\ge 0$, we denote by
\begin{equation}\label{eq:sublevel-sets}
\begin{aligned}
\Primal_g & \coloneqq \{ x \in \Primal\mid \pr{c}{x} \leq v^\star + g \} \, , \\
\Dual_g & \coloneqq \{ s \in \Dual \mid \exists y\in\R^m \text{ with } \mA^\top y+s=c \text{ and } \pr{b}{y} \ge v^\star - g \}
\end{aligned}
\end{equation}
the feasible sublevel sets.
They correspond to the sets of the primal and dual feasible points $(x, s)$ with objective value within $g$ of the optimum $v^\star$, respectively.
\begin{definition}[Max central path]
{Assume $\Primal\neq\emptyset$ and $\Dual\neq\emptyset$.}
The \emph{max central path} is defined as the parametric curve $g\mapsto \mz(g) \coloneqq (\mx(g), \ms(g))\in (\Rplus\cup\{\infty\})^{2n}$, where for $g\ge 0$,
\begin{equation}\label{eq:max-cp-def}\mx_i(g)\coloneqq \sup\{x_i\mid x\in \Primal_g\} \quad \mbox{and}\quad \ms_i(g)\coloneqq \sup\{s_i\mid s\in \Dual_g\}\, ,\qquad \forall i\in [n]\, ,\end{equation}
\end{definition}
The max central path can be seen as a combinatorial proxy to the central path. In particular, for $g=n\mu$, $x\cp(\mu)\in \Primal_{g}$ and $s\cp(\mu)\in \Dual_{g}$. %
One can show the following bounds. Note that the existence of the central path requires strict feasibility of the primal and dual programs, i.e., $\Primalp,\Dualp\neq\emptyset$, whereas the max central path only requires feasibility on both sides. In case of strict feasibility, it also follows that both $\mx(g)$ and $\ms(g)$ are finite valued.
\begin{lemma}[{\cite[Lemma 1.2]{ADLNV2022}}]
If both primal and dual programs in  \ref{LP_primal_dual} are strictly feasible, then
for every $\mu>0$ and $g = n\mu$, we have
\[
 \frac{\mx(g)}{2n}\le x\cp(\mu)\le \mx(g)\, \quad\mbox{and}\quad \frac{\ms(g)}{2n}\le s\cp(\mu)\le \ms(g)\, .
 \]
\end{lemma}
For each $i\in [n]$, the function $g\mapsto \mx_i(g)$ is a piecewise-linear concave function. It
corresponds to a trajectory of the shadow simplex algorithm that interpolates between the objective functions $\pr{c}{x}$ and $-x_i$. The breakpoints correspond to basic feasible solutions, and therefore the number of linear pieces is at most the number of vertices, that is, at most $2^n$; see \cite[Section 4]{ADLNV2022}.
We will use the following definition.
\begin{definition}\label{def:slc-alternative}
Let $f\colon \Rplus\to\Rplus$ be a function and $\slceta\in [0,1]$.
The \emph{straight line complexity} of $f$ with respect to $\slceta$, denoted $\SLC_{\slceta}(f)$, is the infimum of the numbers of pieces among  continuous piecewise-linear functions $h\colon \Rplus\to \Rplus$ such that $\slceta f \leq h \leq f$.\footnote{We remark that our convention on the parameter $\slceta$ differs from \cite{ADLNV2022}, which uses $(1-\theta)f \leq h \leq f$ instead.} For the function $f(x)=\infty$ for all $x\in\Rplus$, we define $\SLC_{\slceta}(f)=1$.
\end{definition}

We will refer to $\eta$ as the \emph{SLC} parameter.
The number of iterations taken by the \emph{Subspace Layered Least Squares (SLLS) IPM} in \cite{ADLNV2022} can be bounded by the sum of straight line complexities of each coordinate of the (primal) max central path.
We now state the guarantees of SLLS IPM in~\cite{ADLNV2022}.

\begin{theorem}[{\cite[Theorem 1.4]{ADLNV2022}}]\label{thm:SLC-bound}
There is an interior-point method that, given an instance of \ref{LP_primal_dual} and a starting solution $(x, s) \in\cN(\beta)$, $0<\beta\leq1/6$, finds a pair of primal and dual optimal solutions $(x\out, s\out) \in\Primal\times \Dual$ in
\begin{equation}
  O\left(\frac{\sqrt{n}}{\beta}\min_{\slceta\in(0,1]}\log\left(\frac{n}{\beta\slceta}\right) \sum_{i = 1}^n \SLC_{\slceta}(\mx_i)\right)
\end{equation}
many iterations.
Let $B=\supp(x\out)$ and $N=\supp(s\out)$. Then, $B\cup N=[n]$, and
the algorithm additionally outputs $(v\out,w\out) \in \R^{2n}$
satisfying
\begin{enumerate}[label=(\roman*)]
\item $v\out \in \im(\mA^\top), v\out_B > \0_B$ and $\mA w\out = \0_m$, $w\out_N > \0_N$, and
\item $\norm{(x\out_B \circ v\out_B,s\out_N\circ w\out_N)-\1_n}_2 \leq \beta$.
\end{enumerate}
The algorithm can be implemented in the real RAM model, and is stepwise strongly polynomial. Further, $(v\out,w\out)$ is obtained by a strongly polynomial computation from a step direction computed in the algorithm. Every iterate stays in $\cN(\beta)$, and the normalized duality gap contracts by a factor of at most $1-\beta/(2\sqrt{n})$ in every iteration.
\end{theorem}
The vectors $(v\out,w\out)$ certify that
the returned optimal solutions $(x\out, s\out)$ are close to the analytic centers of the
corresponding optimal faces.
The SLLS IPM requires at most a $\tilde{O}(n^{1.5})$ factor more iterations than \emph{any} path-following IPM for any self-concordant barrier function. This is because it can be shown that each $\SLC_{\slceta}(\mx_i)$ gives a lower bound on the number of piecewise-linear segments needed to traverse a corresponding wide neighborhood.
Moreover,
as noted above, $\SLC_{1}(\mx_i)\le 2^n$, thus the number of iterations is always at most singly exponential.

We note that the theorem could be equivalently written in terms of the dual straight line complexities (see \cite[Lemma 4.5]{ADLNV2022}).
Further, we note that the SLC  parameter $\eta$ is not important for the overall bound. It is not difficult to show that for $0<\eta<\eta'<1$, $\SLC_{\eta'}(f)=O(\log(1/\eta)/\log(1/\eta'))\SLC_{\eta}(f)$. %

According to Theorem~\ref{thm:SLC-bound}, analyzing the number of iterations of SLLS IPM
boils down to upper-bounding the straight-line complexities of the variables. Note that this is a purely geometric question about understanding the structure of univariate piecewise linear functions $\mx_i(g)$.

\subsection{Our contributions}\label{sec:our_contributions}
Our first main result states that the running time of solving an arbitrary instance of \ref{LP_primal_dual} can be upper bounded by the maximum straight line complexity of four LPs. These include the original LP, as well as modified systems that will be used in the multi-stage initialization process.
\begin{restatable}{theorem}{mainthm}
\label{thm:main-1}
Let $\mA\in \R^{m\times n}$, $b\in\R^m$, $c\in\R^n$, $\rank(\mA)=m$. There exists an algorithm  that either finds primal and dual optimal solutions to \ref{LP_primal_dual}, or finds a certificate of primal or dual infeasibility. The running time of the algorithm is
\[
    O\left(\poly(n)\cdot \min_{\slceta\in(0,1]} \Gamma_\slceta\log\left(\frac{n}{\slceta}\right) \right)\, ,
\]
where $\Gamma_\slceta$ is the maximum over $\SLC_\slceta(x^\m_i)$ for any variable in any of the {primal and dual feasible LPs among the} following four:
\begin{align}
\min \pr{c}{x}\, ,&\quad \mA x=b\,,&x\ge \0_n\,\tag{S-Opt}  \label{sys:1}\\
\min \pr{\1_n}{\bar x}\, ,&\quad \mA x-\mA \bar x=\0_m\,,\quad x+\hat x= \1_n\, , &x,\bar x,\hat x\ge \0_n\, \label{sys:2} \tag{S-Conic} \\
\min \pr{c}{x}\, ,&\quad \mA x=\0_m\,,\quad x+\hat x=\1_n\, ,& x,\hat x\ge \0_n\, \tag{S-Dual} \label{sys:3} \\
\min \pr{\1_n}{\bar x}\, ,&\quad \mA x-\mA \bar x=b\, ,&x,\bar x\ge \0_n\, \tag{S-Primal} \label{sys:4}
\end{align}
{Note that \eqref{sys:2}, \eqref{sys:3}, and \eqref{sys:4} are always primal and dual feasible.}
For rational input $(\mA,b,c)$, the algorithm can be implemented in PSPACE.
\end{restatable}

We highlight as our second main result  a rounding technique that is used in the PSPACE implementation.  This step is not specific to our algorithm, and can be used to convert a broad class of IPMs in the real RAM model to the Turing model. Below,
$\size(\mA,b,c)$ denotes the binary encoding length of the input, formally defined in \Cref{def:size}.

\begin{restatable}{theorem}{roundlemma}\label{lem:main-round}
Assume $\mA\in \eQ^{m\times n}$, $b\in\eQ^m$, $c\in\eQ^n$, $\rank(\mA)=m$, and
let $\beta\in(0,1/16]$.
Given $(x,s)\in {\cal N^2}(\beta)\cap \eQ^{2n}$, there exists a strongly polynomial algorithm that finds a point $(\tilde x,\tilde s)\in \clN(3\beta)\cap \eQ^{2n}$ with $\gap(\tilde x,\tilde s)\le (1+\beta/(4n^2))\gap(x,s)$ and $\size(\tilde x,\tilde s)=\poly(n)\size(\mA,b,c)$.
Moreover, if $\gap(\tilde x,\tilde s)=0$, then the algorithm also returns $(\tilde v,\tilde w)\in\eQ^{2n}$ with $\size(\tilde v,\tilde w)=\poly(n)\size(\mA,b,c)$ such that $(\tilde x,\tilde s)$ and $(\tilde v,\tilde w)$ satisfy the output conditions of \Cref{thm:SLC-bound} with $3\beta$ in place of $\beta$.
\end{restatable}
Interestingly, the rounding process may replace the current iterate by a much better solution further down the central path. As shown in \Cref{prop:simplex}, using this rounding step already guarantees termination in at most an exponential number of iterations, thus already achieving the property highlighted in the title \emph{``Interior point methods are not worse than Simplex''} of the paper \cite{ADLNV2022}.

 \subsubsection{Initialization}
 \label{sec:intro-initialization}
 The SLLS IPM, and more generally path-following IPMs, require an initial point $(x^0,s^0)\in \Primalp\times \Dualp$ near the central path. Such a point may not even exist; in fact, the primal or dual programs in \ref{LP_primal_dual} may be infeasible.
The situation is analogous to Simplex, where Stage I is needed to find a feasible solution by solving an auxiliary LP.
 A range of initialization methods have been developed for IPMs. However, none  seems suitable for our goal
 of solving only auxiliary systems with small straight-line complexity while remaining in the strongly polynomial model.

 A common initialization technique is the  homogeneous self-dual formulation \cite{Ye1994}. However, writing the self-dual formulation of a generalized flow LP changes the combinatorial properties of the problem, and it is not clear if the straight line complexity admits a similar bound. Note also that already in the simpler case of (standard) network flows, the constraint matrix is totally unimodular, while the combined matrix does not have this property.

The following ``big-$M$ formulation'' was used in
 \cite{Vavasis1996}.  In the primal system, we consider the problem  with variables $(x,\bar x,\hat x)\in\R^{3n}$ in the form
 \begin{equation}\label{eq:big-M}
 \min \pr{c}{x} + M \pr{\1_n}{\bar x}\quad \mbox{s.t.}\quad \mA x-\mA \bar x=b\, ,\quad x+\hat x=2M\1_n\, ,\quad x,\bar x,\hat x\ge \0_n\, .\end{equation}
 Here, $\bar x$ represents a negative copy of each variable, penalized at a high cost $M$, and $\hat x$ corresponds to a slack variable for the box constraint $\0_n \le x\le 2M\1_n$. The dual is extended symmetrically, where the dual $\bar s$ of $\bar x$ represents slack variables, and the dual $\hat s$ of $\hat x$ are negative copies of $s$.

 Such a system, along with its dual, is easy to initialize for sufficiently large $M$. Moreover, the constraint matrix remains `nice', e.g., for a standard or generalized flow problem, it can still be interpreted as a capacitated  flow problem, where the $\bar x$ variables correspond to expensive reverse arcs.  In~\cite{Vavasis1996}, it is shown that the $\bar\chi$ condition number of the constraint matrix of the extended system is within a constant factor of the same condition number of the original system $\bar\chi_{\mA}$.

 As long as there exists a pair of primal and dual optimal solutions $(x^\star,s^\star)$ to \ref{LP_primal_dual} with $\|x^\star\|_\infty,\|s^\star\|_\infty<M$, these will also be optimal solutions to the extended formulation for sufficiently large $M$.
 However, finding a suitable large $M$ becomes challenging. In~\cite{Vavasis1996}, such a bound is derived based on $\bar\chi_{\mA}$. This is hard to compute in general; one could use a repeated (squared) guess of this condition number, but this would lead to a $\log\log\bar\chi_{\mA}$ running time dependence.
 Bounds on the norms of optimal solutions are routinely derived using bit-complexity arguments, see e.g.~\cite{gls}; however, this is also not possible in the strongly polynomial model.

\medskip

We present two initialization schemes. Both can
   be interpreted as multi-stage facial reduction strategies
in which carefully fixing or projecting out variables yields an equivalent LP with strictly feasible primal and dual solutions.

The first initialization scheme  in \Cref{sec:simplified-init} relies on a variant of the big-M system \eqref{eq:big-M}. In Lemma~\ref{lem:big-M}, we observe that in the case that primal and dual strictly feasible solutions to \ref{LP_primal_dual} are available,  we can compute a suitable $M$ such that the set of optimal solutions to \ref{LP_primal_dual} and \eqref{eq:big-M} coincide. Moreover, if only a primal strictly feasible solution is given (but no dual solution is available), then we can still choose $M$ such that  all primal optimal solutions to  \eqref{eq:big-M} are feasible to \ref{LP_primal_dual}, and analogously in the dual setting.

Our multi-stage initialization calls the IPM in \Cref{thm:SLC-bound} for different instantiations of \eqref{eq:big-M}, by gradually deriving bounds on the parameters. In Stage~\rom{1}, we address conic feasibility: we first use an auxiliary system to find maximum support solutions  to the conic systems $\mA x=\0_m$, $x\ge \0_n$ and $\mA^\top y+s=\0_n$, $s\ge \0_n$. The extended system, listed as \eqref{sys:2} in \Cref{thm:main-1}, is easy to initialize without any numerical bounds. Stage~\rom{2} uses these solutions to solve two systems where the goal is to find maximum support primal and dual solutions, respectively. The required parameters are computed according to \Cref{lem:big-M}, using the solutions from Stage~\rom{1}. Finally, Stage~\rom{3} uses the primal and dual maximum support solutions to construct a system for optimization. Here, we solve the system obtained by deleting variables that are 0 in every primal solution, and by projecting out variables that are 0 in every dual solution. We obtain numerical parameters from the strictly feasible primal and dual solutions.

\medskip

In Section~\ref{sec:analytic-init}, we develop a stronger, `big-$M$-free' initialization scheme to prove \Cref{thm:main-1}. Whereas in \Cref{sec:simplified-init} we only use the maximum support primal and dual optimal solutions from the output of \Cref{thm:SLC-bound}, here we also make use of the analytic center certificates $(v\out,w\out)$. We proceed in the same three stages, and Stage~\rom{1} with the conic system is the same as before. However, in  Stage~\rom{2} we can use simpler systems, and can initialize them directly using the analytic center certificates from Stage~\rom{1}. Similarly, in Stage~\rom{3}, we can directly initialize the reduced system obtained from  \ref{LP_primal_dual} after deleting and projecting out variables, using the analytic center certificates from Stage~\rom{2}.

The straight line complexity bounds stated in \Cref{thm:main-1} follow because Stage~\rom{1} uses the system \eqref{sys:2}, whereas Stage~\rom{2} uses subsystems of those used in \Cref{sec:simplified-init},
obtained by deleting and projecting out variables from the systems \eqref{sys:4} and \eqref{sys:3}. The Stage~\rom{3} system is obtained similarly from \eqref{sys:1}. In \Cref{lem:project-delete}, we show that deletion and projection of `irrelevant' variables do not increase the straight line complexity.

\subsubsection{Implementation in the Turing model}
\label{sec:impl-turing}
We now give an overview of the strongly polynomial rounding step in \Cref{lem:main-round}. Starting from an iterate $(x,s)\in \cN(\beta)$, our goal is to compute $(\tilde x,\tilde s)\in\clN(3\beta)$, possibly with a slightly worse gap, whose encoding length is polynomially bounded in the length of the input. The output may lie on the boundary of the neighborhood only if it is an exact primal--dual optimal pair; in this case, the current IPM call terminates with this solution instead of being restarted from it.

To argue about the encoding length, we `anchor' the point $(\tilde x,\tilde s)$ to vertices of the primal and dual polyhedra. In strongly polynomial time, we can write a Minkowski--Weyl decomposition of $x$ and $s$ using vertices and extreme rays. However, we cannot simply round the coefficients. In particular, it is possible that $(x,s)$ is written as a `highly unstable' convex combination of primal and dual vertices such that either $\pr{v}{u}<\pr{x}{s}/2^n$ or $\pr{v}{u}>2^n\pr{x}{s}$ for each pair of primal and dual vertices $(v,u)$. We proceed in two stages. First, in \Cref{lem:round-gap}, we try to find a value $\mu^\star\approx \pr{x}{s}/n$ such that $\mu^\star$ has small encoding length. This is easy as long as the combination contains primal and dual vertices $(v,u)$ with $\pr{x}{s}/2^n\le \pr{v}{u}\le 2^n\pr{x}{s}$. In the `highly unstable' situation as above,  it turns out that the direction from $(x,s)$ pointing towards a pair of primal and dual vertices $(v,u)$ with much better gap is a very good movement direction of the IPM. Hence, we can replace $( x, s)$ during the rounding step by a much better iterate that is also numerically more stable. In the second stage, we add a cost bound to our feasible region according to $\mu^\star$. On this bounded polytope, we can now find a  Minkowski--Weyl decomposition and simply round the coefficients. The guarantees of this rounding are based on  the near-monotonicity property of the central path.

We note that, due to the reliance on the Minkowski--Weyl decomposition, the running time of this rounding procedure can be prohibitive in practice; its role is mainly conceptual, showing that rounding in the strongly polynomial model is  possible at all. We present loose running-time bounds in \Cref{sec:mw-explicit}. Developing more efficient  rounding methods remains for future research.

\paragraph{Organization.} The rest of the paper is structured as follows. Section~\ref{sec:prelim} introduces some necessary background.
 Section~\ref{sec:simplified-init} presents the `big-$M$' initialization scheme, followed by the `parameter-free' initialization scheme in Section~\ref{sec:analytic-init};  \Cref{thm:main-1} is proved in \Cref{sec:main-proof}.
  Section~\ref{sec:bit-complexity} describes the rounding procedure and proves \Cref{lem:main-round}. %
Appendix~\ref{sec:mw-explicit} gives a self-contained description of the strongly polynomial implementation of the Minkowski--Weyl theorem.

\section{Preliminaries}\label{sec:prelim}

\paragraph{Notation.}
We let $\Rplusplus$ denote the set of positive reals, and $\Rplus$ the set of nonnegative reals; similarly for $\Z_{>0},\Z_{\ge 0},\Q_{>0}$ and $\Q_{\ge 0}$.
 For $n\in \mathbb{N}$, we let
$[n]\coloneqq\{1,2,\ldots,n\}$.
We let $\0_n,\1_n\in \R^n$ denote the all-zero and all-one vectors, respectively. For $x\in\R^n$, we let $\supp(x)\subseteq [n]$ denote its support. For $\alpha\in\R$, we let $\alpha^+=\max\{\alpha,0\}$ and $\alpha^-=\max\{-\alpha,0\}$; for a vector $x\in\R^n$, we use $x^+$ and $x^-$ coordinatewise.

We let $\ker(\mA)\subseteq \R^n$ and $\im(\mA)\subseteq\R^m$ denote the kernel and image of the matrix $\mA \in \R^{m\times n}$.
The standard inner product of the
two vectors $x,y\in \R^n$  is denoted by $\pr{x}{y}=x^\top y$. For $x,y\in \R^n$, we let $x\circ y=(x_1y_1,\ldots,x_ny_n)$ denote the Hadamard product, and for $y\in\Rplusplus^n$, we  let $x/y=(x_1/y_1,\ldots,x_n/y_n)$.

\paragraph{Pseudoinverse.} For a matrix $\mM\in\R^{m\times n}$, we let $\mM^+\in\R^{n\times m}$ denote its Moore--Penrose pseudoinverse, i.e., the unique matrix such that $\mM^+ \mM u=u$ for all $u\in\im(\mM^\top)$, and $\ker(\mM^+)=\ker(\mM^\top)$. We use the following well-known properties:
\begin{lemma}\label{lem:pseudoinverse-prop}
For a matrix $\mM\in\R^{m\times n}$, we have that $\mM^+ \in \R^{n \times m}$, $\mM\mM^+\mM=\mM$, $\mM\mM^+$ and $\mI_m-\mM\mM^+$ are the orthogonal projection matrices onto the linear spaces $\im(\mM)$ and $\ker(\mM^\top)$, respectively.
\end{lemma}
These properties are standard. For computing $\mM^+$ and other basic linear algebra objects in the model of \Cref{sec:comp-mod}, where numbers cannot be reduced to lowest terms, we use the following lemma, proved in \Cref{sec:mw-cost}. The bounds use the na\"ive cubic bound for matrix multiplication, and we make no attempt to optimize them.
\begin{lemma}[Strongly polynomial linear algebra]\label{lem:linalg}
Let $\mM\in\eQ^{m\times n}$ with $m\le n$. Each of the following can be computed by an algorithm that is strongly polynomial in the Turing model on encoded rationals of size $\poly(n)\size(\mM)$. In particular, the output has size $\poly(n)\size(\mM)$.
\begin{enumerate}[label=(\roman*), ref=(\roman*)]\setlength{\itemsep}{0pt}
\item\label{it:linalg-basis} A basis of $\mM$, that is, $B\subseteq[n]$ such that $\mM_B$ is a maximal linearly independent set of columns of $\mM$, in $O(n^5)$ operations.
\item\label{it:linalg-inv} For $m=n$, the determinant $\det\mM$ and the adjugate $\adjt(\mM)$, and, if $\mM$ is nonsingular, the inverse $\mM^{-1}$, in $O(n^4)$ operations.
\item\label{it:linalg-pinv} The pseudoinverse $\mM^+$, in $O(n^5)$ operations.
\end{enumerate}
The same holds for $m>n$ in \ref{it:linalg-pinv}, with the roles of $m$ and $n$ exchanged, since $(\mM^\top)^+=(\mM^+)^\top$.
\end{lemma}

An important use of pseudoinverses will be for projecting out variables from an LP. Consider a partition $I\cup J=[n]$, and define
\begin{equation}\label{eq:proj-matrix}
\mPi\coloneqq \mI_m - \mA_{I} \mA_{I}^+\,,\quad
    \mH\coloneqq \mPi \mA_{J}\,,\quad\mbox{and}\quad h\coloneqq \mPi b\, .
\end{equation}
Here, $\mPi$ is the projection matrix onto $\ker(\mA_{I}^\top)$, i.e., the orthogonal complement of $\im(\mA_{I})$. The following fact maps solutions between the systems $\mA x=b$ and $\mH x_{J}=h$; the proof follows by the definitions and by \Cref{lem:pseudoinverse-prop}.
\begin{fact}\label{lem:lift} The linear system $\mA x=b$ is feasible if and only if $\mH x_{J}=h$ is feasible. Moreover,  for any solution $x_{J}\in\R^{J}$ such that $\mH x_{J}=h$, setting
\[
    x_{I}=\mA_{I}^+(b-\mA_{J} x_{J})
\]
satisfies $\mA_{J} x_{J}+\mA_{I} x_{I}=b$.
\end{fact}

\paragraph{Complementary solutions.}
Let us recall the  Goldman--Tucker theorem on strict complementarity \cite{Goldman1956} asserting that any maximum support primal and dual optimal solutions are strictly complementary.
\begin{theorem}[Goldman \& Tucker] Assume that both  the primal and dual LPs in \ref{LP_primal_dual} are feasible. Then, there exists
a unique partition $B\cup N=[n]$ such that there exist
optimal primal and dual solutions $x\in\Rplus^n$ and $(y,s)\in \R^m\times\Rplus^n$ with $B=\supp(x)$ and $N=\supp(s)$. %
\end{theorem}
We derive corollaries for the conic, dual-feasibility, and primal-feasibility settings. The corresponding complementary partitions will play important roles in our multi-stage initialization approaches.
\begin{corollary}\label{cor:maxsupp} Let $\mA\in\R^{m\times n}$, $b\in \R^m$, $c\in\R^n$.
\begin{enumerate}[label=(\roman*)]
\item\label{part:maxsupp-conic} There exists a unique partition $B\cup N=[n]$ and solutions $x,s\in\R^n$, $y\in\R^m$ such that
\[
 \mA  x=\0_m\, ,\, x\ge \0_n\, , \,  \supp(x)=B\,\quad\mbox{and}\quad \mA^\top y+s=\0_n\, , \, s\ge \0_n\, ,\, \supp(s)=N\, .
\]
\item\label{part:maxsupp-dual} Assume that $\mA^\top y+s=c$, $s\ge \0_n$ is feasible. Then,
there exists a unique partition $B\cup N=[n]$ and solutions $x,s\in\R^n$, $y\in\R^m$ such that
\[
  \mA x=\0_m\, , \, x\ge \0_n\, ,\, \supp(x)=B\,\quad\mbox{and}\quad \mA^\top y+s=c, \, s\ge \0_n\, ,\, \supp(s)=N\, .
\]
\item\label{part:maxsupp-primal} Assume that $\mA x=b$, $x\ge \0_n$ is feasible. Then,
there exists a unique partition $B\cup N=[n]$ and solutions $x,s\in\R^n$, $y\in\R^m$ such that
\[
  \mA x=b\, , \, x\ge \0_n\, ,\, \supp(x)=B\,\quad\mbox{and}\quad \mA^\top y+s=\0_n, \, s\ge \0_n\, ,\, \supp(s)=N\, .
\]
\end{enumerate}
\end{corollary}

From the interior point literature, initialization (Sections~\ref{sec:simplified-init} and \ref{sec:analytic-init}) only requires
the following lemma that shows that the choice of parameter $\mu = \pr{x}{s}/n$ is
essentially optimal for the purpose of minimizing centrality error at
$(x,s)$. Further preliminaries on interior point methods are deferred to \Cref{sec:ipm-preliminaries}.

\begin{lemma}[{\cite[Lemma 4.4]{Monteiro2003}}] \label{lem:choice-of-mu}
 For $\beta\in(0,1/4]$, let
$(x,s) \in \Primalp \times \Dualp$ and $\mu' > 0$ satisfy
$\|x\circ s/\mu' - \1_n\|_2 \leq \beta$. Then, for $\mu \coloneqq \pr{x}{s}/n$, we have that
\begin{equation}
  \left\|\frac{(x\circ s)}{\mu} -\1_n\right\|_2 \le \frac{\beta}{1-\beta}\ , \quad \text{ and } \quad \left(1-\frac{\beta}{\sqrt{n}}\right)\mu' \leq \mu \leq \left(1+\frac{\beta}{\sqrt{n}}\right)\mu'\ .
\end{equation}
\end{lemma}

\paragraph{The sublevel sets.}
Assume $\Primal,\Dual\neq\emptyset$, in which case \ref{LP_primal_dual} admits a  pair of primal and dual optimal solutions $(x^\star, y^\star,s^\star)$ with optimum value $\pr{c}{x^\star}=\pr{b}{y^\star}=v^\star$. Recall that this holds precisely if these two solutions are complementary: $\pr{x^\star}{s^\star}=0$; in particular, $x_i^\star s^\star_i=0$ for all $i\in [n]$.

Recall the definitions of the sublevel sets $\Primal_g$, $\Dual_g$ in \eqref{eq:sublevel-sets} as the sets of primal and dual solutions with objective value within $g$ from the optimum value $v^\star$.
The duality gap of any pair $(x, y, s)$ of primal-dual feasible points of~\ref{LP_primal_dual} satisfies $\pr{c}{x} - \pr{b}{y} = \pr{x}{s}$. In particular, we have $\pr{x}{s^\star} = \pr{c}{x} - v^\star$ and $\pr{x^\star}{s} = v^\star - \pr{b}{y}$. Thus, the two sets $\Primal_g$ and $\Dual_g$ are equivalently given by
\[
\Primal_g = \{x \in \Primal : \pr{x}{s^\star} \leq g\}\, \quad\mbox{and} \quad \Dual_g = \{s \in \Dual : \pr{x^\star}{s} \leq g\} \, .
\]
These expressions are in fact independent of the choice of optimal solutions
$(x^\star, y^\star,s^\star)$.
The following is immediate.
\begin{proposition}\label{prop:mcp_well_defined}
If $\Primalp,\Dualp\neq\emptyset$,
then for any $g,g'\ge 0$, and $i\in[n]$, $\mx_i(g)\cdot\ms_i(g')\le g+g'$ holds, and consequently, the sets  $\Primal_g$ and $\Dual_g$ are bounded.
\end{proposition}
When $\Primalp=\emptyset$ or $\Dualp=\emptyset$, but $\Primal,\Dual\neq\emptyset$, we may have one of the following extreme possibilities for some $i\in [n]$:
\begin{itemize}
  \item {\em Identically zero:} $\mx_i(g)=0$ for all $g\ge 0$ or $\ms_i(g)=0$ for all $g\ge 0$.
  \item {\em Identically infinite:}  $\mx_i(g)=\infty$ for all $g\ge 0$ or $\ms_i(g)=\infty$ for all $g\ge 0$.
\end{itemize}
These can be easily characterized as follows.
\begin{lemma}
Assume $\Primal,\Dual\neq\emptyset$.  Let $i\in[n]$.
\begin{enumerate}[label=(\roman*)]
\item If $\mx_i(g)=\infty$ for some $g\in\Rplus$, then $\mx_i(g')=\infty$ for every $g'\in \Rplus$, that  is, $\mx_i(g)$  is identically infinite.
If $\ms_i(g)=\infty$ for some $g\in\Rplus$, then $\ms_i(g)$  is identically infinite.
\item\label{part:mcp-x-zero} $\mx_i$ is identically zero if and only if
there exists $y'\in\R^m$, $s'\in\Rplus^n$, $\mA^\top y'+s'=\0_n$ such that $s'_i>0$, and $\pr{b}{y'}=0$.
\item\label{part:mcp-s-zero} $\ms_i$ is identically zero if and only if
there exists $x'\in\Rplus^n$, $\mA x'=\0_m$ such that $x'_i>0$, and $\pr{c}{x'}=0$.
\end{enumerate}
\end{lemma}
\begin{proof} For part \textbf{(i)}, $\mx_i(g)=\infty$ implies the existence of a vector $z\in\Rplus^n$, $\mA z= \0$, $z_i > 0$ and $\pr{c}{z}\le 0$. We further have $\pr{c}{z}=0$ since $\pr{c}{z}<0$ would imply $\Dual=\emptyset$. Hence, $x^\star+\alpha z$ is an optimal solution for any $\alpha\in\Rplus$, showing that $\mx_i(0)=\infty$, and consequently $\mx_i(g)$ is identically infinite. The proof for $\ms_i(g)$ is analogous.

For part \textbf{(ii)}, first assume $\mx_i$ is identically zero, meaning that $x_i=0$ for all $\mA x=b$, $x\ge 0$. Then, the existence of $y'\in\R^m$, $s'\in\Rplus^n$, $\mA^\top y'+s'=\0_n$ such that $s'_i>0$ is immediate by  \Cref{cor:maxsupp}\ref{part:maxsupp-primal}. For the objective value, let $x'$ be the primal solution as in \Cref{cor:maxsupp}\ref{part:maxsupp-primal}. Then, $\pr{b}{y'}=\pr{\mA x'}{y'}=\pr{x'}{\mA^\top y'}=-\pr{x'}{s'}=0$, since $x'$ and $s'$ have complementary supports. Conversely, if there exists a solution $\mA^\top y'+s'= \0_n$, $s'\ge\0_n$, $s'_i>0$ with $b^\top y'=0$, then for any $x\in\R^n$ such that $\mA x=b$, $x\ge 0$, it follows that $x_i s'_i\le x^\top s'=-x^\top \mA^\top y'=b^\top y'=0$, implying $x_i=0$. Consequently, $\mx_i$ is identically zero.
Part \textbf{(iii)} follows analogously.
 \end{proof}
 Note that according to parts \ref{part:mcp-x-zero} and \ref{part:mcp-s-zero}, $\mx_i$ is identically zero if and only if $\ms_i$ is identically infinite, and $\ms_i$ is identically zero if and only if $\mx_i$ is identically infinite.

\subsection{Operations preserving straight-line complexity}\label{sec:preserve}
The following lemma asserts that column rescalings do not affect the straight line complexity bounds.
\begin{lemma}\label{lem:scale-invariance}
Given an instance of \ref{LP_primal_dual} with data $(\mA,b,c)$, let $\mD\in\R^{n\times n}$ be a positive diagonal matrix, and consider the LP with data $(\mA \mD, b, \mD c)$. Let $(\mx,\ms)$ and $(\bmx,\bms)$ denote the max central paths of the two LPs. Then, $\SLC_{\eta}(\mx_i)=\SLC_{\eta}(\bmx_i)$ and $\SLC_{\eta}(\ms_i)=\SLC_{\eta}(\bms_i)$ for any $\eta\in [0,1]$ and $i\in [n]$.
\end{lemma}
\begin{proof}
We prove the claim only for the primal max central paths; the proof for the dual case is analogous.
Note that there is a one-to-one mapping $x\mapsto \mD^{-1}x$ between the feasible solutions of the two systems, with $\pr{c}{x}=\pr{\mD c}{\mD^{-1}x}$. This implies that $\bmx(g)=\mD^{-1}\mx(g)$ for all $g\ge 0$. It is now easy to see  that the two functions $\mx_i$ and $\bmx_i$ have the same straight line complexities.
\end{proof}
Our next lemma shows that the straight line complexities of the remaining variables are unchanged under two reductions: deleting the variables whose primal max central paths are identically zero, and projecting out the variables whose dual max central paths are identically zero.
\begin{lemma}\label{lem:project-delete}
Consider an instance of \ref{LP_primal_dual} along with a partition $P\cup D\cup R=[n]$ of the variables such that for every primal and dual feasible solution $(x,s)$ to \ref{LP_primal_dual},
$x_D=\0_D$ and $s_P=\0_P$ holds. Let $s$ be any dual feasible solution.
\begin{enumerate}[label=(\roman*)]
\item Consider the system
\begin{equation}\label{eq:reduce-sys-1}
\min\pr{s_{P\cup R}}{x_{P\cup R}}\, , \quad \mbox{s.t.} \quad \mA_{P\cup R} x_{P\cup R}= b \, , x_{P\cup R}\ge \0_{P\cup R}\, .
\end{equation}
Then, for any $i\in P\cup R$, we have $x^\m_i(g)=\bar x^\m_i(g)$ for all $g\in\Rplus$, where $x^\m_i(g)$ and $\bar x^\m_i(g)$ denote the max-central paths of the corresponding variables in \ref{LP_primal_dual} and in \eqref{eq:reduce-sys-1}, respectively.
\item Define
\[
\mPi\coloneqq \mI_m - \mA_{P} \mA_{P}^+\,,\quad
    \mH\coloneqq \mPi \mA_{D\cup R}\,,\quad\mbox{and}\quad h\coloneqq \mPi b\, ,
\]
and consider the system
\begin{equation}\label{eq:reduce-sys-2}
\min\pr{s_{D\cup R}}{x_{D\cup R}}\, , \quad \mbox{s.t.} \quad \mH x_{D\cup R}=h\, , x_{D\cup R}\ge \0_{D\cup R}\, .
\end{equation}
Then, for any $i\in D\cup R$, we have $x^\m_i(g)=\bar x^\m_i(g)$ for all $g\in\Rplus$, where $x^\m_i(g)$ and $\bar x^\m_i(g)$ denote the max-central paths of the corresponding variables in \ref{LP_primal_dual} and in \eqref{eq:reduce-sys-2}, respectively.
\item Define
\[
\mPi\coloneqq \mI_m - \mA_{P} \mA_{P}^+\,,\quad
    \mH\coloneqq \mPi \mA_{R}\,,\quad\mbox{and}\quad h\coloneqq \mPi b\, ,
\]
and consider the system
\begin{equation}\label{eq:reduce-sys-2-pd}
\min\pr{s_{R}}{x_{R}}\, , \quad \mbox{s.t.} \quad \mH x_{R}=h\, , x_{R}\ge \0_{R}\, .
\end{equation}
Then, for any $i\in R$, we have $x^\m_i(g)=\bar x^\m_i(g)$ for all $g\in\Rplus$, where $x^\m_i(g)$ and $\bar x^\m_i(g)$ denote the max-central paths of the corresponding variables in \ref{LP_primal_dual} and in \eqref{eq:reduce-sys-2-pd}, respectively.
\end{enumerate}
\end{lemma}
\begin{proof}{\bf Part (i).} First, note that in \ref{LP_primal_dual} and in $\LP(\mA,b,s)$, all variables $x_i$ have the same straight line complexities because the objective value is shifted by a constant.
Since \eqref{eq:reduce-sys-1} arises from $\LP(\mA,b,s)$ by deleting primal variables that are 0 in every feasible solution,  $x\in\R^n$ is a primal feasible solution to $\LP(\mA,b,s)$ if and only if  $x_{D}=\0_{D}$ and $x_{P\cup R}$ is feasible to \eqref{eq:reduce-sys-1}. This also implies that the two LPs have the same optimum value. The claim follows by the definition of the max central path.
\paragraph{Part (ii).}
As in Part (i), the max-central paths in \ref{LP_primal_dual} and in $\LP(\mA,b,s)$ coincide. It therefore suffices to compare $\LP(\mA,b,s)$ with \eqref{eq:reduce-sys-2}. We do so by exhibiting maps in both directions between the two feasible regions that preserve the objective value.

By \Cref{cor:maxsupp}\ref{part:maxsupp-dual}, the maximum support partition is $B\cup N=[n]$ with $N\subseteq D\cup R$ and $B\supseteq P$, and there exists a vector $z\in\Rplus^n$ with $\mA z=\0_m$ and $\supp(z)=B$. Moreover, $\pr{s}{z}=0$, as $\supp(s)\subseteq N$.

Let $x\in\Rplus^n$ with $\mA x=b$. Since $\mPi\mA_{P}=\0$, applying $\mPi$ to $\mA x=b$ gives $\mH x_{D\cup R}=h$; hence $x_{D\cup R}$ is feasible to \eqref{eq:reduce-sys-2}, and $\pr{s}{x}=\pr{s_{D\cup R}}{x_{D\cup R}}$ because $s_{P}=\0_{P}$. This map preserves the coordinates in $D\cup R$.

Conversely, let $\bar x\in\Rplus^{D\cup R}$ with $\mH\bar x=h$. By \Cref{lem:lift}, the vector $x$ with $x_{D\cup R}=\bar x$ and $x_{P}=\mA_{P}^+(b-\mA_{D\cup R}\bar x)$ satisfies $\mA x=b$, although possibly $x_{P}\not\ge\0_{P}$. Since $z_{P}>\0_{P}$, the shifted vector $x+\alpha z$ is nonnegative for $\alpha>0$ large enough. It satisfies $\mA(x+\alpha z)=b$ and, by $\pr{s}{z}=0$, also $\pr{s}{x+\alpha z}=\pr{s}{x}=\pr{s_{D\cup R}}{\bar x}$; thus $x+\alpha z$ is feasible to $\LP(\mA,b,s)$ with the same objective value as $\bar x$ in \eqref{eq:reduce-sys-2}. Furthermore, $\bar x_N = (x+\alpha z)_{N}$ by $\supp(z)=B$, so the map preserves the coordinates in $N$.

We thus have maps in both directions between the feasible regions of $\LP(\mA,b,s)$ and \eqref{eq:reduce-sys-2} that preserve objective values. Hence, the two LPs have the same optimum value, and both maps also preserve the optimality gap. As the first map preserves the coordinates in $D\cup R$ and the second map the coordinates in $N$, we obtain $\bar x^\m_i(g)\ge x^\m_i(g)$ for all $i\in D\cup R$ and $\bar x^\m_i(g)\le x^\m_i(g)$ for all $i\in N\subseteq D\cup R$, for all $g\in\Rplus$. Finally, consider a coordinate $i\in (D\cup R)\cap B$. Since $z_i>0$, for an optimal solution $x^\star$, $x^\star+\alpha z$ is another optimal solution with $(x^\star+\alpha z)_i$  arbitrarily large. Consequently, $\mx_i$ is identically infinite, and since  $\bar x^\m_i(g)\ge x^\m_i(g)$, $\bar x^\m_i$ is also identically infinite.

\paragraph{Part (iii).}  We first use part {\em (i)} to delete all primal variables in $D$. We claim that in the resulting system \eqref{eq:reduce-sys-1}, $\tilde s_P=\0_P$ holds for any feasible dual solution $(\tilde y,\tilde s)$. Indeed, by \Cref{cor:maxsupp}\ref{part:maxsupp-primal}, there exists a solution $(\hat y,\hat s)$ with $\mA^\top \hat y+\hat s=\0_n$, $\hat s\ge \0_n$, $D\subseteq \supp(\hat s)$. Note that, setting $\tilde s_D\coloneqq s_D-\mA_D^\top\tilde y$, for $\alpha>0$ such that $\tilde s_D+\alpha \hat s_D\ge \0_D$, $\tilde s+\alpha \hat s$ is also a feasible dual solution to the original system
\ref{LP_primal_dual}. Hence, if $\tilde s_P\gneq\0_P$, then we obtain a feasible dual solution to \ref{LP_primal_dual} not identically zero on $P$, a contradiction. Consequently, after deleting the primal variables in $D$, we can use part {\em (ii)} to project out the dual variables in $P$, and the claim follows.
\end{proof}

\section{Multi-stage big-$M$ initialization}\label{sec:simplified-init}
\subsection{The extended system}
\label{sec:extended-system}

In this section, we describe our first initialization procedure that is based on
the following variant of the ``big-$M$ system'' \eqref{eq:big-M} of \cite{Vavasis1996}, using different parameters for the primal and dual objectives.

\begin{equation}
\label{eq:big-M-sys} \tag{$\MSysName(\mA,b,c,M,K)$}
\begin{aligned}
\min \; \pr{c}{x}+&M\pr{\1_n}{\bar x} \\
\mA x-\mA \bar x & = b  \\
x+\hat x&=2K\1_n\\
x, \bar x,\hat x & \geq \0_n \, ,
\end{aligned}
\qquad \qquad
\begin{aligned}
\max \; \pr{b}{y}-&2K\pr{\1_n}{\hat s}\\
\mA^\top y + s-\hat s & = c \\
-\mA^\top y+\bar s&=M\1_n\\
s,\bar s,\hat s & \geq \0_n \, .
\end{aligned}
\end{equation}
The auxiliary variables play the following roles. The variable $\bar x$, penalized by $M$ in the objective, absorbs primal infeasibility: the first constraint reads $\mA(x-\bar x)=b$. The variable $\hat x=2K\1_n-x$ is the slack of the box bound $x\le 2K\1_n$ (the factor $2$ places the initial central point at $x=K\1_n$, see \Cref{lem:big-M}\ref{part:center}). The complementary pairs are $(x,s)$, $(\bar x,\bar s)$, and $(\hat x,\hat s)$.

The following lemma shows that for sufficiently large values of $K$ and $M$, optimal solutions to \ref{eq:big-M-sys} correspond to feasible or optimal solutions to \ref{LP_primal_dual}. Moreover, an initial central point for \ref{eq:big-M-sys} can be easily obtained.

\begin{lemma}\label{lem:big-M}
Let $\mA\in\R^{m\times n}$, $b\in\R^m$,  $d\in\R^n$ be such that $\mA d=b$, $c\in\R^n$, $\beta\in(0,1/2]$ %
Assume that
\begin{equation}\label{eq:M-1}
   K>\frac{4}{\beta}\|d\|_2\quad \mbox{and}\quad
M>\frac{4}{\beta}\|c\|_2\, .
 \end{equation}
Then, the following hold.
\begin{enumerate}[label=(\roman*)]
\item\label{part:center}
The vectors
\[
\begin{aligned}
(x,\bar x,\hat x)&\coloneqq (K\1_n,K\1_n-d,K\1_n)\, \quad\mbox{ and }\\
(y,s,\bar s,\hat s)&\coloneqq(\0_m,M\1_n+c,M\1_n,M\1_n)
\end{aligned}
\]
are feasible primal and dual solutions to \ref{eq:big-M-sys} and are in the neighborhood $\cN(\beta)$ for this system.

\item\label{part:both}
If $c,d> \0_n$, and
\begin{equation}\label{eq:N-3}
    K>\frac{\pr{c}{d}}{\min_{i\in[n]}c_i}\, ,\quad \mbox{and}\quad
M>\frac{\pr{c}{d}}{\min_{i\in[n]}d_i}\,
\end{equation}
hold, then every primal optimal solution $(x,\bar x,\hat x)$ and dual optimal solution $(y,s,\bar s,\hat s)$ to \ref{eq:big-M-sys}
satisfy $\bar x=\hat s=\0_n$. Consequently,
 $x$ and $(y,s)$ are primal and dual optimal solutions to \ref{LP_primal_dual}.
\item\label{part:primal}
If $d> \0_n$, and further
\begin{equation}\label{eq:M-2}
	M>\frac{4K\|c\|_1}{\min_{i\in[n]}d_i}\, ,
\end{equation}
holds,
then every primal optimal solution $(x,\bar x,\hat x)$ to \ref{eq:big-M-sys} satisfies $\bar x=\0_n$, and therefore  $x$ is feasible to the primal of \ref{LP_primal_dual}.
\item\label{part:dual}
If $c>\0_n$, and  further
\begin{equation}\label{eq:K-2}
    K>\frac{M\|d\|_1}{\min_{i\in[n]}c_i}\, ,
\end{equation}
holds,
then every dual optimal solution $(y,s,\bar s,\hat s)$ to \ref{eq:big-M-sys} satisfies $\hat s=\0_n$, and therefore $(y,s)$ is feasible to the dual of
\ref{LP_primal_dual}.
\end{enumerate}
\end{lemma}
\begin{proof}
{\bf Part \ref{part:center}.} Feasibility is immediate by the choice of $K$ and $M$.
We can write
\[
    \begin{aligned}
    \left\|\frac{(x,\bar x,\hat x)\circ(s,\bar s,\hat s)}{KM}-\1_{3n}\right\|_2&= \left\|\left(\frac{c}{M},-\frac{d}{K},\0_n\right)\right\|_2\\
    &\le \left\|\frac{c}{M}\right\|_2+\left\|\frac{d}{K}\right\|_2\le \frac{\beta}{4}+\frac{\beta}{4}=\frac{\beta}{2}\, .
    \end{aligned}
\]
Then, $(x,\bar x,\hat x),(s,\bar s,\hat s)$ are in the neighborhood $\cN(\beta)$ by \Cref{lem:choice-of-mu}, since $(\beta/2)/(1-\beta/2)<\beta$ for $\beta\in (0,1/2]$.

\paragraph{Part \ref{part:both}.}
Since $c,d > \0_n$, the vector $(d,\0_n,2K\1_n-d)$ is primal feasible and $(\0_m,c,M\1_n,\0_n)$ is dual feasible, with gap $\gamma\coloneqq\pr{c}{d}$ between them. In the remainder of the proof, let $(x,\bar x,\hat x)$ and $(y,s,\bar s,\hat s)$ be any pair of primal and dual optimal solutions to \ref{eq:big-M-sys}.
They satisfy complementary slackness:
\begin{equation}\label{eq:opt-compl}
    \pr{x}{s}=\pr{\bar x}{\bar s}=\pr{\hat x}{\hat s}=0\, .
\end{equation}
Let us start by showing that $\hat s=\0_n$.
The gap between  the primal optimal $(x,\bar x,\hat x)$  and dual feasible $(\0_m,c,M\1_n,\0_n)$ is at most $\gamma$, and therefore
 $\pr{x}{c}\le \gamma$, and using the choice of $K$,
\[
    0\le x_i\le \frac{\gamma}{c_i}<K\, ,\quad\forall i\in [n]\, .
\]
Therefore, $\hat x_i=2K-x_i>0$. Then, complementary slackness implies $\hat s_i=0$ for all $i\in[n]$.

The analogous argument showing $\bar x=\0_n$ is as follows. The gap between the dual optimum and the primal feasible $(d,\0_n,2K\1_n-d)$ must be $\le \gamma$, implying
$\pr{d}{s}\le \gamma$. Since both vectors are nonnegative,  using the definition of $M$ we get
\[
   s_i\le \frac{\gamma}{d_i}<M\, ,\quad\forall i\in [n]\, .
\]
From the dual constraints, for every $i\in [n]$ it follows that
\[
\begin{aligned}
    \bar s_i &= [\mA^\top y]_i + M \ge [\mA^\top y]_i - \hat s_{i} + M = c_i - s_i + M > c_i \ge 0.
\end{aligned}
\]
Consequently, by complementary slackness, $\bar x_i=0$ for all $i\in [n]$.
Hence, $x$ and $(y,s)$ are primal and dual feasible solutions to \ref{LP_primal_dual}. Moreover, they are complementary because $\pr{x}{s}=0$ as noted in \eqref{eq:opt-compl}. This shows that they are primal and dual optima.

\paragraph{Part \ref{part:primal}.} The proof follows similarly to the primal part of the proof above, with the `reference' primal and dual feasible solutions  $(d,\0_n,2K\1_n-d)$
and $(\0_m,c^+,M\1_n,c^-)$ to \ref{eq:big-M-sys}. These have duality gap
\[
    \gamma\coloneqq\pr{d}{c^+}+\pr{2K\1_n-d}{c^-}\le 2K\|c\|_1\, .
\]
Then, the gap between $(d,\0_n,2K\1_n-d)$ and a dual optimal solution
$(y,s,\bar s,\hat s)$ is $\le \gamma$. By the choice of $M$, we  obtain
\[
    0\le s_i\le \frac{\gamma}{d_i}<\frac{M}{2}\, \quad\forall i\in [n]\, .
\]
We can now write  for every $i\in [n]$,
\[
\bar s_i=[\mA^\top y]_i + M \ge [\mA^\top y]_i - \hat s_{i} + M = c_i - s_i + M > 0\, ,
\]
using that $s_i<M/2$ and $\|c\|_\infty\le \|c\|_2< M/4$.
By complementary slackness, every primal optimal solution $(x,\bar x, \hat x)$ satisfies $\bar x_i = 0$ for all $i\in [n]$.

\paragraph{Part \ref{part:dual}.} We now use the `reference' solutions
$(d^+,d^-,2K\1_n-d^+)$ and $(\0_m,c,M\1_n,\0_n)$ with duality gap
\[
    \gamma\coloneqq \pr{d^+}{c}+\pr{d^-}{M\1_n}\le M\|d\|_1\, .
\]
The gap between a primal optimal solution
 $(x,\bar x,\hat x)$ and $(\0_m,c,M\1_n,\0_n)$  is $\le \gamma$, and therefore
\[
    \pr{x}{c}\le \gamma \le M \|d\|_1\, ,
\]
and by the choice of $K$, we get $x_i \le c_i^{-1} \cdot M\|d\|_1 < K$ for all $i\in [n]$, which in turn gives $\hat x_i=2K-x_i>0$ and thereby $\hat s_i=0$ for all $i\in[n]$.\end{proof}

\subsection{The algorithm}\label{sec:big-M-alg}
The algorithm proceeds in three stages, each refining a partition of $[n]$. The \emph{final optimal} primal-dual solution to \ref{LP_primal_dual} is marked with a star $(x^\star,y^\star,s^\star)$, while plain symbols ($x,y,s$) denote generic decision variables or feasible points of the system under discussion; a superscript marks the stage of origin---$\conic$ for the Stage~\rom{1} conic solution, $\pf$ and $\df$ for the Stage~\rom{2} primal- and dual-feasibility solutions, and $\opt$ for the Stage~\rom{3} solve. Consistent with \Cref{sec:extended-system}, the two auxiliary variables of every big-$M$ system are marked by an accent: $\bar x$ is the $M$-penalized residual variable (so the first constraint reads $\mA(x-\bar x)=b$) and $\hat x$ is the slack of the bound $x\le 2K\1_n$, with complementary dual variables $\bar s,\hat s$. Certificates carry a tilde: $(\tilde y,\tilde s)$ from $\PMax$ and $\tilde x$ from $\DMax$. Index sets use a primal/dual base with the stage on the superscript: $\Pc,\Dc$ are the Stage~\rom{1} conic supports, $\Pf,\Df$ the Stage~\rom{2} feasible supports, and $R,I,J$ the Stage~\rom{3} refinement.

We frequently shift a vector $u\in\R^n$ by a nonnegative multiple of a vector $w\in\Rplus^n$, where $u_i\ge0$ for all $i\notin\supp(w)$. The smallest $\alpha\ge 0$ with $u+\alpha w\ge\0_n$ is
\[
\alpha_0\coloneqq\max\left(\{0\}\cup\left\{-u_i/w_i \mid i\in\supp(w)\right\}\right)\, .
\]
Whenever we choose $\alpha\ge0$ \emph{large enough} such that $u+\alpha w\ge\0_n$, we take $\alpha\coloneqq\alpha_0$. Whenever we require in addition that $u_i+\alpha w_i>0$ for all $i\in\supp(w)$, we take $\alpha\coloneqq\alpha_0+1$. Then, $u_i+\alpha w_i=(u_i+\alpha_0 w_i)+w_i\ge w_i>0$ for all $i\in\supp(w)$. In both cases, $\alpha$ is computed in strongly polynomial time and has size $O(\size(u,w))$.

We assume that the system of linear equations $\mA x=b$ is feasible: we are given a vector $d\in\R^n$ such that $\mA d=b$. We note that from the output of \Cref{thm:SLC-bound}, we only use the maximum support primal and dual solutions $x\out$ and $s\out$ in this section. We also fix $\beta\in (0,1/6]$ to be used throughout this section.

\subsubsection{Stage \rom{1}: Support partition}\label{sec:stage-I}
\Cref{cor:maxsupp}\ref{part:maxsupp-conic} postulates the existence of a partition $\Pc \cup \Dc = [n]$
and $x^\conic,s^\conic\in \R^n$, $y^\conic\in\R^m$ such that $\mA x^\conic=\0_m$, $x^\conic\ge \0$, $s^\conic=-\mA^\top y^\conic$, $s^\conic\ge \0$ with $\supp(x^\conic)=\Pc$ and  $\supp(s^\conic)=\Dc$.  Our goal in the first stage is to find this partition and the conic solutions $(x^\conic,s^\conic)$.

We can obtain these by running the IPM in \Cref{thm:SLC-bound} to solve $\MSys(\mA,\0_m,\0_n,1,1)$, that is,
\begin{equation}
\label{eq:big-M-sys-conic}
\begin{aligned}
\min \; & \pr{\1_n}{\bar x} \\
\mA x-\mA \bar x & = \0_m  \\
x+\hat x&=2\cdot\1_n\\
x, \bar x,\hat x & \geq \0_n \, ,
\end{aligned}
\qquad \qquad
\begin{aligned}
\max \; &-2\pr{\1_n}{\hat s}\\
\mA^\top y + s-\hat s & = \0_n \\
-\mA^\top y+\bar s&=\1_n\\
s,\bar s,\hat s & \geq \0_n \, .
\end{aligned}
\end{equation}
Clearly, we can initialize this with the solutions $(\1_n,\1_n,\1_n)$ and $(\0_m,\1_n,\1_n,\1_n)$ exactly on the central path. Let $(x^\conic,\bar x^\conic,\hat x^\conic)$ and $(y^\conic,s^\conic,\bar s^\conic,\hat s^\conic)$ denote the maximum support optimal solutions returned by the algorithm. The optimum value is 0, and therefore $\bar x^\conic=\hat s^\conic=\0_n$. We thus get $\mA x^\conic=\0_m$, $x^\conic\ge \0_n$ and $s^\conic=-\mA^\top y^\conic$, $s^\conic\ge \0_n$. Further, $\Pc=\supp(x^\conic)$ and $\Dc=\supp(s^\conic)$ form a partition of $[n]$, as the returned solutions are of maximum support according to \Cref{thm:SLC-bound}.

\subsubsection{Stage \rom{2}: Maximum-support feasibility}\label{sec:stage-II}
In the second stage, we define two subroutines to obtain the solutions asserted in \Cref{cor:maxsupp}\ref{part:maxsupp-dual} and \ref{part:maxsupp-primal}. The inputs and outputs of these subroutines are specified as follows.

\begin{subroutine}[\PMaxName]\label{sub:PMax}
$\PMax(\mA,b)$ takes a matrix $\mA\in\R^{m\times n}$ and a vector $b\in\R^m$ as input and returns either of the following outputs:
\begin{enumerate}[label=(\roman*)]
\item\label{itm:PMax-fea} A feasible solution $x^\pf\in\R^n$ to the primal system $\mA x=b$, $x\ge \0_n$ with maximum support $\Pf:=\supp(x^\pf)$, along with a dual certificate $(\tilde y,\tilde s)\in \R^m \times \R^n$ satisfying $\tilde s=-\mA^\top \tilde y$, $\tilde s\ge \0_n$ and $\supp(\tilde s)=[n]\setminus \Pf$.
\item\label{itm:PMax-inf} an infeasibility certificate $(\tilde y,\tilde s)\in \R^{m}\times \R^n$ for the primal system, where $\tilde s=-\mA^\top \tilde y$, $\tilde s\ge \0_n$ and $\pr{b}{\tilde y}>0$.
\end{enumerate}
\end{subroutine}

\begin{subroutine}[\DMaxName]\label{sub:DMax}
$\DMax(\mA,c)$ takes a matrix $\mA\in\R^{m\times n}$ and vector $c\in\R^n$ as input and returns either of the following outputs:
\begin{enumerate}[label=(\roman*)]
\item\label{itm:DMax-fea}  A feasible solution $(y^\df,s^\df)\in\R^m\times \R^n$ to the dual system $\mA^\top y+s=c$, $s\ge \0_n$ with maximum support $\Df:=\supp(s^\df)$, along with a primal certificate $\tilde x\in \R^n$ satisfying $\mA \tilde x=\0_m$, $\tilde x\ge \0_n$,  $\supp(\tilde x)=[n]\setminus \Df$.
\item\label{itm:DMax-inf} an infeasibility certificate $\tilde x\in\R^n$ for the dual system, where $\mA \tilde x=\0_m$, $\tilde x\ge \0_n$ and $\pr{c}{\tilde x}<0$.
\end{enumerate}
\end{subroutine}

We note that if the primal (dual) system is feasible, the output of $\PMax$ ($\DMax$) satisfies the following additional property.

\begin{lemma}\label{lem:b-y-c-x}
If $\mA x=b$, $x\ge \0_n$ is feasible, the output  $(\tilde y,\tilde s)$ of $\PMax(\mA,b)$ satisfies $\pr{b}{\tilde y}=0$.
If $\mA^\top y+s=c$, $s\ge \0_n$ is feasible, the output $\tilde x$ of $\DMax(\mA,c)$ satisfies  $\pr{c}{\tilde x}=0$.
\end{lemma}
\begin{proof}
In case~\ref{itm:PMax-fea}, let $x^\pf$ denote the primal output with support $\Pf$ such that $\supp(\tilde s)=[n]\setminus \Pf$. Then, $\pr{b}{\tilde y}=\pr{\mA x^\pf}{\tilde y}=\pr{x^\pf}{\mA^\top \tilde y}=-\pr{x^\pf}{\tilde s}=0$, because $x^\pf$ and $\tilde s$ have complementary supports. The second claim follows similarly.
\end{proof}
\paragraph{Dual feasibility.}
We start by describing $\DMax(\mA,c)$. Recall the conic solutions $x^\conic$, $(y^\conic,s^\conic)$ and the partition $(\Pc,\Dc)$ from Stage~\rom{1}.
First, we note that it  suffices to solve $\DMax(\mA_{\Pc},c_{\Pc})$. Given any solution $y\in\R^m$ such that $s \coloneqq c - \mA^\top y$ satisfies $s_{\Pc} \ge \0_{\Pc}$, the pair $(y+\alpha y^\conic,s+\alpha s^\conic)$ satisfies $s+\alpha s^\conic\ge \0_n$ and $\supp(s+\alpha s^\conic)= \supp(s_{\Pc}) \cup \Dc$ for sufficiently large $\alpha$. Hence, we may assume that $\Pc \neq \emptyset$. %

We solve $\DMax(\mA_{\Pc},c_{\Pc})$ by solving the system $\MSys(\mA_{\Pc},\0_m,c_{\Pc},M,K)$ for $M$ and $K$ chosen large enough to satisfy \eqref{eq:M-1} and \eqref{eq:M-2} in \Cref{lem:big-M} with the primal solution $d$ chosen as $x^\conic_{\Pc}>\0_{\Pc}$:
\begin{equation}
\label{eq:big-M-sys-2}
\begin{aligned}
\min \; \pr{c_{\Pc}}{x_{\Pc}}&+M\pr{\1_{\Pc}}{\bar x_{\Pc}} \\
\mA_{\Pc} x_{\Pc}-\mA_{\Pc} \bar x_{\Pc} & = \0_m  \\
x_{\Pc}+\hat x_{\Pc}&=2{K}\cdot\1_{\Pc}\\
x_{\Pc}, \bar x_{\Pc},\hat x_{\Pc} & \geq \0_{\Pc} \, ,
\end{aligned}
\qquad \qquad
\begin{aligned}
\max \; -2{K}&\pr{\1_{\Pc}}{\hat s_{\Pc}}\\
\mA_{\Pc}^\top y +s_{\Pc}-\hat s_{\Pc} & = c_{\Pc} \\
-\mA_{\Pc}^\top y+\bar s_{\Pc}&=M\1_{\Pc}\\
s_{\Pc},\bar s_{\Pc},\hat s_{\Pc} & \geq \0_{\Pc} \, .
\end{aligned}
\end{equation}
This is the instantiation of \ref{eq:big-M-sys} with $b = \0_m$ and cost vector $c_{\Pc}$ over the columns $\Pc$.
Note that $\mA_{\Pc}$ may not have full row rank, even though $\mA$ has full row rank.
Nevertheless, we can remove the redundant rows before running the IPM in \Cref{thm:SLC-bound}.
This is because any dual optimal solution of the resulting system can be trivially lifted to a dual optimal solution of the original system by filling $y$ with additional zeros. %

Let
$(x^\df_{\Pc},\bar x^\df_{\Pc},\hat x^\df_{\Pc})$  and $(y',s^\df_{\Pc},\bar s^\df_{\Pc},\hat s^\df_{\Pc})$  be the primal and dual optimal solutions to \eqref{eq:big-M-sys-2} obtained via \Cref{thm:SLC-bound}.
According to \Cref{lem:big-M}\ref{part:primal},
we have $\bar x^\df_{\Pc}=\0_{\Pc}$.
The dual objective function is nonpositive, so the optimum value is $\le 0$.
If the optimum value is $\pr{c_{\Pc}}{x^\df_{\Pc}}<0$, then we obtain a certificate that the dual program in $\LP(\mA_{\Pc},\0_m,c_{\Pc})$ is infeasible; we return the Farkas certificate  $\tilde x=(x^\df_{\Pc},\0_{\Dc})$ for the dual program  in $\LP(\mA,b,c)$. Otherwise, $\pr{c_{\Pc}}{x^\df_{\Pc}}=0$. Since the dual optimum value is also 0,  we also have $\hat s^\df_{\Pc}=\0_{\Pc}$.
 Therefore, $\mA_{\Pc}^\top y'+s^\df_{\Pc}=c_{\Pc}$, $s^\df_{\Pc}\ge \0_{\Pc}$, and $(y',s^{\df}_{\Pc})$ is a maximum support such solution by \Cref{thm:SLC-bound}. As noted above, $y^\df\coloneqq y'+\alpha y^\conic$ and $s^\df\coloneqq c-\mA^\top y^\df$ for large enough $\alpha$ give the output $(y^\df,s^\df)$ of $\DMax(\mA,c)$. Further, we can return $\tilde x=(x^\df_{\Pc},\0_{\Dc})$ as the corresponding primal certificate.

\paragraph{Primal feasibility.}
We now turn to the subroutine $\PMax(\mA,b)$.
We may assume that $\Dc \neq \emptyset$, as otherwise $d + \alpha x^\conic$ satisfies $\mA(d+\alpha x^\conic) = b$ and $d + \alpha x^\conic > \0_n$ for sufficiently large $\alpha\geq 0$.%
We define  $\mPi$ as the projection matrix onto $\ker(\mA_{\Pc}^\top)$ as in \eqref{eq:proj-matrix}, that is,
\[
\mPi\coloneqq \mI_m - \mA_{\Pc} \mA_{\Pc}^+\,,\quad
    \mH\coloneqq \mPi \mA_{\Dc}\,,\quad\mbox{and}\quad h\coloneqq \mPi b\, ,
\]
with the convention that $\mPi \coloneqq \mI_m$ if $\Pc = \emptyset$. Recall \Cref{lem:lift} on mapping solutions between the systems $\mA x=b$ and $\mH x_{\Dc}=h$.

From Stage~\rom{1}, we have the conic dual solution $(y^\conic,s^\conic)$ such that
$s^\conic=-\mA^\top y^\conic$, $s^\conic\ge \0$ with $\supp(s^\conic)=\Dc$.
Let us solve
  $\MSys(\mH,h,s^\conic_{\Dc},M,K)$ with $s^\conic_{\Dc}>\0_{\Dc}$, where $M$ and $K$ are chosen large enough to satisfy \eqref{eq:M-1} and \eqref{eq:K-2} in \Cref{lem:big-M}:
\begin{equation}
\label{eq:big-M-sys-3}
\begin{aligned}
\min \; \pr{s^\conic_{\Dc}}{x_{\Dc}}+{M}&\pr{\1_{\Dc}}{\bar x_{\Dc}} \\
\mH x_{\Dc}-\mH \bar x_{\Dc} & = h  \\
x_{\Dc}+\hat x_{\Dc}&=2K\1_{\Dc}\\
x_{\Dc}, \bar x_{\Dc},\hat x_{\Dc} & \geq \0_{\Dc} \, ,
\end{aligned}
\qquad \qquad
\begin{aligned}
\max \; \pr{h}{y}-&2K\pr{\1_{\Dc}}{\hat s_{\Dc}}\\
\mH^\top y + s_{\Dc}-\hat s_{\Dc} & = s^\conic_{\Dc} \\
-\mH^\top y+\bar s_{\Dc}&={M}\1_{\Dc}\\
s_{\Dc},\bar s_{\Dc},\hat s_{\Dc} & \geq \0_{\Dc} \, .
\end{aligned}
\end{equation}
This is the instantiation of \ref{eq:big-M-sys} with cost vector $s^\conic_{\Dc}$ on the projected system $(\mH,h)$.
{Note again that $\mH$ may not have full row rank, but we can remove the redundant rows from $\mH$ before running the IPM in \Cref{thm:SLC-bound}.} %

The next lemma shows how we can construct the output for $\PMax(\mA,b)$ from the solution to this system. The big-$M$ solve produces the optimum on the block $\Dc$; we write its main variable as $x^\pf_{\Dc}$, and the recovery extends it to a full feasible vector $x^\pf$ by constructing the $\Pc$-block.
\begin{lemma}\label{lem:PMax-recover}
 Let $(x^\pf_{\Dc},\bar x^\pf_{\Dc},\hat x^\pf_{\Dc})$ and $(y^\pf,s^\pf_{\Dc},\bar s^\pf_{\Dc},\hat s^\pf_{\Dc})$ be the primal and dual optimal solutions to \eqref{eq:big-M-sys-3} obtained via \Cref{thm:SLC-bound}.
 Let $x^\conic$ and $(y^\conic,s^\conic)$ be the maximum support primal and dual conic solutions from Stage~\rom{1} such that $\mA x^\conic=\0_m$, $x^\conic\ge \0$, $s^\conic=-\mA^\top y^\conic$, $s^\conic\ge \0$ with $\supp(x^\conic)=\Pc$ and  $\supp(s^\conic)=\Dc$.
  Let
\[
    z\coloneqq \mA_{\Pc}^+(b-\mA_{\Dc} x^\pf_{\Dc})\, .
\]
Let us pick $\alpha\ge 0$ such that $z+\alpha x^\conic_{\Pc}>\0$, and let
\[
\begin{aligned}
    x^\pf&=(x^\pf_{\Dc},x^\pf_{\Pc})\coloneqq \left(x^\pf_{\Dc},
   z\right)+\alpha x^\conic\, ,
    \\
     \tilde y&\coloneqq \mPi y^\pf+y^\conic\, ,\quad\mbox{and}\quad
    \tilde s\coloneqq -\mA^\top \tilde y\, .
\end{aligned}
\]
Then, $\tilde s\ge \0_n$, $\tilde s_{\Pc}=\0_{\Pc}$, and either
\begin{enumerate}
    \item $\pr{b}{\tilde y}>0$, or
    \item $\pr{b}{\tilde y}=0$, $\mA x^\pf=b$, $x^\pf\ge \0$, and $\supp(x^\pf)\cup \supp(\tilde s)=[n]$.
\end{enumerate}
\end{lemma}

\begin{proof}
Since $s^\conic_{\Dc}>\0_{\Dc}$, \Cref{lem:big-M}\ref{part:dual} implies that $\hat s^\pf_{\Dc}=\0_{\Dc}$, that is, $\mH^\top y^\pf+s^\pf_{\Dc}=s^\conic_{\Dc}$.
Substituting the definitions $\mH=\mPi \mA_{\Dc}$ and $s^\conic_{\Dc}=-\mA^\top_{\Dc} y^\conic$, we get $\mA_{\Dc}^\top \mPi y^\pf+s^\pf_{\Dc}=-\mA^\top_{\Dc} y^\conic$, that is,
\[
  \mA_{\Dc}^\top \tilde y+s^\pf_{\Dc}=\0_{\Dc}\, ,
\]
i.e., $\tilde s_{\Dc}=s^\pf_{\Dc}\ge \0_{\Dc}$.
Since $\mPi$ is the orthogonal projection matrix onto $\ker(\mA_{\Pc}^\top)$ (\Cref{lem:pseudoinverse-prop}), and $\mA^\top_{\Pc} y^\conic=\0_{\Pc}$, we see that $ \mA_{\Pc}^\top \tilde y=\0_{\Pc}$. In summary, $\tilde s\ge \0_n$ and $\tilde s_{\Pc}=\0_{\Pc}$.

Let us now consider the dual objective $\pr{h}{y}-2K\pr{\1_{\Dc}}{\hat s_{\Dc}}$  in \eqref{eq:big-M-sys-3}.
First, note that $\supp(s^\conic)=\Dc$, together with complementarity $\Pc = [n] \setminus \Dc$, implies $y^\conic\in \ker(\mA_{\Pc}^\top)$ and therefore $\mPi y^\conic=y^\conic$. In particular, $\mH^\top y^\conic=\mA_{\Dc}^\top\mPi y^\conic=\mA_{\Dc}^\top y^\conic$, and $\pr{h}{y^\conic}=\pr{\mPi b}{y^\conic}=\pr{b}{\mPi y^\conic}=\pr{b}{y^\conic}$.
Thus,
we get that $(y,s,\bar s,\hat s) \coloneqq(-y^\conic,\0_{\Dc},\1_{\Dc}-\mA_{\Dc}^\top y^\conic,\0_{\Dc})$ is a dual feasible solution with objective value $-\pr{h}{y^\conic}=-\pr{b}{y^\conic}$; here $\bar s_{\Dc} = \1_{\Dc}-\mA_{\Dc}^\top y^\conic=\1_{\Dc}+s^\conic_{\Dc}\ge \1_{\Dc} > \0_{\Dc}$.
Since $\hat s^\pf_{\Dc}=\0_{\Dc}$, the dual optimum is $ \pr{h}{y^\pf}\ge -\pr{b}{y^\conic}$, and therefore
\[
    \pr{b}{\tilde y}=\pr{b}{\mPi y^\pf+y^\conic}=\pr{h}{y^\pf}+\pr{b}{y^\conic}\ge 0\, .
 \]
If $\pr{b}{\tilde y}>0$, then we return $(\tilde y,\tilde s)$ as an infeasibility certificate for the system $\mA x=b$, $x\ge \0_n$ as the output of $\PMax(\mA,b)$.

The remaining case is when $\pr{b}{\tilde y}=0$, or equivalently, when $(-y^\conic,\0_{\Dc},\1_{\Dc}-\mA_{\Dc}^\top y^\conic,\0_{\Dc})$ is a dual optimal solution.
Complementary slackness between this and the primal  optimal solution $(x^\pf_{\Dc},\bar x^\pf_{\Dc},\hat x^\pf_{\Dc})$ implies $\bar x^\pf_{\Dc}=\0_{\Dc}$. We show that $\mA x^\pf=b$, $x^\pf\ge \0_n$. First, $\mA_{\Dc} x^\pf_{\Dc}+\mA_{\Pc} z=b$ follows by \Cref{lem:lift}; since $\mA x^\conic=\0_m$, we obtain $\mA x^\pf=b$. For nonnegativity, we have $x^\pf_{\Dc}\ge \0_{\Dc}$, and $\alpha$ can be chosen as required in the statement because $\supp(x^\conic)=\Pc$. This immediately gives $x^\pf_{\Pc}=z+\alpha x^\conic_{\Pc}>\0_{\Pc}$.

Finally, $\supp(x^\pf)\cup \supp(\tilde s)=[n]$ follows because we started with strictly complementary optimal solutions to \eqref{eq:big-M-sys-3}.
In particular, $\supp(x^\pf_{\Dc})$ and $\supp(s^\pf_{\Dc})$ form a partition of $\Dc$.
The proof is complete by noticing that $\tilde s_{\Dc} = s^\pf_{\Dc}$, and that $\supp(x^\pf)=\Pc\cup \supp(x^\pf_{\Dc})$.
\end{proof}

\subsubsection{Stage \rom{3}: Optimal partition and solution}\label{sec:stage-III}

The overall algorithm first calls $\PMax(\mA,b)$ and $\DMax(\mA,c)$. Assume that both calls return feasible outputs. Let $x^\pf$ be the maximum support primal solution returned by $\PMax(\mA,b)$ with support
$\Pf=\supp(x^\pf)$ and dual certificate $(\tilde y,\tilde s)$, and let $(y^\df,s^\df)$ be the maximum support dual solution returned by $\DMax(\mA,c)$ with $\Df=\supp(s^\df)$ and primal certificate $\tilde x$.
Let
\[
    R\coloneqq \Pf\cap \Df\, ,\quad I\coloneqq \Pf\setminus \Df\, ,\quad J\coloneqq \Df\setminus \Pf\, .
\]
 Since $\Pc\subseteq \Pf$ and $\Dc\subseteq \Df$, it follows that $[n]=\Pc\cup \Dc\subseteq \Pf\cup \Df=R\cup I\cup J$. Further, note that for the certificates we have $\supp(\tilde s)=[n]\setminus \Pf=J$, and $\supp(\tilde x)=[n]\setminus \Df=I$. Thus $I$ is the set of variables where every dual feasible solution vanishes, hence any primal solution is complementary on this set. Analogously for $J$ with roles reversed.
 Stage III resolves the variables in $R$, where both primal and dual feasible solutions can be strictly positive.
 We may assume that $R\neq \emptyset$, as otherwise $x^{\pf}$ and $(y^{\df}, s^{\df})$ are maximum support primal and dual optimal solutions to \ref{LP_primal_dual}. %

We now remove all coordinates in $J$ and project out the coordinates in $I$. That is, we let
    \[
\mPi\coloneqq \mI_m - \mA_{I} \mA_{I}^+\,,\quad
    \mH\coloneqq \mPi \mA_{R}\,,\quad\mbox{and}\quad h\coloneqq \mPi b\, ,
\]
{with the convention that $\mPi \coloneqq \mI_m$ if $I = \emptyset$}.
We then solve $\MSys(\mH,h,s_{R}^\df,M,K)$ with positive cost vector $s_{R}^\df>\0_{R}$ and given primal solution $d=x_{R}^\pf>\0_{R}$:
\begin{equation}
\label{eq:big-M-sys-4}
\begin{aligned}
\min \; \pr{s_{R}^\df}{x_R}+&M\pr{\1_R}{\bar x_R} \\
\mH x_R-\mH \bar x_R & = h  \\
x_R+\hat x_R&=2K\1_R\\
x_R, \bar x_R,\hat x_R & \geq \0_R \, ,
\end{aligned}
\qquad \qquad
\begin{aligned}
\max \; \pr{h}{y}-&2K\pr{\1_R}{\hat s_R}\\
\mH^\top y + s_R-\hat s_R & = s^{\df}_R \\
-\mH^\top y+\bar s_R&=M\1_R\\
s_R,\bar s_R,\hat s_R & \geq \0_R \, .
\end{aligned}
\end{equation}
This is the instantiation of \ref{eq:big-M-sys} on the projected system $(\mH,h)$ with cost $s^\df_{R}$, where $M$ and $K$ are chosen large enough to satisfy \eqref{eq:M-1} and \eqref{eq:N-3} in \Cref{lem:big-M}.
Note again that $\mH$ may not have full row rank, but we can remove the redundant rows from $\mH$ before running the IPM in \Cref{thm:SLC-bound}.

The next lemma, analogously to \Cref{lem:PMax-recover}, shows how to convert these to optimal primal and dual solutions to \ref{LP_primal_dual}.
\begin{lemma}\label{lem:opt-recover}
 Let $(x^\opt_{R},\bar x^\opt_{R},\hat x^\opt_{R})$ and $(y^\opt,s^\opt_{R},\bar s^\opt_{R},\hat s^\opt_{R})$ be the primal and dual optimal solutions to \eqref{eq:big-M-sys-4} obtained via \Cref{thm:SLC-bound}. Let $(\tilde y,\tilde s)$ be the certificate returned by $\PMax(\mA,b)$ and $\tilde x$ the certificate returned by $\DMax(\mA,c)$ such that
\[
\begin{aligned}
    \tilde s&=-\mA^\top \tilde y\, ,\, \pr{b}{\tilde y}=0\, ,\,\tilde s\ge \0_n\, ,\, \supp(\tilde s)=J\, ,\quad\mbox{and}\quad\\
     \mA \tilde x&=\0_m\,,\, \pr{c}{\tilde x}=0\,,\,\tilde x\ge \0_n\,,\,\supp(\tilde x)=I\, .
     \end{aligned}
\]
Let us define
\[
    z\coloneqq \mA_{I}^+(b-\mA_{R} x^\opt_{R})\, ,
\]
and pick $\alpha,\lambda\ge 0$ such that $z+\alpha \tilde x_{I}> \0_{I}$ and $\mA^{\top}_{J}(\mPi y^\opt + y^\df + \lambda \tilde y)<c_{J}$. Let
\[
\begin{aligned}
    x^\star&=(x^\star_{R},x^\star_{I},x^\star_{J})\coloneqq \left(x^\opt_{R},z,\0_{J}\right)+\alpha \tilde x\, ,
    \\
    y^\star&\coloneqq \mPi y^\opt+y^\df+\lambda \tilde y\, ,\quad\mbox{and}\quad
    s^\star\coloneqq c-\mA^\top  y^\star\, .
\end{aligned}
\]
Then, $x^\star$ and $(y^\star,s^\star)$ are maximum support primal and dual optimal solutions to \ref{LP_primal_dual}.
\end{lemma}
\begin{proof}
Since $x^\pf_{R},s^\df_{R}>\0_{R}$, \Cref{lem:big-M}\ref{part:both} implies that $x^\opt_{R}$ and $(y^\opt,s^\opt_{R})$ are primal and dual optimal solutions to $\LP(\mH,h,s^\df_{R})$. We show that their modifications $x^\star$ and $(y^\star,s^\star)$  give primal and dual optima on the original system \ref{LP_primal_dual}. The supports $\supp(\tilde x)=I$ and   $\supp(\tilde s)=J$ enable the selection of suitable $\alpha,\lambda\ge 0$.

On the primal side, the choice of $z$, with the same argument as in \Cref{lem:lift},  and $\mA \tilde x=\0_m$ guarantees that $\mA x^\star=b$. We have $x^\star\ge \0_n$ by $x^\opt_{R}\ge \0_{R}$ and by the choice of $\alpha$. Thus, $x^\star$ is primal feasible.

Turning to the dual side, on the coordinates in $R$ we get
\[
  \mH^\top y^\opt+s^\opt_{R}=s_{R}^\df\,,\quad   \mA_{R}^\top y^\df+s_{R}^\df=c_{R}\, ,\quad \mbox{and}\quad \mA^\top_{R} \tilde y=\0_{R}\, .
\]
 Since $\mH^\top y^\opt=\mA_{R}^\top \mPi y^\opt$, we get that
$s^\star_{R}=s^\opt_{R}\ge \0_{R}$. On the coordinates in $I$ we have
\[
   \mA_{I}^\top \mPi y^\opt=\0_{I}\,\,,\quad   \mA_{I}^\top y^\df=c_{I}\, ,\quad \mbox{and}\quad \mA^\top_{I} \tilde y=\0_{I}\, ,
\]
where the first equation holds because $\mPi$ is an orthogonal projection matrix onto $\ker(\mA_{I}^\top)$. Hence, $s^\star_{I}=\0_{I}$. Finally $s^\star_{J}>\0_{J}$ by the choice of $\lambda$.

To show optimality and strict complementarity, we note that $\supp(x^\star)=\supp(x^\opt_{R})\cup I$ and $\supp(s^\star)=\supp(s^\opt_{R})\cup J$. The optima $x^\opt_{R}$ and $s^\opt_{R}$ are strictly complementary on $R$, since \Cref{thm:SLC-bound} returns maximum-support optimal solutions to \eqref{eq:big-M-sys-4}; hence $\supp(x^\opt_{R})$ and $\supp(s^\opt_{R})$ partition $R$. Together with the disjoint sets $I,J$, it follows that $[n]$ is the disjoint union of $\supp(x^\star)$ and $\supp(s^\star)$, proving that $x^\star$ and $(y^\star,s^\star)$ are maximum support optima.
\end{proof}

\subsection{Proof of a weaker version of Theorem~\ref{thm:main-1}} \label{sec:proof-weaker}

The proof of Theorem~\ref{thm:main-1} will require the stronger form of the initialization process described in \Cref{sec:analytic-init}. This also requires the dual certificates provided in \Cref{thm:SLC-bound}; on the other hand, one can use simpler systems in Stages~\rom{2} and~\rom{3}.

We now explain the slightly weaker result that could be derived based on the algorithm in this section. First, the conic system \eqref{eq:big-M-sys-conic}  used in
Stage~\rom{1} is the same as \eqref{sys:2}, except that the right-hand side is $2\cdot \1_n$ in \eqref{eq:big-M-sys-conic}. This scaling clearly maintains the same straight line complexities. Hence, the straight line complexities of the variables in \eqref{eq:big-M-sys-conic} are the same as those for \eqref{sys:2}.

Consider now \eqref{eq:big-M-sys-2} solved in Stage~\rom{2} for dual feasibility. This arises from
\begin{equation}
\min \pr{c}{x}+M\pr{\1_n}{\bar x}\, , \quad \mA x-\mA \bar x=\0_m\,,\quad x+\hat x=2K\cdot \1_n\, , \quad x,\bar x,\hat x\ge \0_n\, ,\tag{S-Dual'} \label{sys:3p}
\end{equation}
after deleting the variables corresponding to $\Dc$.
Similarly, \eqref{eq:big-M-sys-3}  for primal feasibility arises from
\begin{equation}
\min \pr{s^\conic}{x}+ M\pr{\1_n}{\bar x}\, ,\quad \mA x-\mA \bar x=b\,,\quad x+\hat x=2K\cdot\1_n\, ,\quad x,\bar x,\hat x\ge \0_n\, , \tag{S-Primal'} \label{sys:4p}
\end{equation}
after projecting out the variables corresponding to $\Pc$.
Finally, \eqref{eq:big-M-sys-4} solved for optimization in Stage~\rom{3} arises from
\begin{equation}
\min \pr{s^\df}{x}+M\pr{\1_n}{\bar x}\, ,\quad \mA x-\mA \bar x=b\,,\quad x+\hat x=2K\cdot\1_n\, ,\quad x,\bar x,\hat x\ge \0_n\, , \tag{S-Opt'} \label{sys:1p}
\end{equation}
after deleting and projecting out variables.

Hence, one could derive a version of \Cref{thm:main-1} where the running time is bounded in terms of the maximum of the straight line complexities of the systems \eqref{sys:1p}, \eqref{sys:2}, \eqref{sys:3p}, and \eqref{sys:4p}, for suitable choices of $M$, $K$, and the cost functions $(c, M\1_n)$, $(s^\conic, \1_n)$, and $(s^\df, M\1_n)$, and after deleting and projecting variables. We do not state such a form; note that \eqref{sys:1}, \eqref{sys:2}, \eqref{sys:3}, and \eqref{sys:4} in \Cref{thm:main-1} do not require such parameter choices, and \Cref{thm:main-1} does not require deleting and projecting variables.

However, we note that for many natural classes of LPs, the systems \eqref{sys:1p}, \eqref{sys:2}, \eqref{sys:3p}, and \eqref{sys:4p} would still fall into the same class, and the class is closed under variable deletion and projection. %
 For example, if the starting LP corresponds to a 2VPI primal-dual pair, all obtained systems also belong to this family, after straight-line-complexity-preserving transformations; we omit the details. In particular, for generalized flow problems, variable deletion and projection correspond to deleting and contracting arcs, respectively.

\section{Multi-stage analytic center initialization}\label{sec:analytic-init}

We now present the second initialization approach that yields the guarantees in \Cref{thm:main-1}. We make use of the dual certificates provided in \Cref{thm:SLC-bound} that show that the optimal solutions returned are near the analytic centers of their optimal faces.
We proceed in the same three stages as in \Cref{sec:big-M-alg}.
We will use this theorem for different neighborhoods in the three stages, with neighborhood parameters
\begin{equation}
 \label{def:beta-parameters}
 \beta^\conic \coloneqq 2^{-16}\, ,\quad \beta^\feas \coloneqq 2^{-8}\, ,\quad  \beta^\opt\coloneqq  2^{-4}\,
\end{equation}
for the conic, feasibility, and optimization stages, respectively.
These parameters leave enough room for the certificates to hold only with $3\beta$ in place of $\beta$, as happens in the PSPACE implementation; see the proof of \Cref{thm:main-1}.
 Again, we use a vector $d\in\R^n$ such that $\mA d=b$.
Such a $d$ with $\size(d)=\poly(n)\size(\mA,b)$ can be computed in $O(n^5)$ operations as $d\coloneqq\mA^+b$, using \Cref{lem:linalg}\ref{it:linalg-pinv}.
\subsection{Stage \rom{1}: Conic feasibility}\label{sec:ac-stage-I}
This stage is identical to the algorithm in \Cref{sec:big-M-alg}. We run the IPM in \Cref{thm:SLC-bound} to solve $\MSys(\mA,\0_m,\0_n,1,1)$ (i.e., solving system \eqref{eq:big-M-sys-conic}) using the neighborhood parameter $\beta^\conic$.
However, from the output we will also require the dual certificate. That is, we obtain:
\begin{itemize}
    \item A partition $\Pc \cup \Dc = [n]$, and primal and dual optimal solutions
$(x^\conic,\bar x^\conic,\hat x^\conic)$ and $(y^\conic,s^\conic,\bar s^\conic,\hat s^\conic)$ to \eqref{eq:big-M-sys-conic} with $\bar x^\conic=\hat s^\conic=\0_n$, and
$\supp(x^\conic)=\Pc$ and  $\supp(s^\conic)=\Dc$.
\item Vectors $(v^\conic,\hat v^\conic)\in \R^{2n}$, $u^\conic\in\R^{m}$,  $v^\conic=\mA^\top u^\conic+\hat v^\conic$ such that
\begin{equation}\label{eq:conic-cert-B}
    \left\|(x^\conic_{\Pc}\circ v^\conic_{\Pc},\hat x^\conic_{\Pc}\circ \hat v^\conic_{\Pc})-\1_{2|\Pc|}\right\|_2 \le\beta^\conic\, ,
\end{equation}
\item Vectors $(w^\conic,\bar w^\conic)\in\R^{2n}$, $\mA(w^\conic-\bar w^\conic)=\0_m$ such that
\begin{equation}\label{eq:conic-cert-N}
    \left\|(s^\conic_{\Dc}\circ w^\conic_{\Dc},\bar s^\conic_{\Dc}\circ \bar w^\conic_{\Dc})-\1_{2|\Dc|}\right\|_2 \le\beta^\conic\, .
    \end{equation}
\end{itemize}
We want to emphasize that the conditions in \eqref{eq:conic-cert-B} and \eqref{eq:conic-cert-N} imply that $w^\conic_{\Dc}, \bar w^\conic_{\Dc}, v^\conic_{\Pc},\hat v^\conic_{\Pc}$ are all strictly positive.
Here, we only presented the parts of the certificates that will be relevant in Stage \rom{2}.

\subsection{Stage \rom{2}: Maximum support feasible solutions}\label{sec:ac-stage-II}

In the second stage, we define two modifications of $\PMax$ and $\DMax$~that deliver stronger outputs as follows. Here, $\Pc$ and $\Dc$ are the partition from Stage \rom{1}.

\begin{subroutine}[\PMaxAName]\label{sub:PMaxA}
$\PMaxA(\mA,b)$ takes a matrix $\mA\in\R^{m\times n}$ and vector $b\in\R^m$ as input and returns either of the following outputs:
\begin{enumerate}[label=(\roman*)]
\item\label{itm:PMaxA-fea} A maximum support solution $x^\pf\in\R^n$ to the primal feasibility system $\mA x=b$, $x\ge \0_n$ with $\Pf=\supp(x^\pf)$, and a dual certificate $(\tilde y,\tilde s)$ with $\tilde s=-\mA^\top \tilde y$, $\tilde s\ge \0_n$, $\supp(\tilde s)=[n]\setminus \Pf$, along with vectors
$u^\pf\in\R^{m}$, $v^\pf=\mA^\top u^\pf\in\R^{n}$, $v^\pf_{\Pc}=\0_{\Pc}$ such that
\begin{equation}\label{eq:Pmax-cert}
      \left\|(x^\pf_{\Pf\cap \Dc}\circ  v^\pf_{\Pf\cap \Dc})-\1_{\Pf\cap \Dc}\right\|_2\le\beta^\feas\, .
\end{equation}
\item\label{itm:PMaxA-inf} An infeasibility certificate $(\tilde y,\tilde s)\in \R^{m}\times \R^n$ for the primal problem with $\tilde s=-\mA^\top \tilde y$, $\tilde s\ge \0_n$, $\pr{b}{\tilde y}>0$.
\end{enumerate}
\end{subroutine}

\begin{subroutine}[\DMaxAName]\label{sub:DMaxA}
$\DMaxA(\mA,c)$ takes a matrix $\mA\in\R^{m\times n}$ and vector $c\in\R^n$ as input and returns either of the following:
\begin{enumerate}[label=(\roman*)]
\item\label{itm:DMaxA-fea} A solution $(y^\df,s^\df)\in\R^m\times \R^n$ with maximum support $s^\df$ to $\mA^\top y+s=c$, $s\ge \0_n$ with support $\Df=\supp(s^\df)$, along with a primal certificate $\mA \tilde x=\0_m$, $\tilde x\ge \0_n$,  $\supp(\tilde x)=[n]\setminus \Df$, as well as  a vector
$w^\df\in\R^{n}$, $\mA w^\df=\0_m$, $w^\df_{\Dc}=\0_{\Dc}$ such that
\begin{equation}\label{eq:Dmax-cert}
      \left\|(s^\df_{\Df\cap \Pc}\circ  w^\df_{\Df\cap \Pc})-\1_{\Df\cap \Pc}\right\|_2\le\beta^\feas\, .
\end{equation}
\item\label{itm:DMaxA-inf} An infeasibility certificate $\tilde x\in\R^n$ for the dual problem satisfying $\mA \tilde x=\0_m$, $\tilde x\ge \0_n$, $\pr{c}{\tilde x}<0$.
\end{enumerate}
\end{subroutine}

\paragraph{Dual feasibility}
We start by describing $\DMaxA(\mA,c)$. As in Section~\ref{sec:big-M-alg}, we can restrict our attention to the subproblem defined by $(\mA_{\Pc},c_{\Pc})$; this will also suffice to produce the certificate \eqref{eq:Dmax-cert}.
We may assume that $\Pc \neq \emptyset$ and $c_{\Pc}\neq \0_{\Pc}$. Otherwise, we have $\Df = \Dc$, so we output $(y^{\df}, s^{\df}) \coloneqq (\alpha y^\conic, c+ \alpha s^\conic)$ for sufficiently large $\alpha\geq 0$ such that $s^{\df}_{\Dc} > \0_{\Dc}$. We also output the primal certificate $\tilde x \coloneqq x^\conic$ and the vector $w^{\df} \coloneqq x^\conic$, noticing that \eqref{eq:Dmax-cert} is vacuously true.

Instead of \eqref{eq:big-M-sys-2}, we solve a simpler system where the residual variables $\bar x_{\Pc}$ and $\bar s_{\Pc}$ are dropped, and no $M$ parameter is needed. Namely, we consider the system
\begin{equation}
\label{eq:big-M-sys-2-A}
\begin{aligned}
\min &\; \pr{c_{\Pc}}{x_{\Pc}}\\
\mA_{\Pc} x_{\Pc} & = \0_m  \\
x_{\Pc}+\hat x_{\Pc}&=2\cdot \1_{\Pc}\\
x_{\Pc}\, ,\hat x_{\Pc} & \geq \0_{\Pc} \, ,
\end{aligned}
\qquad \qquad
\begin{aligned}
\max \; -2&\pr{\1_{\Pc}}{\hat s_{\Pc}}\\
\mA_{\Pc}^\top y +s_{\Pc}-\hat s_{\Pc} & = c_{\Pc} \\
s_{\Pc},\hat s_{\Pc} & \geq \0_{\Pc} \, .
\end{aligned}
\end{equation}
We use the primal solution $(x^\conic,\hat x^\conic)\in \R^{2n}$ from Stage~\rom{1} along with the certificate vectors
$(v^\conic,\hat v^\conic)\in \R^{2n}$, $u^\conic\in\R^{m}$,  $v^\conic=\mA^\top u^\conic+\hat v^\conic$
 satisfying \eqref{eq:conic-cert-B} to find an initial solution to this system. %
Namely, we define
\[
\begin{aligned}
\mu&\coloneqq \frac{2\|c_{\Pc}\|_1}{\beta^\conic}\, ,\\
     y&\coloneqq -\mu  u^\conic \, ,\\
    (s_{\Pc},\hat s_{\Pc})&\coloneqq (c_{\Pc}+\mu v^\conic_{\Pc},\mu \hat v^\conic_{\Pc})\, .
\end{aligned}
\]
{Note that $\mu>0$ by our assumption.}
We then get the centrality bound between the primal solution $(x^\conic_{\Pc},\hat x^\conic_{\Pc})$ and the dual solution $(y,s_{\Pc},\hat s_{\Pc} )$ as
\begin{equation}
    \begin{aligned}
    \left\|\frac{(x^\conic_{\Pc}, \hat x^\conic_{\Pc}) \circ (s_{\Pc},\hat s_{\Pc} )}{\mu} - \1_{2|\Pc|}\right\|_2 &=
    \left\|(x^\conic_{\Pc}, \hat x^\conic_{\Pc}) \circ (v^\conic_{\Pc}, \hat v^\conic_{\Pc}) - \1_{2|\Pc|} + \frac{(x^\conic_{\Pc},\hat x^\conic_{\Pc})\circ(c_{\Pc}, \0_{\Pc})}{\mu}\right\|_2
     \\
     &\le  \left\|(x^\conic_{\Pc}, \hat x^\conic_{\Pc}) \circ (v^\conic_{\Pc}, \hat v^\conic_{\Pc}) - \1_{2|\Pc|}\right\|_2 + \frac{\left\|x^\conic_{\Pc}\circ c_{\Pc}\right\|_2}{\mu} \\
     &\le \beta^\conic + \beta^\conic = 2\beta^\conic\, ,
    \end{aligned}
\end{equation}
where the last inequality uses that $\|x^\conic_{\Pc}\|_\infty\le 2$, and therefore $\|x^\conic_{\Pc}\circ c_{\Pc}\|_2\le 2\|c_{\Pc}\|_2$. From \Cref{lem:choice-of-mu}, and the choice of $\beta^\conic$ and $\beta^\feas$, we get that
\[
    \left((x^\conic_{\Pc}, \hat x^\conic_{\Pc}) ,(s_{\Pc},\hat s_{\Pc} )\right) \in \cal{N}^2\left(\frac{2\beta^\conic}{1 - 2 \beta^\conic}\right) \subseteq \cal{N}^2({\beta^\feas})\, .
\]

Hence, we can run the IPM in \Cref{thm:SLC-bound} for the system \eqref{eq:big-M-sys-2-A}
using this starting solution and neighborhood parameter $\beta=\beta^\feas$. Let
$(x^\df_{\Pc},\hat x^\df_{\Pc})$  and $(y',s^\df_{\Pc},\hat s^\df_{\Pc})$  denote the primal and dual optimal solutions returned.
As in \Cref{sec:big-M-alg}, if $\pr{c_{\Pc}}{x^\df_{\Pc}}<0$, then we obtain a Farkas certificate $\tilde x\coloneqq (x^\df_{\Pc},\0_{\Dc})$ for the dual program in \ref{LP_primal_dual}. Otherwise, $\pr{c_{\Pc}}{x^\df_{\Pc}}=0$ and consequently, $\hat s^\df_{\Pc}=\0_{\Pc}$. In this case, we obtain the maximum support dual solution $(y',s^\df_{\Pc})$ to $\LP(\mA_{\Pc},\0_m,c_{\Pc})$, that is, $\mA_{\Pc}^\top y'+s^\df_{\Pc}=c_{\Pc}$ and $s^\df_{\Pc}\ge \0_{\Pc}$. Let $D'\coloneqq \supp(s^\df_{\Pc})$. Using the conic solution $(y^\conic, s^\conic)$, we set $y^\df\coloneqq y'+\alpha y^\conic$ and $s^\df\coloneqq c-\mA^\top y^\df$ for large enough $\alpha>0$, with $\alpha$ chosen such that $s^\df_{\Dc}>\0_{\Dc}$. This gives a dual solution to \ref{LP_primal_dual}. The shift leaves the block $s^\df_{\Pc}$ unchanged since $\mA_{\Pc}^\top y^\conic=\0_{\Pc}$, so $\Df\coloneqq\supp(s^\df)=D'\cup\Dc$ and in particular $\Df\cap \Pc=D'$. We also get the desired primal certificate $\tilde x\coloneqq(x^\df_{\Pc},\0_{\Dc})$, which satisfies $\mA \tilde x = \0_m$, $\tilde x\geq \0_n$ and $\supp(\tilde x) = [n]\setminus D^{\df}$.

 From  \Cref{thm:SLC-bound}, we also obtained $(w',\hat w')\in\R^{2|\Pc|}$, such that $\mA_{\Pc} w'=\0_m$, $w'+\hat w'=\0_{\Pc}$, and
 \[
     \left\|(s^\df_{D'}\circ w'_{D'})-\1_{D'}\right\|_2\le \beta^\feas\, ,
 \]
 and thus, $w^\df\coloneqq(w',\0_{\Dc})$ gives the required analytic center certificate  \eqref{eq:Dmax-cert} in the feasibility case, noting that $D'=\Df\cap \Pc$.

\paragraph{Primal feasibility.}
As in \Cref{sec:big-M-alg}, we define
\[
\mPi\coloneqq \mI_m - \mA_{\Pc} \mA_{\Pc}^+\,,\quad
    \mH\coloneqq \mPi \mA_{\Dc}\,,\quad\mbox{and}\quad h\coloneqq \mPi b\, .
\]
{We may assume that $\Dc \neq \emptyset$ and $d_{\Dc} \neq \0_{\Dc}$. Otherwise, we have $\Pf = \Pc$, so we output $x^{\pf} \coloneqq d + \alpha x^\conic$ for sufficiently large $\alpha\geq 0$ such that $x^{\pf}_{\Pc}>\0_{\Pc}$. We also output the dual certificate $(\tilde y, \tilde s) \coloneqq (y^\conic, s^\conic)$ and the vectors $u^{\pf} \coloneqq -y^\conic$, $v^{\pf} \coloneqq A^\top u^{\pf} = s^\conic$, noticing that \eqref{eq:Pmax-cert} is vacuously true.} %

 From the system \eqref{eq:big-M-sys-3}, we drop the bound variables $\hat x_{\Dc}$ and $\hat s_{\Dc}$, and accordingly the parameter $K$, and consider
\begin{equation}
\label{eq:big-M-sys-3-A}
\begin{aligned}
\min \; \pr{\1_{\Dc}}{\bar x_{\Dc}} \\
\mH x_{\Dc}-\mH \bar x_{\Dc} & = h  \\
x_{\Dc}, \bar x_{\Dc}& \geq \0_{\Dc} \, ,
\end{aligned}
\qquad \qquad
\begin{aligned}
\max &\; \pr{h}{y}\\
\mH^\top y +s_{\Dc}& = \0_{\Dc}\\
-\mH^\top y+\bar s_{\Dc}&=\1_{\Dc}\\
s_{\Dc},\bar s_{\Dc} & \geq \0_{\Dc} \, .
\end{aligned}
\end{equation}
We initialize the system using the dual solution $(y^\conic,s^\conic,\bar s^\conic)$ from Stage~\rom{1}, along with the certificate $(w^\conic,\bar w^\conic)$ as in \eqref{eq:conic-cert-N}. Observe that $\mPi y^\conic = y^\conic$ because $y^\conic \in \ker(\mA^\top_{\Pc})$; this implies $\mH^\top y^\conic=\mA_{\Dc}^\top y^\conic$.
Further, for the solution $d\in\R^n$ of $\mA d=b$, we have $\mH d_{\Dc}=h$ because $\mPi \mA_{\Pc} d_{\Pc}=\0_m$.
We construct the primal solution as
\[
\begin{aligned}
\mu&\coloneqq \frac{\|d_{\Dc}\|_1}{\beta^\conic}\, ,\\
    (x_{\Dc},\bar x_{\Dc})&\coloneqq (d_{\Dc}+\mu w^\conic_{\Dc},\mu \bar w^\conic_{\Dc})\, .
\end{aligned}
\]
{Note that $\mu>0$ by our assumption.}

Primal feasibility is verified by observing that $ x_{\Dc}, \bar x_{\Dc}\geq \0_{\Dc}$ and $\mH(x_{\Dc}-\bar x_{\Dc})=\mH d_{\Dc}+\mu\mH(w^\conic_{\Dc}-\bar w^\conic_{\Dc})=h$, where the last equality is due to $\mPi \mA_{\Dc}(w^\conic_{\Dc}-\bar w^\conic_{\Dc}) = \mPi\mA(w^\conic - \bar w^\conic) = \0_m$.
An argument analogous to that in the dual feasibility case shows that
\[
     \left((x_{\Dc},  \bar x_{\Dc}) ,(s^\conic_{\Dc},\bar s^\conic_{\Dc})\right) \in \cal{N}^2({\beta^\feas})\, .
\]
Indeed, in the conic system the optimum has $\hat s^\conic=\0$, so $s^\conic=-\mA^\top y^\conic$ and $\bar s^\conic=\1_n+\mA^\top y^\conic=\1_n-s^\conic$; since $\bar s^\conic\ge\0$, this gives $s^\conic\le\1_n$. Therefore, we get
\begin{align*}
\norm{\frac{(x_{\Dc},\bar x_{\Dc})\circ(s^\conic_{\Dc},\bar s^\conic_{\Dc})}{\mu}-\1_{2|\Dc|}}_2
&=\norm{\big((s^\conic_{\Dc}\circ w^\conic_{\Dc},\bar s^\conic_{\Dc}\circ\bar w^\conic_{\Dc})-\1_{2|\Dc|}\big)+\Big(\tfrac{d_{\Dc}\circ s^\conic_{\Dc}}{\mu},\0_{\Dc}\Big)}_2 \\
&\le \beta^\conic + \beta^\conic = 2\beta^\conic\, ,
\end{align*}
where the last inequality uses \eqref{eq:conic-cert-N} and the definition of $\mu$. Therefore, using \Cref{lem:choice-of-mu},
\begin{equation}
    \left((x_{\Dc},\bar x_{\Dc}),(s^\conic_{\Dc},\bar s^\conic_{\Dc})\right)\in\cal{N}^2\!\big(\tfrac{2\beta^\conic}{1-2\beta^\conic}\big)\subseteq\cal{N}^2(\beta^\feas).
\end{equation}
The next lemma is the adaptation of \Cref{lem:PMax-recover}, also including the analytic center certificate as in \eqref{eq:Pmax-cert}. The proof follows the same lines.
\begin{lemma}\label{lem:PMax-recover-A}
 Let $(x^\pf_{\Dc},\bar x^\pf_{\Dc})$ and $(y^\pf,s^\pf_{\Dc},\bar s^\pf_{\Dc})$ be the primal and dual optimal solutions to \eqref{eq:big-M-sys-3-A} obtained via \Cref{thm:SLC-bound}, along with the dual analytic center certificate $u'\in\R^{m}$, $v'=\mH^\top u'$.\footnote{A prime marks the certificate of the reduced system; its lift below carries the stage superscript.}
Let
$x^\conic,s^\conic\in \R^n$, $y^\conic\in\R^m$ be the maximum support primal and dual conic solutions from Stage~\rom{1} such that $\mA x^\conic=\0_m$, $x^\conic\ge \0_n$, $s^\conic=-\mA^\top y^\conic$, $s^\conic\ge \0_n$ with $\supp(x^\conic)=\Pc$ and  $\supp(s^\conic)=\Dc$.
 Let
\[
    z\coloneqq \mA_{\Pc}^+(b-\mA_{\Dc} x^\pf_{\Dc})\, .
\]
Let us pick $\alpha\ge 0$ such that $z+\alpha x^\conic_{\Pc}>\0$, and let
\[
\begin{aligned}
    x^\pf&\coloneqq \left(x^\pf_{\Dc}, x^\pf_{\Pc}\right) \coloneqq \left(x^\pf_{\Dc},
   z\right)+\alpha x^\conic\, ,
    \\
     \tilde y&\coloneqq \mPi y^\pf\, ,\quad\mbox{and}\quad
    \tilde s\coloneqq -\mA^\top \tilde y\, ,\\
 u^\pf&\coloneqq \mPi u'\, ,\quad\mbox{and}\quad
     v^\pf\coloneqq \mA^\top u^\pf\, .
\end{aligned}
\]
Then, $\tilde s\ge \0_n$, $\tilde s_{\Pc}=\0_{\Pc}$, and either
\begin{enumerate}
    \item $\pr{b}{\tilde y}>0$, or
    \item $\pr{b}{\tilde y}=0$ and $\mA x^\pf=b$, $x^\pf\ge \0$.
\end{enumerate}
In the latter case, $\supp(x^\pf)\cup \supp(\tilde s)=[n]$. Further, we have $v^\pf_{\Pc}=\0_{\Pc}$, and for $P'\coloneqq\supp(x^\pf_{\Dc})$, we have
\[
    \|x^\pf_{P'}\circ v^\pf_{P'}-\1_{P'}\|_2\le \beta^\feas\, .
\]
\end{lemma}
\begin{proof}
The argument parallels that of \Cref{lem:PMax-recover}, with the simplification that the cost on $x_{\Dc}$ in \eqref{eq:big-M-sys-3-A} is $\0_{\Dc}$; in addition we track the analytic center certificate.

Since the cost of $x_{\Dc}$ in \eqref{eq:big-M-sys-3-A} is $\0_{\Dc}$, dual feasibility reads $\mH^\top y^\pf+s^\pf_{\Dc}=\0_{\Dc}$. Substituting $\mH=\mPi \mA_{\Dc}$ and using that $\mPi$ is the orthogonal projection onto $\ker(\mA_{\Pc}^\top)$ (\Cref{lem:pseudoinverse-prop}), so that $\mPi^\top=\mPi$, we get
\[
    \mA_{\Dc}^\top \tilde y=\mA_{\Dc}^\top \mPi y^\pf=\mH^\top y^\pf=-s^\pf_{\Dc}\, ,
\]
that is, $\tilde s_{\Dc}=s^\pf_{\Dc}\ge \0_{\Dc}$. Moreover, $\mA_{\Pc}^\top \mPi=(\mPi \mA_{\Pc})^\top=\0$ gives $\tilde s_{\Pc}=-\mA_{\Pc}^\top \mPi y^\pf=\0_{\Pc}$. In summary, $\tilde s\ge \0_n$ and $\tilde s_{\Pc}=\0_{\Pc}$.

For the dichotomy, note that $\pr{b}{\tilde y}=\pr{b}{\mPi y^\pf}=\pr{\mPi b}{y^\pf}=\pr{h}{y^\pf}$, the optimal value of \eqref{eq:big-M-sys-3-A}, which is nonnegative since the primal minimizes $\pr{\1_{\Dc}}{\bar x_{\Dc}}\ge 0$. If $\pr{b}{\tilde y}>0$, then $(\tilde y,\tilde s)$ is an infeasibility certificate for $\mA x=b$, $x\ge \0_n$, returned as the output of $\PMaxA(\mA,b)$.

In the remaining case $\pr{b}{\tilde y}=\pr{h}{y^\pf}=0$, the primal optimum is $0$, hence $\bar x^\pf_{\Dc}=\0_{\Dc}$ and $\mH x^\pf_{\Dc}=h$. By \Cref{lem:lift}, $\mA_{\Dc} x^\pf_{\Dc}+\mA_{\Pc} z=b$, and since $\mA x^\conic=\0_m$ we obtain $\mA x^\pf=b$. Next, we have $x^\pf\geq \0_n$ because $x^\pf_{\Dc}\geq \0_{\Dc}$, $\supp(x^\conic)=\Pc$ and $\alpha$ is chosen so that $x^\pf_{\Pc}=z+\alpha x^\conic_{\Pc}>\0_{\Pc}$. Finally, $\supp(x^\pf)\cup \supp(\tilde s)=[n]$ follows as in \Cref{lem:PMax-recover}: \eqref{eq:big-M-sys-3-A} is solved to strictly complementary optima, so $\supp(x^\pf_{\Dc})$ and $\supp(s^\pf_{\Dc})$ partition $\Dc$, and with $\tilde s_{\Dc}=s^\pf_{\Dc}$ and $\supp(x^\pf)=\Pc\cup \supp(x^\pf_{\Dc})$ the claim follows.

It remains to verify the analytic center certificate. Recall that $u^\pf=\mPi u'$ and $v^\pf=\mA^\top u^\pf$. Since $\mA_{\Pc}^\top\mPi=\0$, we get $v^\pf_{\Pc}=\mA_{\Pc}^\top \mPi u'=\0_{\Pc}$. For the $\Dc$-coordinates, using $\mPi^\top=\mPi$,
\[
    v^\pf_{\Dc}=\mA_{\Dc}^\top \mPi u'=(\mPi \mA_{\Dc})^\top u'=\mH^\top u'=v'\, .
\]
By \Cref{thm:SLC-bound} applied to \eqref{eq:big-M-sys-3-A}, the certificate $v'=\mH^\top u'$ satisfies $\|(x^\pf_{P'}\circ v'_{P'})-\1_{P'}\|_2\le \beta^\feas$ on the primal optimal support $P'=\supp(x^\pf_{\Dc})$. Since $v^\pf_{P'}=v'_{P'}$, we conclude
\[
    \|x^\pf_{P'}\circ v^\pf_{P'}-\1_{P'}\|_2\le \beta^\feas\, ,
\]
as required.
\end{proof}

\subsection{Stage \rom{3}: Optimization}\label{sec:ac-stage-III}
We start by calling $\PMaxA(\mA,b)$ and $\DMaxA(\mA,c)$. Assume that both calls return feasible outputs. Let $x^\pf$ be the maximum support primal solution returned by $\PMaxA(\mA,b)$ with support $\Pf=\supp(x^\pf)$ and dual certificate $(\tilde y,\tilde s)$, and let $(y^\df,s^\df)$ be the maximum support dual solution returned by $\DMaxA(\mA,c)$ with $\Df=\supp(s^\df)$ and primal certificate $\tilde x$. Further, we consider the analytic center certificates $(u^\pf,v^\pf)$ and $w^\df$ as in \eqref{eq:Pmax-cert} and \eqref{eq:Dmax-cert}.
We define the index sets
\[
    R\coloneqq \Pf\cap \Df\, ,\quad I\coloneqq \Pf\setminus \Df\, ,\quad J\coloneqq \Df\setminus \Pf\, ,\quad P'\coloneqq \Pf\cap \Dc\, ,\quad D'\coloneqq \Df\cap \Pc\, .
\]
Note that $P' \cap D' = \emptyset$ and
\[
    R=P'\cup D'\, .
\]
{We may assume that $R\neq \emptyset$, as otherwise $x^{\pf}$ and $(y^{\df}, s^{\df})$ are maximum support primal and dual optimal solutions to \ref{LP_primal_dual}.} %

Similarly to \Cref{sec:big-M-alg}, we  remove all coordinates in $J$ and project out the coordinates in $I$ by letting
    \[
\mPi\coloneqq \mI_m - \mA_{I} \mA_{I}^+\,,\quad
    \mH\coloneqq \mPi \mA_{R}\,,\quad\mbox{and}\quad h\coloneqq \mPi b\, ,
\]
{with the convention that $\mPi \coloneqq \mI_m$ if $I = \emptyset$}.
 We now work directly with the reduced LP
\begin{equation}
\label{eq:big-M-sys-4-A}
\begin{aligned}
\min \;& \pr{s^\df_R}{x_R} \\
\mH x_R & = h  \\
x_R & \geq \0_R \, ,
\end{aligned}
\qquad \qquad
\begin{aligned}
\max \; \pr{h}{y}\\
\mH^\top y +s_R & = s_R^\df \\
s_R & \geq \0_R \, .
\end{aligned}
\end{equation}
We can initialize this system using the outputs of $\PMaxA(\mA,b)$ and $\DMaxA(\mA,c)$. Namely, we set
 \[
\begin{aligned}
\mu&\coloneqq \frac{\|x^\pf_R\circ s^\df_R\|_1}{\beta^\feas}\, ,\\
    x_R&\coloneqq x^\pf_R+\mu (\0_{P'},w^\df_{D'}) \, \\
    y&\coloneqq -\mu u^\pf\, ,\quad\mbox{and}\quad s_R\coloneqq s^\df_R - \mH^\top y =  s^\df_R+\mu (v^\pf_{P'},v^\pf_{D'}) = s^\df_R+\mu (v^\pf_{P'},\0_{D'}) \, .
\end{aligned}
\]
Here the first equality uses $\mH^\top u^\pf=\mA_R^\top u^\pf=(v^\pf_{P'},v^\pf_{D'})$: indeed $\mA_I^\top u^\pf=v^\pf_I=\0$ (as $I\subseteq\Pc$ and $v^\pf_{\Pc}=\0$), so $\mPi u^\pf=u^\pf$ and $\mH^\top u^\pf=\mA_R^\top\mPi u^\pf=\mA_R^\top u^\pf$; the last equality uses $v^\pf_{D'}=\0_{D'}$ since $D'\subseteq\Pc$.
To verify the primal feasibility of $x_R$, observe that applying $\mPi$ to $\mA x^{\pf} = b$ yields $\mH x_R^\pf=h$ because $x^{\pf}_J = \0_J$.
Similarly, applying $\mPi$ to $\mA w^{\df} = \0_m$ yields $\mH_{D'} w^{\df}_{D'} = \0_m$ because $w^{\df}_{\Dc} = \0_{\Dc}$.
Finally, $x^{\pf}_R > \0_R$ and $w^{\df}_{D'}>\0_{D'}$, where the latter is due to \eqref{eq:Dmax-cert}.

These solutions satisfy
\begin{equation}
    \begin{aligned}
    \left\|\frac{x_R \circ s_R}{\mu} - \1_{R}\right\|_2 &=
    \left\|\frac{x^\pf_R \circ s^\df_R}{\mu}+ (x^\pf_R\circ (v^\pf_{P'},\0_{D'}))-(\1_{P'},\0_{D'}) + ((\0_{P'},w^\df_{D'})\circ s^\df_R)-(\0_{P'},\1_{D'}) \right\|_2\\
    &\le \left\|\frac{x^\pf_R \circ s^\df_R}{\mu}\right\|_2+ \left\|(x^\pf_{P'}\circ v^\pf_{P'})-\1_{P'}\right\|_2 + \left\|(w^\df_{D'}\circ s^\df_{D'})-\1_{D'}\right\|_2 \\
     &\le 3\beta^\feas
    \end{aligned}
\end{equation}
where the last inequality uses the definition of $\mu$ and the bounds from \eqref{eq:Pmax-cert} and \eqref{eq:Dmax-cert}.
From \Cref{lem:choice-of-mu}, and the choice of $\beta^\feas$ and $\beta^\opt$, we get that
\begin{equation}
    \left(x_R,s_R\right) \in \cal{N}^2\left(\frac{3\beta^\feas}{1 - 3 \beta^\feas}\right) \subseteq \cal{N}^2({\beta^\opt})\, .
\end{equation}
We solve the system \eqref{eq:big-M-sys-4-A}
with the IPM solver as in \Cref{thm:SLC-bound}, for neighborhood parameter $\beta=\beta^\opt$.
We now adapt \Cref{lem:opt-recover} to convert these to optimal primal and dual solutions to the original system \ref{LP_primal_dual}. The proof follows the same lines; we add it for completeness.

\begin{lemma}\label{lem:opt-recover-A}
 Let $x^\opt_R$ and $(y^\opt,s^\opt_R)$ be the primal and dual optimal solutions to \eqref{eq:big-M-sys-4-A} obtained via \Cref{thm:SLC-bound}. Let $(\tilde y,\tilde s)$ be the certificate returned by $\PMaxA(\mA,b)$ and $\tilde x$ the certificate returned by $\DMaxA(\mA,c)$ such that
\[
\begin{aligned}
    \tilde s&=-\mA^\top \tilde y\, ,\, \pr{b}{\tilde y}=0\, ,\,\tilde s\ge \0_n\, ,\, \supp(\tilde s)=J\, ,\quad\mbox{and}\quad\\
     \mA\tilde x&=\0_m\,,\, \pr{c}{\tilde x}=0\,,\,\tilde x\ge 0\,,\,\supp(\tilde x)=I\, .
     \end{aligned}
\]
Let us define
\[
    z\coloneqq \mA_I^+(b-\mA_R x^\opt_R)\, ,
\]
and pick $\alpha,\lambda\ge 0$ large enough such that $z+\alpha\tilde x_{I}>0$ and $\mA_J^\top (\mPi y^\opt+y^\df + \lambda \tilde y) < c_J$. Let
\[
\begin{aligned}
    x^\star&=(x^\star_R,x^\star_{I},x^\star_J)\coloneqq \left(x^\opt_R,z,\0_J\right)+\alpha \tilde x\, ,
    \\
    y^\star&\coloneqq \mPi y^\opt+y^\df+\lambda \tilde y\, ,\quad\mbox{and}\quad
    s^\star\coloneqq c-\mA^\top  y^\star\, .
\end{aligned}
\]
Then, $x^\star$ and $(y^\star,s^\star)$ are maximum support primal and dual optimal solutions to \ref{LP_primal_dual}.
\end{lemma}
\begin{proof}
Since \eqref{eq:big-M-sys-4-A} is precisely $\LP(\mH,h,s^\df_R)$, $x^\opt_R$ and $(y^\opt,s^\opt_R)$ are primal and dual optimal solutions to it. We show that their modifications $x^\star$ and $(y^\star,s^\star)$ give primal and dual optima on the original system \ref{LP_primal_dual}. The supports $\supp(\tilde x)=I$ and $\supp(\tilde s)=J$ enable the selection of suitable $\alpha,\lambda\ge 0$.

On the primal side, \Cref{lem:lift} and $\mA\tilde x=\0_m$ guarantee that $\mA x^\star=b$. We have $x^\star\ge \0_n$ by $x^\opt_R\ge \0_R$ and by the choice of $\alpha$. Thus, $x^\star$ is primal feasible.

Turning to the dual side, on the coordinates in $R$ we get
\[
  \mH^\top y^\opt+s^\opt_R=s_R^\df\,,\quad   \mA_R^\top y^\df+s_R^\df=c_R\, ,\quad \mbox{and}\quad \mA^\top_R \tilde y=\0_R\, .
\]
Since $\mH^\top y^\opt=\mA_R^\top \mPi y^\opt$, we get that
$s^\star_R= c_R - \mA_R^\top y^\star = s_R^\df  + (s^\opt_R - s_R^\df) - \lambda \mA_R^\top \tilde y = s^\opt_R \ge \0_R$. On the coordinates in $I$ we have
\[
   \mA_I^\top \mPi y^\opt=\0_I\,\,,\quad   \mA_I^\top y^\df=c_I\, ,\quad \mbox{and}\quad \mA^\top_I \tilde y=\0_I\, ,
\]
where the first equation holds because $\mPi$ is an orthogonal projection matrix onto $\ker(\mA_I^\top)$. Hence, $s^\star_I= \0_I$. Finally, $s^\star_J>\0_J$ by the choice of $\lambda$: writing $s^\star_J=c_J-\mA_J^\top y^\star=c_J-\mA_J^\top(\mPi y^\opt+y^\df)+\lambda\tilde s_J$ and using $\tilde s_J>\0_J$ (as $\supp(\tilde s)=J$), choosing $\lambda$ large enough makes $s^\star_J>\0_J$.

To show optimality, we show complementarity. We note that $\supp(x^\star)=\supp(x^\opt_R)\cup I$ and $\supp(s^\star)=\supp(s^\opt_R)\cup J$. We have that $x^\star_R = x^\opt_R$ and $s^\star_R = s^\opt_R$, which have complementary support as optimal solutions to \eqref{eq:big-M-sys-4-A}. Since \eqref{eq:big-M-sys-4-A} is solved to strictly complementary optima, $\supp(x^\opt_R)$ and $\supp(s^\opt_R)$ partition $R$; together with $I\subseteq \supp(x^\star)$ and $J\subseteq \supp(s^\star)$ this gives $\supp(x^\star)\cup\supp(s^\star)=[n]$ with disjoint supports. Hence, $x^\star$ and $(y^\star,s^\star)$ are maximum support primal and dual optimal solutions.
\end{proof}

\subsection{Proof of Theorem~\ref{thm:main-1} }\label{sec:main-proof}
We are ready to prove \Cref{thm:main-1}, restated here for convenience. The remaining ingredient, \Cref{lem:main-round}, will be proved in \Cref{sec:bit-complexity}.

\mainthm*
\begin{proof}
We use the three-stage initialization as described in this section. Namely, we call the IPM in \Cref{thm:SLC-bound} on the four LPs
\eqref{eq:big-M-sys-conic}, \eqref{eq:big-M-sys-2-A}, \eqref{eq:big-M-sys-3-A}, and \eqref{eq:big-M-sys-4-A}.
Note that the matrices $\mA_{\Pc}$ and $\mH$ may not have full row rank, even though $\mA$ has full row rank.
Nevertheless, we can remove the redundant rows before running the IPM in \Cref{thm:SLC-bound}.
This is because any dual optimal solution of the resulting system can be trivially lifted to a dual optimal solution of the original system by filling $y$ with additional zeros.

As already noted in \Cref{sec:proof-weaker}, the straight line complexity of the variables in \eqref{eq:big-M-sys-conic} are the same as those for \eqref{sys:2}.
Next, consider the system \eqref{eq:big-M-sys-2-A} solved for dual feasibility.
To obtain this system from \eqref{sys:3}, we first replace the right-hand side vector $\1_n$ by $2\cdot \1_n$, which does not change the straight line complexity of the variables $(x, \hat x)$. Since every feasible solution satisfies $x_{\Dc} = \0_{\Dc}$, by \Cref{lem:project-delete}{\em (i)} we can delete the variables $x_{\Dc}$ without changing the max central path of $(x_{\Pc}, \hat x$). As the variables $\hat x_{\Dc}$ do not interact with $(x_{\Pc}, \hat x_{\Pc})$, we can delete them to arrive at \eqref{eq:big-M-sys-2-A} without affecting the max central path of $(x_{\Pc}, \hat x_{\Pc})$.

Similarly, we show the system \eqref{eq:big-M-sys-3-A} solved for primal feasibility is obtained from \eqref{sys:4}. Since every dual feasible solution satisfies $s_{\Pc} = \0_{\Pc}$, by \Cref{lem:project-delete}{\em(ii)} we can project out the variables $x_{\Pc}$ without changing the max central path of $(x_{\Dc}, \bar x)$. As $\mPi (-\mA_{\Pc}) = \0$, we can delete the variables $\bar x_{\Pc}$ to arrive at \eqref{eq:big-M-sys-3-A} without affecting the max central path of $(x_{\Dc}, \bar x_{\Dc})$.

Finally, we show how the system \eqref{eq:big-M-sys-4-A} solved in the optimization stage is obtained from \eqref{sys:1}.
Recall that $I = [n]\setminus \Df$ and $J = [n]\setminus \Pf$, that is, every feasible primal solution satisfies $x_J = \0_J$ and every feasible dual solution satisfies  $s_I = \0_I$. Thus, by \Cref{lem:project-delete}{\em(iii)}, the variables in \eqref{eq:big-M-sys-4-A} have the same straight line complexities as in \ref{LP_primal_dual}.

We now give a PSPACE implementation of the SLLS IPM. We use that the IPM is stepwise strongly polynomial, stays in the $\cN(\beta)$ neighborhood, and shrinks the optimality gap by a factor of at most $1-\beta/(2\sqrt{n})$ in every iteration, as stated in \Cref{thm:SLC-bound}.
For each iterate $(x,s)$, we call \Cref{lem:main-round} to round it to another iterate $(\tilde x,\tilde s)\in \clN(3\beta)$ and $\gap(\tilde x,\tilde s)\le (1+\beta/(4n^2))\gap(x,s)$ such that $\size(\tilde x,\tilde s)$ is polynomially bounded. Thus, every iterate of the SLLS IPM, followed by the rounding step, still decreases the gap by a factor of
\begin{equation}
\label{eq:rounded-progress}
\Bigl(1-\frac{\beta}{2\sqrt n}\Bigr)\Bigl(1+\frac{\beta}{4n^2}\Bigr)\le 1-\frac{\beta}{2\sqrt n}+\frac{\beta}{4n^2}\le 1-\frac{\beta}{4\sqrt n}\, .\
\end{equation}
After each such rounding, unless $\gap(\tilde x, \tilde s) = 0$, we can move back to the narrower neighborhood ${\cal N^2}(\beta)$ by at most two corrector steps (see \Cref{lem:corrector}).
If $\gap(\tilde x,\tilde s)=0$, then $(\tilde x,\tilde s)$ is a pair of primal and dual optimal solutions, and we terminate the current IPM call, returning $(\tilde x,\tilde s)$ together with the certificates $(\tilde v,\tilde w)$ from \Cref{lem:main-round}. These satisfy the output conditions of \Cref{thm:SLC-bound} with $3\beta$ in place of $\beta$. Otherwise, the IPM call terminates as in \Cref{thm:SLC-bound}, with certificates satisfying the output conditions for $\beta$ itself. Hence, in the PSPACE implementation, the certificate bounds \eqref{eq:conic-cert-B}, \eqref{eq:conic-cert-N}, \eqref{eq:Pmax-cert} and \eqref{eq:Dmax-cert} hold with $3\beta^\conic$ and $3\beta^\feas$ in place of $\beta^\conic$ and $\beta^\feas$, respectively.
The analysis in \Cref{sec:analytic-init} remains valid with these weaker bounds. The bounds $\|x^\conic_{\Pc}\|_\infty\le 2$ and $s^\conic\le\1_n$ used in Stage~\rom{2} follow from the structure of \eqref{eq:big-M-sys-conic}. In both initializations of Stage~\rom{2}, the bound $2\beta^\conic$ becomes $3\beta^\conic+\beta^\conic=4\beta^\conic$, and by choice of the parameters in \eqref{def:beta-parameters} we have $4\beta^\conic/(1-4\beta^\conic)\le\beta^\feas$. As in the original argument, this inequality can be used for \Cref{lem:choice-of-mu}, which converts the centrality error with respect to an arbitrary gap parameter into membership in $\cN(\beta^\feas)$, where we measure with respect to the normalized duality gap of the starting point. \Cref{lem:PMax-recover-A} passes the certificate bound from \Cref{thm:SLC-bound} through unchanged, so \eqref{eq:Pmax-cert} holds with $3\beta^\feas$. In the initialization of Stage~\rom{3}, the bound $3\beta^\feas$ becomes $\beta^\feas+3\beta^\feas+3\beta^\feas=7\beta^\feas$, and $7\beta^\feas/(1-7\beta^\feas) <2^{-5}<\beta^\opt$, which again allows us to apply \Cref{lem:choice-of-mu}. Finally, \Cref{lem:main-round} applies to all stages as $\beta^\opt=1/16$.

We still need to bound the encoding lengths across the four IPM calls. For a call on a system with data $(\mA',b',c')$, every iterate has size $\poly(n)\size(\mA',b',c')$ by \Cref{lem:main-round}, and the output, including the certificates, is either returned by \Cref{lem:main-round} or computed from the last iterate in strongly polynomial time, so it also has size $\poly(n)\size(\mA',b',c')$. The data $(\mA',b',c')$ of each system, its starting point, and the final output are obtained from $(\mA,b,c)$ and the outputs of the previous calls. By \Cref{lem:linalg}, all computations in \Cref{sec:analytic-init} outside the IPM calls are strongly polynomial in the Turing model. This includes the choice of shift scalars such as $\alpha$ and $\lambda$ in \Cref{lem:opt-recover-A}. As there are only four calls, all numbers throughout the algorithm have size $\poly(n)\size(\mA,b,c)$.

What remains to be argued is that the SLLS IPM with rounding still satisfies the same iteration bound as \Cref{thm:SLC-bound}. Note that the rounding step pushes each iterate only slightly up the path, which intuitively should not matter.

To formally justify this, we will explain why the analysis used to prove \Cref{thm:SLC-bound} applies to the iterate sequence produced by the rounded algorithm. As it is relatively direct, we simply highlight which statements of the proof of~\Cref{thm:SLC-bound} in \cite{ADLNV2022} need to be verified for the rounded IPM.

At a high level, the authors of \cite{ADLNV2022} decompose the central path into a number of $\eta$-\emph{polarized segments} proportional to SLC. Each of these requires $O(\sqrt{n}/\beta \ln(n/(\eta \beta)))$ iterations to traverse in an amortized sense.
Let $z^{0},\dots,z^{T} \in \cN(\beta)$ denote the iterates of the IPM. The key statement is \cite[Lemma 7.5]{ADLNV2022}, which states that for $D \coloneqq C_{\rm it}\sqrt{n}/\beta\ln(n/(\eta \beta))$, where $C_{\rm it} \geq 1$ is an absolute constant, and for any two iterates $z^{i}, z^{j}$ with $j \geq D+i$, either $z^{j}$ has passed the end of $z^{i}$'s polarized segment, or $z^{j}$'s lifting operator within this segment has one fewer large singular value. In \cite[Section 8.2]{ADLNV2022}, it is shown that any sequence of iterates that satisfies the above bound also satisfies the SLC bound of \Cref{thm:SLC-bound}. Thus, it suffices to verify the lemma. Lemma 7.5 relies only on the following three properties: that the iterates stay in the $\cN(\beta)$ neighborhood, that the IPM reduces the gap by at least a $1-\Omega(1/\sqrt{n})$ factor at each iteration, and finally that the IPM takes steps at least as large as the SLLS step whenever this is possible.
Since the rounding steps lose only a factor in the gap that can be recovered by a single short step (by \eqref{eq:rounded-progress}), and the iterates return to $\cN(\beta)$ after the corrector steps, a constant multiple of the iterations suffices for the rounded IPM, i.e., the lemma holds with $2C_{\rm it}$.

Alternatively, one could use the $\ell_2$-Trust Region IPM algorithm in the more recent paper \cite{Dadush2026}. Theorem 1.3 in this paper asserts that, for any $\mu_0\ge \mu_1\ge 0$,  starting from any point $z_0\in \cN(\beta)$ with $\gap(z_0)\in [\mu_1,\mu_0]$, the algorithm reaches an iterate $z_1$ with $\gap(z_1)\le \mu_1$ in $O\bigl(\SLC(\cN(\beta),\mu_1,\mu_0)\bigr)$ iterations, each of which takes strongly polynomial time.
\end{proof}

\section{A strongly polynomial rounding procedure}\label{sec:bit-complexity}
In this section, we prove the main rounding statement \Cref{lem:main-round} used for the PSPACE implementation in \Cref{thm:main-1}.
Throughout, all numbers are encoded rationals in the sense of \Cref{def:size}, and we use the conventions of \Cref{sec:comp-mod} for how to compute with them.

Addition and multiplication of encoded rationals are associative and commutative: for $(p_i,q_i)\in\eQ$, $i\in[k]$, we have
\begin{equation}
\label{eq:sum-prod-eQ}
\sum_i (p_i, q_i) = \Bigl(\sum_{i=1}^k p_i\prod_{j\ne i}q_j\, ,\ \prod_{i=1}^k q_i\Bigr)\, ,
\qquad
\prod_{i=1}^k (p_i, q_i)=\Bigl(\prod_{i=1}^k p_i\, ,\ \prod_{i=1}^k q_i\Bigr)\, ,
\end{equation}
independently of the order of the terms and of the bracketing. Hence, iterated sums and products are well-defined elements of $\eQ$.

Distributivity, however, does not hold. For instance,
\[
\bigl((1,1)+(1,1)\bigr)\cdot(1,2)=(2,1)\cdot(1,2)=(2,2)\, ,
\qquad\text{whereas}\qquad
(1,1)\cdot(1,2)+(1,1)\cdot(1,2)=(1,2)+(1,2)=(4,4)\, ,
\]
two different representations of the same value $1$. Thus, the size of a computed number depends on how it is computed, and our size bounds always refer to the specific computations performed by our algorithms.

We restate \Cref{lem:main-round} for convenience.
\roundlemma*

\Cref{lem:main-round} is not specific to the SLLS IPM. As indicated in the introduction, the rounding step can be interleaved with the iterations of any log-barrier path-following IPM that satisfies the following two requirements. First, the method maintains rational iterates in the $\ell_2$-neighborhood $\cN(\beta)$ for some $\beta\le 1/16$, and each iteration can be implemented in strongly polynomial time in the Turing model; that is, given an iterate $(x,s)\in\Q^{2n}$, the next iterate is computed using $\mathrm{poly}(n)$ arithmetic operations and comparisons on numbers of size $\mathrm{poly}(n)\size(\mA,b,c,x,s)$. Second, the method can be restarted from any point in $\clN(3\beta)$ whose gap is larger by a factor at most $(1+\beta/(4n^2))$, requiring that within $\mathrm{poly}(n)$ iterations, it returns to $\cN(\beta)$ without losing its progress guarantees, or terminates if the point is optimal. Under these two requirements, applying the rounding step after every iteration keeps the encoding length of all iterates polynomially bounded in $\size(\mA,b,c)$. Hence, if the total number of iterations of the original method is strongly polynomial, the modified method is a strongly polynomial algorithm in the Turing model.
The rounding steps could also be used in weakly polynomial algorithms. However, if the running time is allowed to depend on the total encoding length $L=\size(\mA,b,c)$, it may be more efficient to use other methods such as \cite{ghadiri2023bit}.
It is interesting to note that the new iterate $(\tilde x,\tilde s)$ could have much smaller duality gap than $(x,s)$: the rounding algorithm itself may take long steps down the central path. In particular, using the rounding it turns out that the number of iterates is at most $\poly(n)$ times the number of vertices of the primal-dual polytope $\cal P \times \cal D$, which is at most $2^n \cdot 2^n = 2^{2n}$.

\medskip
We use the
 following simple bounds from \cite{gls}. As in \cite{gls}, these bound the sizes of values, that is, of lowest-terms representations. \Cref{rem:size-bounds-encoded} below discusses which of them hold for encoded rationals.
\begin{lemma}\label{lem:size-bounds}
The following bounds hold:
\begin{enumerate}[label=(\roman*), ref=(\roman*)]
    \item\label{it:sums} For $\alpha_1,\ldots,\alpha_k\in\Q$ we have $\size\bigl(\sum_{i=1}^k\alpha_i\bigr)\le 2\sum_{i=1}^k\size(\alpha_i)$ and $\size\bigl(\prod_{i=1}^k\alpha_i\bigr)\le\sum_{i=1}^k\size(\alpha_i)$.
    \item\label{it:inverse} For an invertible matrix $\mM\in\Q^{n\times n}$ we have $\size(\mM^{-1})\le 4n^2\size(\mM)$.
    \item\label{it:vertex} For $\mA\in\Q^{m\times n}$, $b\in\Q^m$, $c\in \Q^n$, let  $\cal P = \{x\in\R^n \mid \mA x = b, x \ge \0_n\}$ and $\cal D = \{s\in\R^n \mid \exists y\in\R^m, \mA^\top y+s = c, s \ge \0_n\}$ denote the primal and dual feasible regions. Then every vertex of $\cal P$ has size $O(n^3\size(\mA,b))$, and every extreme ray of $\cal P$ has a generator of size $O(n^3\size(\mA,b))$; the same bounds hold for $\cal D$ with $\size(\mA,c)$ in place of $\size(\mA,b)$.
\end{enumerate}
\end{lemma}

\begin{remark}[The same bounds on $\eQ$]\label{rem:size-bounds-encoded}
Part \ref{it:sums} also holds for encoded rationals. By \eqref{eq:sum-prod-eQ}, for $(p_1,q_1),\ldots,(p_k,q_k)\in\eQ$ we have
\[
\size\Bigl(\sum_{i=1}^k (p_i,q_i)\Bigr)\le 2\sum_{i=1}^k\size(p_i,q_i)\, ,\qquad
\size\Bigl(\prod_{i=1}^k (p_i,q_i)\Bigr)\le \sum_{i=1}^k\size(p_i,q_i)\, ,
\]
and likewise entrywise for sums and products of encoded vectors and matrices.

Parts \ref{it:inverse} and \ref{it:vertex}, however, bound the lowest-terms representations of $\mM^{-1}$ and of the vertices and extreme rays of $\cal P$, and our algorithms do not compute these representations. Instead, we bound the sizes of the representations returned by our algorithms. Given $\mM$, the procedure of \Cref{lem:linalg}\ref{it:linalg-inv} returns $\tilde\mM$ with $\val(\tilde\mM)=\mM^{-1}$, and given $\mA$, $b$ and $x\in\cal P$, the algorithm of \Cref{lem:minkowski-weyl} returns vectors $\tilde v,\tilde r\in\eQ^n$ with $\val(\tilde v)=v$ and $\val(\tilde r)=r$ for vertices $v$ and extreme rays $r$ of $\cal P$. Bounds on their sizes are given in \Cref{lem:linalg,lem:minkowski-weyl}.
\end{remark}

We next give a simple rounding step for rational numbers.
Suppose that, given $\alpha\in \Qplusplus$, we can find a number $r\in \Qplusplus$ of low bit complexity that agrees with $\alpha$ up to a factor of $2^{O(n)}$. The rounding step then `anchors' $\alpha$ to $r$: it changes the value of $\alpha$ only slightly, but brings its bit complexity down to roughly that of $r$.
\begin{lemma}[Anchoring numbers]\label{lem:anchoring}
Let $\alpha,r\in \eQplusplus$, $n, K\in\Zplusplus$ be such that $r/2^n \le \alpha \le 2^n r$. Then, there is an algorithm requiring $O(\log(n+K))$ arithmetic operations that returns a number $\hat \alpha\in \eQ$ with $\alpha\le \hat \alpha\le (1+1/K)\alpha$ and $\size(\hat \alpha)\le \size(r)+O(n+\log K)$.
\end{lemma}
\begin{proof}
Let $m$ be the smallest power of two that is not smaller than $n$. In $O(\log(m)) = O(\log(n))$ arithmetic operations we can compute the numbers $2^{2^0}, 2^{2^1}, 2^{2^2}, \ldots, 2^m$ by repeated squaring. Using binary search on the exponent, in $O(\log n)$ iterations we can find an integer $k\in [-m,m]$ such that $2^k r\le \alpha\le 2^{k+1}r$. Throughout the search, we maintain the current value $2^{k'}r$ for the exponent $k'$ being probed, starting from $k'=0$. As the search interval is halved in each step, consecutive probes differ in the exponent by a power of two that is among our precomputed values; hence, each iteration takes a single multiplication or division, followed by a comparison. Divisions are not an issue here: dividing the current value by $2^{2^j}$ is equivalent to multiplying $\alpha$ by $2^{2^j}$ on the other side of the comparison. Note that all values arising in the search are of the form $r$ times a power of two with exponent in $[-m,m]$, and that the search also returns the value $2^kr$ itself.
Also, by repeated squaring in $O(\log\log K)$ arithmetic operations we can find a value $q\in [0,2\lceil\log_2 K\rceil]$ such that $2^{q/2} \le K \le 2^q$. Then, in $O(q)=O(\log K)$ arithmetic operations we find $p\in [0,2^q]$ such that $\alpha \le r 2^k(1+p/2^q) \le (1+1/K)\alpha$.  Then, $\hat\alpha \coloneqq r2^k(1+p/2^q)$ satisfies the claimed properties.
\end{proof}

We next state the algorithmic version of the Minkowski--Weyl decomposition theorem. The proof follows the standard lines of finding first a vertex $v^{(1)}$ in the support of $x$, and `shooting' a ray from $v^{(1)}$ through $x$. If it hits a face, then we recurse. Otherwise, the ray can be similarly decomposed to a convex combination of extreme rays. See \cite{Schrijver} and Theorem 6.5.11 in \cite{gls}. We need explicit bounds on the number of arithmetic operations and on the encoding lengths of all numbers computed, in the model of \Cref{sec:comp-mod}. We give a self-contained proof for polyhedra in the standard form $\mA x=b$, $x\ge \0_n$, in \Cref{sec:mw-explicit}, using the cubic bound on matrix multiplication.
\begin{lemma}[Minkowski--Weyl decomposition]\label{lem:minkowski-weyl}
For $\mA\in\eQ^{m\times n}$ with $m\le n$ and $b\in\eQ^m$, let  $\cal P = \{x\in\R^n \mid \mA x = b, x \ge \0_n\}$. Assume we are given a point $x\in\cal P\cap\eQ^n$. Then, there is a strongly polynomial algorithm that returns extreme points $v^{(1)},\ldots,v^{(k)}$, extreme rays $r^{(1)},\ldots,r^{(\ell)}$, and multipliers $\lambda_1,\ldots,\lambda_k\ge 0$,  $\eta_1,\ldots,\eta_\ell\ge 0$ such that $\sum_{i=1}^k \lambda_i=1$, $k+\ell\le n+1$, and
\[
x=\sum_{i=1}^k \lambda_i v^{(i)}+\sum_{j=1}^\ell \eta_j r^{(j)}\, .
\]
The algorithm performs $O(n^7)$ arithmetic operations and comparisons and runs in space $\poly(n)\size(\mA,b,x)$. Every vertex returned has size $\poly(n)\size(\mA,b)$, and every extreme ray and every multiplier returned has size $\poly(n)\size(\mA,b,x)$.
\end{lemma}

\subsection{Preliminaries on  interior point methods}\label{sec:ipm-preliminaries}
Before proving \Cref{lem:main-round} in \Cref{sec:rounding-proof}, we
summarize some simple facts needed from the interior-point literature.
The following simple formula allows us to express gaps between different primal and dual solutions.
\begin{proposition}\label{prop:gap-formula} Given $x,x' $ such that $x - x' \in \ker(\mA)$, and $s,s'$ such that $s - s' \in \im(\mA^\top)$, we have that
\[
\pr{x}{s} + \pr{x'}{s'} = \pr{x}{s'} + \pr{x'}{s}.
\]
In particular, if $\pr{x'}{s'}=0$, then
\[
\pr{x}{s} = \pr{x}{s'} + \pr{x'}{s}.
\]
\end{proposition}
\begin{proof}
This follows since $x-x'\in \ker(\mA)$ and $s-s'\in \im(\mA^\top)$ live in orthogonal subspaces.
\end{proof}

 A key property of the central path is \emph{`near monotonicity'},
 formulated in the following lemma, see \cite[Lemma 16]{Vavasis1996}.
 \begin{lemma}\label{lem:near-monotonicity}
For the central path point $z\cp(\mu)=(x\cp(\mu),s\cp(\mu))$ and any $z=(x,s)\in \Primal\times \Dual$, we have
\[
z\le n\left(1+\frac{\gap(z)}{\mu}\right) z\cp(\mu)\, .
\]
Moreover, if  $\gap(z)\le \mu$, then $z\le 2nz\cp(\mu)$.
 \end{lemma}
We formulate a simple linearity property of the duality gap, see e.g. \cite[Proposition 2.29]{ADLNV2022}.
\begin{proposition}[Linearity of duality gap] \label{prop:gap-linear}
  Given $x^{(1)},\dots,x^{(k)} \in \Primal$, $s^{(1)},\dots,s^{(k)} \in \Dual$  and $\lambda \in \Rplus^k$ such that $\sum_{i=1}^k \lambda_i = 1$, we have that
  \begin{equation*}
  \gap\left(\sum_{i=1}^k \lambda_i x^{(i)},\sum_{i=1}^k \lambda_i s^{(i)}\right) = \sum_{i=1}^k \lambda_i \gap(x^{(i)},s^{(i)}) \,.
  \end{equation*}
\end{proposition}

We use the following proximity properties. For the first one, see e.g., \cite[Lemma 5.4]{Gonzaga92} and \cite[Proposition 2.1]{Monteiro2003}.
\begin{proposition}\label{prop:near-central}
  Let $z = (x, s) \in
   {\cal N}^2(\beta)$ for $\beta\in (0,1/4]$ and  $\mu=\gap(z)$, and
   consider the central path point
 $\pz\cp(\mu)=(\px\cp(\mu),\ps\cp(\mu))$. For each $i\in[n]$,
 \[
 \begin{aligned}
 \frac{\px_i}{1+2\beta}\le \frac{1-2\beta}{1-\beta}\mathinner{} \px_i&\le \px\cp_i(\mu)\le
 \frac{\px_i}{1-\beta}\, ,\quad \mbox{and}\\
 \frac{\ps_i}{1+2\beta}\le  \frac{1-2\beta}{1-\beta}\mathinner{} \ps_i&\le \ps\cp_i(\mu)\le
 \frac{\ps_i}{1-\beta}\, .
 \end{aligned}
 \]
 \end{proposition}

The next lemma formalizes the well-known property of standard corrector steps, see e.g., \cite{MTY}.
\begin{lemma}[Corrector step]\label{lem:corrector}
Let $\mA\in\eQ^{m\times n}$ have full row rank, let $\beta\in(0,1/4]$, and let $(x,s)\in \cN(2\beta)\cap\eQ^{2n}$. There is a strongly polynomial algorithm that finds a point $(x',s')\in\cN(\beta)\cap\eQ^{2n}$ with $\gap(x',s')=\gap(x,s)$ and $\size(x',s')=\poly(n)\size(\mA,x,s)$. The algorithm solves a single nonsingular system of $2n+m$ linear equations, whose data are formed from $\mA$, $x$ and $s$.
\end{lemma}

\subsection{The rounding procedure}\label{sec:rounding-proof}
We prove \Cref{lem:main-round} in two steps. First, we show that every iterate $(x,s)$ can be replaced by another feasible pair $(\hat x,\hat s)$ whose normalized gap has bounded bit-complexity. This may involve an unbounded decrease in the gap value, as the algorithm identifies an improving direction along the central path.
In the second step, we replace such a point by another one nearby that has {approximately} the same normalized gap and also bounded bit-complexity.

\begin{lemma}\label{lem:round-gap}
Let $\mA\in \eQ^{m\times n}$, $b\in \eQ^m$, $c\in \eQ^n$, let $\beta\in(0,1/16]$ be as fixed in \Cref{sec:CP}, and
let $(x,s)\in{\cal N^2}(\beta)\cap \eQ^{2n}$. There is a strongly polynomial algorithm that returns a solution $(\hat x,\hat s)\in \clN(2\beta)\cap \eQ^{2n}$ with $\gap(\hat x,\hat s)\le \gap(x,s)$ and a value $\mu^\star\in \eQ$ such that $\gap(\hat x,\hat s)\le \mu^\star\le (1+\beta/(4n^2))\gap(\hat x,\hat s)$, and  $\size(\mu^\star)= \poly(n)\size(\mA,b,c)$.
Moreover, one of the following two outcomes occurs.
\begin{enumerate}[label=(\roman*)]
\item\label{it:round-gap-pos} $\mu^\star>0$, and $(\hat x,\hat s)\in\cN(2\beta)$.
\item\label{it:round-gap-zero} $\mu^\star=0$, so $(\hat x,\hat s)$ is a pair of primal and dual optimal solutions, and the algorithm additionally returns $(\hat v,\hat w)\in\eQ^{2n}$ such that $B\coloneqq\supp(\hat x)$ and $N\coloneqq\supp(\hat s)$ partition $[n]$, $\hat v\in\im(\mA^\top)$, $\hat v_B>\0_B$, $\mA\hat w=\0_m$, $\hat w_N>\0_N$, and $\norm{(\hat x_B\circ\hat v_B,\hat s_N\circ\hat w_N)-\1_n}_2\le 2\beta$; that is, $(\hat x,\hat s)$ and $(\hat v,\hat w)$ satisfy the output conditions of \Cref{thm:SLC-bound} with $2\beta$ in place of $\beta$.
\end{enumerate}
\end{lemma}
\begin{proof}
Let $\mu=\gap(x,s)$; thus, $\pr{x}{s}=n\mu$. Let us construct a Minkowski--Weyl decomposition of $(x,s)\in \Primal\times \Dual$ into vertices and extreme rays of $\Primal\times\Dual$. Note that every vertex is of the form $(v^{(i)},u^{(i)})$, where $v^{(i)}$ is a vertex of $\Primal$ and $u^{(i)}$ is a vertex of $\Dual$. Similarly, every extreme ray of $\Primal \times \Dual$ can be written as $(r^{(j)}, 0)$ or $(0,t^{(j)})$, where the vector $r^{(j)}$ is an extreme ray of $\Primal$ and $t^{(j)}$ is an extreme ray of $\Dual$. The number of vertices and extreme rays in the decomposition is at most $2n+1$.

To apply \Cref{lem:minkowski-weyl} we need $\Primal\times\Dual$ in standard form, and so we describe $\Dual$ without the variable $y$. Let
\[
\mB\coloneqq\det(\mA\mA^\top)\,\mI_n-\mA^\top\adjt(\mA\mA^\top)\,\mA\in\eQ^{n\times n}\, .
\]
As $\mA$ has full row rank, $\mA\mA^\top$ is nonsingular, and $\mB=\det(\mA\mA^\top)\bigl(\mI_n-\mA^\top(\mA\mA^\top)^{-1}\mA\bigr)$ is a nonzero multiple of the orthogonal projection onto $\ker(\mA)$. Hence $\ker(\mB)=\ker(\mA)^\perp=\im(\mA^\top)$, so $s\in\Dual$ if and only if $\mB(c-s)=\0_n$ and $s\ge\0_n$, that is,
\[
\Primal\times\Dual=\left\{(x,s)\in\R^{2n} \;\middle|\; \begin{pmatrix}\mA&\0_{m\times n}\\ \0_{n\times n}&\mB\end{pmatrix}\begin{pmatrix}x\\ s\end{pmatrix}=\begin{pmatrix}b\\ \mB c\end{pmatrix},\ (x,s)\ge\0_{2n}\right\}\, .
\]
This describes the same set, so its vertices and extreme rays are exactly those of $\Primal\times\Dual$. The system does not have full row rank, as $\rank(\mB)=n-m$, but this is not required by \Cref{lem:minkowski-weyl}. The matrix $\mB$ is obtained by computing $\det(\mA\mA^\top)$ and $\adjt(\mA\mA^\top)$ by \Cref{lem:linalg}\ref{it:linalg-inv}, and matrix products, in $O(n^4)$ arithmetic operations in total. As $\size(\mA\mA^\top)=\poly(n)\size(\mA)$ by \Cref{rem:size-bounds-encoded}, \Cref{lem:linalg} and \Cref{rem:size-bounds-encoded} give $\size(\mB)=\poly(n)\size(\mA)$ and $\size(\mB c)=\poly(n)\size(\mA,c)$. The system has $m+n\le 2n$ rows and $2n$ columns. Therefore, \Cref{lem:minkowski-weyl} bounds this decomposition by $O(n^7)$ arithmetic operations and comparisons, performed in space $\poly(n)\bigl(\size(\mA,b,c)+\size(x,s)\bigr)$, and guarantees $\size(v^{(i)},u^{(i)})=\poly(n)\size(\mA,b,c)$. We write the decomposition in the form
\[
(x,s)=\sum_{i=1}^k \lambda_i (v^{(i)},u^{(i)})+(r,t)\, ,
\]
where $\lambda\ge0$, $\sum_{i=1}^k\lambda_i=1$, and $(r,t)$ is a ray of $\Primal\times \Dual$ (i.e., the sum of all extreme ray terms).
Let us partition the vertices according to their objective values:
\[
\begin{aligned}
I^-=\{i\in [k]\,:\, \gap({v^{(i)}},{u^{(i)}})\le \mu\}\,,\quad
I^+=\{i\in [k]\,:\, \gap({v^{(i)}},{u^{(i)}})> \mu\}\, , \\
 \end{aligned}
\]
 We note that $I^-\neq\emptyset$ but $I^+$ might be empty.
Indeed, let $(\bar v,\bar u)\coloneqq\sum_{i=1}^k\lambda_i(v^{(i)},u^{(i)})$. As $\mA r=\0_m$ and $t\in\im(\mA^\top)$, we have $\pr{r}{t}=0$, so nonnegativity and \Cref{prop:gap-linear} give
\[
n\mu=\pr{\bar v+r}{\bar u+t}=\pr{\bar v}{\bar u}+\pr{\bar v}{t}+\pr{r}{\bar u}\ge n\,\gap(\bar v,\bar u)=n\sum_{i=1}^k\lambda_i\gap(v^{(i)},u^{(i)})\, ,
\]
and hence $\gap(v^{(i)},u^{(i)})\le\mu$ for some $i$.
The pairs $(v^{(i)},u^{(i)})$ returned by the decomposition satisfy $\size(v^{(i)},u^{(i)})=\poly(n)\size(\mA,b,c)$ by \Cref{lem:minkowski-weyl}, and $\gap(v^{(i)},u^{(i)})$ is obtained from them by $O(n)$ additions and multiplications, so $\size(\gap(v^{(i)},u^{(i)}))=\poly(n)\size(\mA,b,c)$ by \Cref{rem:size-bounds-encoded}.
Let us define the accuracy parameter
\begin{equation}\label{eq:delta-def}
\delta\coloneqq\frac{\beta}{16n^2}\, .
\end{equation}
 If there is an index
$i\in [k]$ with $\gap(v^{(i)},u^{(i)})\in [\mu\delta^2/(4n),4n\mu/\delta^2]$, then we apply \Cref{lem:anchoring} with $\alpha\coloneqq\mu$, $r\coloneqq\gap(v^{(i)},u^{(i)})$, accuracy parameter $K$ the smallest power of two with $K\ge 1/\delta$, found by $O(\log n)$ doublings, and with $N$ in place of $n$, where $N$ is the smallest power of two with $2^N\ge 4n/\delta^2$. Then $N=O(\log n)$, $N$ is found by $O(\log\log n)$ squarings, and $r/2^N\le\alpha\le 2^N r$ by the choice of $i$. The lemma returns $\mu^\star\in\eQ$ with $\mu\le \mu^\star\le (1+\delta)\mu$ in $O(\log(N+K))=O(\log n)$ arithmetic operations, and $\size(\mu^\star)\le\size(r)+O(N+\log K)=\poly(n)\size(\mA,b,c)$ by the bound on $\size(\gap(v^{(i)},u^{(i)}))$ above. Thus, we output $(\hat x,\hat s)=(x,s)$ with the value $\mu^\star$ and outcome \ref{it:round-gap-pos} holds, as $\mu^\star\ge\mu>0$ and $(x,s)\in\cN(\beta)\subseteq\cN(2\beta)$.

For the rest of the proof, we can assume that
\begin{equation}\label{eq:far-away}
\gap(v^{(i)},u^{(i)})<\frac{\delta^2\mu}{4n}\, ,\quad\forall i\in I^-\quad\mbox{and}\quad \gap(v^{(j)},u^{(j)})>\frac{4n\mu}{\delta^2}\, ,\quad \forall j\in I^+\, .
\end{equation}

Let us define
\begin{equation}
\begin{aligned}
\alpha&\coloneqq \sum_{i\in I^+}\lambda_i\, ,\\
(x^-,s^-) &\coloneqq\frac{1}{1-\alpha}\sum_{i\in I^-}\lambda_i (v^{(i)},u^{(i)})\, ,\quad \mu^-\coloneqq \gap(x^-,s^-)\,
\end{aligned}
\end{equation}
Note that $\alpha < 1$ by the same argument that shows $I^- \neq \emptyset$.

\begin{claim}
\label{cl:far-away}
$(x^-,s^-)\le (1+\delta)(x,s)$, and
for every $i\in [n]$, either $x^-_i\ge (1-\delta)x_i$ or $s^-_i\ge (1-\delta)s_i$.
\end{claim}
\begin{claimproof}
We distinguish two cases according to whether $\alpha = 0$.

If $\alpha = 0$, then $(x, s) - (x^-, s^-) = (r,t) \ge (\0_n, \0_n)$. In particular, $x_i \ge x_i^-$ and $s_i \ge s_i^-$ for all $i \in [n]$. Furthermore, as $r$ and $t$ live in orthogonal subspaces, we have that $\pr{r}{t} = 0$ and so for all $i \in [n]$ either $r_i = 0$ or $t_i = 0$, implying that either $x^-_i = x_i$ or $s^-_i = s_i$.

If $\alpha > 0$, we define
\begin{equation}\label{eq:mupm}
\begin{aligned}
(x^+,s^+)&\coloneqq\frac{1}{\alpha}\left(\sum_{i\in I^+}\lambda_i (v^{(i)},u^{(i)})+(r,t)\right)
\, ,\quad \mu^+\coloneqq \gap(x^+,s^+)\, .
\end{aligned}
\end{equation}
Thus, $(x^-,s^-),(x^+,s^+)\in \Primal\times \Dual$, $(x,s)=(1-\alpha)(x^-,s^-)+\alpha (x^+,s^+)$. Note that $\mu^-\le \max_{i\in I^-}\gap(v^{(i)},u^{(i)})$ and $\mu^+\ge \min_{i\in I^+}\gap(v^{(i)},u^{(i)})$, and thus \eqref{eq:far-away} yields
\begin{equation}\label{eq:far-away-2}
\mu^-<\frac{\delta^2\mu}{4n}\quad \mbox{and}\quad \mu^+>\frac{4n\mu}{\delta^2}\, .
\end{equation}
By the linearity of the duality gap (\Cref{prop:gap-linear}), we have $\mu=(1-\alpha)\mu^-+\alpha \mu^+$ and thus $\alpha\le \mu/\mu^+$, which together with \eqref{eq:far-away-2} gives
\begin{equation}\label{eq:alpha-bound}
\alpha\le \frac{\delta^2}{4n}\, .
\end{equation}

First, assume for a contradiction
 that $x_i^-\ge (1+\delta)x_i$  for some $i\in [n]$; the case  $s_i^-\ge (1+\delta)s_i$ follows analogously.
We use $(1-\alpha)x_i^-+\alpha x_i^+=x_i$. The lower bound on $x_i^-$ and $x_i^+ \ge 0$ imply $1-\alpha\le 1/(1+\delta)$, that is, $\alpha\ge \delta/(1+\delta)$, contradicting \eqref{eq:alpha-bound}.

Next, assume for a contradiction that
 $x^-_i\le (1-\delta)x_i$ and $s^-_i\le (1-\delta)s_i$ for some $i\in [n]$. Since $(1-\alpha)x^-_i+\alpha x^+_i=x_i$ and $(1-\alpha)s^-_i+\alpha s^+_i=s_i$, the assumptions on $x^-_i$ and $s^-_i$ yield
$\alpha x^+_i\ge \delta x_i$ and $\alpha s^+_i\ge \delta s_i$.
Thus,
\[
  \delta^2 x_is_i\le \alpha^2 x_i^+ s_i^+ \le
\alpha^2 \pr{x^+}{s^+}= n\alpha^2 \mu^+ \le n\alpha \mu\, ,
\]
where  the final inequality uses $\alpha \mu^+\le \mu$. Since $(x,s)\in\cN(\beta)$, we also get $x_is_i\ge (1-\beta)\mu$, which implies
\[
  \delta^2 (1-\beta)\le n\alpha\, ,
\]
in contradiction with \eqref{eq:alpha-bound}.
 \end{claimproof}
\begin{claim}\label{cl:rho}
For any $\rho \in(8n^2\mu^-/(\beta\mu),1]$ we have
\[
(\hat x,\hat s)\coloneqq (1-\rho)(x^-,s^-)+\rho(x,s)\in\cN(2\beta)\, .
\]
\end{claim}
\begin{claimproof}
Let us define
\[
B\coloneqq\{i\in [n]\, :\,   x^-_i\ge (1-\delta)x_i\}\quad\mbox{and}\quad N\coloneqq\{i\in [n]\, :\, s^-_i\ge (1-\delta)s_i\}\, .\]
By \Cref{cl:far-away}, $B\cup N=[n]$.
Further, $B\cap N=\emptyset$ by \eqref{eq:far-away-2} and $x_is_i\geq (1-\beta)\mu$ for all $i \in [n]$, where the latter fact is due to $(x,s) \in \cN(\beta)$.
Together, \Cref{cl:far-away} and \eqref{eq:far-away-2} imply that for all $i \in B$ we have
\begin{equation}\label{eq:s-minus-B}
s_i^-\le \frac{n\mu^-}{x^-_i}\le \frac{n\mu^-}{(1-\delta)x_i}\le \frac{n\mu^-}{(1-\delta)(1-\beta)\mu}\cdot s_i \le \frac{2n \mu^-}{\mu} \cdot s_i \le  \frac{\beta}{4n}\cdot \rho s_i\, ,
\end{equation}
by the lower bound on $\rho$.
Hence,
\begin{equation}\label{eq:hat-s-B}
  \hat s_i=(1-\rho)s_i^- +\rho s_i \in \left[1,1+\frac{\beta}{4n}\right] \rho s_i\, .
\end{equation}
For $i\in B$, we also have $x_i^-\in [1-\delta,1+\delta] x_i$ and therefore $\hat x_i\in [1-\delta,1+\delta]x_i$. Using the definition of $\delta$ and the analogous bounds for $i\in N$, we obtain
\begin{equation}
\hat x_i \hat s_i\in \left[1-\delta,1+\frac{\beta}{2n}\right] \rho x_is_i\, , \quad\forall i\in [n]\,.
\end{equation}
Therefore,
\[
\begin{aligned}
  \left\|\frac{\hat x\circ \hat s}{\rho \mu}
  -\1_n\right\|_2&= \left\|\frac{\hat x \circ \hat s -\rho x\circ s}{\rho \mu} + \frac{x\circ s}{\mu}
  -\1_n\right\|_2\\
  &\le \left\|\frac{\hat x \circ \hat s -\rho x\circ s}{\rho \mu}\right\|_2+ \left\|\frac{x\circ s}{\mu}
  -\1_n\right\|_2 \\
  & \le \frac{\beta}{2n} \cdot \left\|\frac{x \circ s}{\mu}\right\|_2 + \beta \\
  & \le \sqrt{n}\cdot\frac{\beta \cdot (1+\beta)}{2n}+\beta \le \Big(\frac{1 + 1/16}{2}+1\Big)\beta = \frac{49}{32} \beta\, .
  \end{aligned}
\]
Here, we used that $\delta \le \beta/(2n)$ by definition.
Using that $\beta\in (0,1/16]$, \Cref{lem:choice-of-mu} guarantees that $(\hat x,\hat s)\in \cN(\tfrac{49}{32}\beta/(1-\tfrac{49}{32}\beta))\subseteq \cN(2\beta)$.
\end{claimproof}
Let us now pick $i\in I^-$ with $\gap(v^{(i)},u^{(i)})$ maximal. By \Cref{prop:gap-linear}, $\mu^-$ is a convex combination of the gaps of the pairs indexed by $I^-$, and therefore
\begin{equation}\label{eq:mu-minus-bound}
0\le\mu^-\le\gap(v^{(i)},u^{(i)})\, .
\end{equation}
We distinguish the cases $\gap(v^{(i)},u^{(i)})>0$ and $\gap(v^{(i)},u^{(i)})=0$.

\paragraph{Case $\gap(v^{(i)},u^{(i)})>0$.}
We let
  \begin{equation}
  \label{eq:mu-star-definition}
  \mu^\star\coloneqq \frac{9n^2}{\beta} \cdot \gap(v^{(i)},u^{(i)})\, .
  \end{equation}
Then, $\size(\mu^\star)= \poly(n)\size(\mA,b,c)$ by the bound on $\size(\gap(v^{(i)},u^{(i)}))$ above together with \Cref{rem:size-bounds-encoded}, since $9n^2/\beta$ has size $O(\log n)$, and
by \eqref{eq:far-away} and \eqref{eq:delta-def},
\[
\mu^\star<\frac{9n^2}{\beta}\cdot\frac{\delta^2\mu}{4n}=\frac{9\beta}{1024\,n^3}\cdot\mu<\mu\, .
\]
Moreover, \eqref{eq:mu-minus-bound} gives $\mu^-\le\gap(v^{(i)},u^{(i)})=\beta\mu^\star/(9n^2)<\mu^\star$. Let us now select  $\rho$ such that  $(1-\rho)\mu^- +\rho \mu=\mu^\star$, that is, $\rho=(\mu^\star-\mu^-)/(\mu-\mu^-)$. As $\mu^-<\mu^\star<\mu$, we have $\rho\in(0,1)$, and
\[
\rho > \frac{\mu^\star-\mu^-}{\mu}\ge\Big(\frac{9n^2}{\beta}-1\Big)\frac{\mu^-}{\mu}\ge\frac{8n^2\mu^-}{\beta\mu}\, ,
\]
using $\beta\le 1\le n^2$ in the last step, as required for \Cref{cl:rho}. For $(\hat x,\hat s)=(1-\rho)(x^-,s^-)+\rho(x,s)$, by \Cref{prop:gap-linear} we have $\gap(\hat x,\hat s)=\mu^\star$, and $(\hat x,\hat s)\in\cN(2\beta)$ by \Cref{cl:rho}. Since $\mu^\star>0$, outcome \ref{it:round-gap-pos} holds. Again, we can compute $(\hat x,\hat s)$ in strongly polynomial time.

\paragraph{Case $\gap(v^{(i)},u^{(i)})=0$.} Then $\mu^-=0$ by \eqref{eq:mu-minus-bound}. We output $(\hat x,\hat s)\coloneqq(x^-,s^-)$ together with $\mu^\star\coloneqq 0$ and the directions
\[
\hat v\coloneqq\frac{s-\hat s}{\mu}\, ,\qquad \hat w\coloneqq\frac{x-\hat x}{\mu}\, .
\]
Then $\gap(\hat x,\hat s)=0\le\gap(x,s)$, the bounds $\gap(\hat x,\hat s)\le\mu^\star\le(1+\beta/(4n^2))\gap(\hat x,\hat s)$ hold with equality, $\size(\mu^\star)=O(1)$, and $(\hat x,\hat s,\hat v,\hat w)$ is computed in strongly polynomial time. Moreover, as $\mu^-=0$, \Cref{cl:rho} applies to every $\rho\in(0,1]$, so $(\hat x,\hat s)+\rho\mu(\hat w,\hat v)=(1-\rho)(x^-,s^-)+\rho(x,s)\in\cN(2\beta)$ for all $\rho\in(0,1]$; letting $\rho\downarrow 0$ shows $(\hat x,\hat s)\in\clN(2\beta)$.

It remains to verify the conditions on $(\hat v,\hat w)$. As $x,\hat x\in\Primal$ and $s,\hat s\in\Dual$, we have $\mA\hat w=\0_m$ and $\hat v\in\im(\mA^\top)$. By \Cref{cl:far-away}, for every $j\in[n]$ we have $\hat x_j\ge(1-\delta)x_j>0$ or $\hat s_j\ge(1-\delta)s_j>0$, so $B\cup N=[n]$; further, $B\cap N=\emptyset$, since $\pr{\hat x}{\hat s}=n\mu^-=0$ and $\hat x,\hat s\ge\0_n$. For $j\in B$, we have $\hat s_j=0<(1-\delta)s_j$, so \Cref{cl:far-away} gives $\hat x_j\in[1-\delta,1+\delta]\,x_j$; moreover, $\hat v_j=s_j/\mu>0$ and $\hat x_j\hat v_j=(\hat x_j/x_j)\cdot x_js_j/\mu$. Similarly, for $j\in N$ we have $\hat s_j\in[1-\delta,1+\delta]\,s_j$, $\hat w_j=x_j/\mu>0$ and $\hat s_j\hat w_j=(\hat s_j/s_j)\cdot x_js_j/\mu$. We index the vector $(\hat x_B\circ\hat v_B,\hat s_N\circ\hat w_N)$ by $[n]$, and let $p\coloneqq x\circ s/\mu$. Then $(\hat x_B\circ\hat v_B,\hat s_N\circ\hat w_N)=\sigma\circ p$ for some $\sigma\in\R^n$ with $\|\sigma-\1_n\|_\infty\le\delta$. As $(x,s)\in\cN(\beta)$, we have $\|p-\1_n\|_2\le\beta$ and $\|p\|_2\le\sqrt n+\beta\le 2\sqrt n$, and therefore
\[
\norm{\sigma\circ p-\1_n}_2\le\norm{(\sigma-\1_n)\circ p}_2+\norm{p-\1_n}_2\le 2\sqrt n\,\delta+\beta\le 2\beta
\]
by the definition of $\delta$ in \eqref{eq:delta-def}. Hence outcome \ref{it:round-gap-zero} holds.
\end{proof}
\begin{remark}
The path along which we `shoot down' in the above proof exhibits a \emph{`polarized'} structure: primal variables in $B$ stay roughly the same and scale down linearly with $\mu$ in $N$; for dual variables, the role of $N$ and $B$ is reversed. We note that this is the crucial structural property exploited in the SLLS steps of \cite{ADLNV2022}.
\end{remark}

We need one more lemma to complete the proof of \Cref{lem:main-round}.
\begin{lemma}\label{lem:approximating-analytic-center}
Let $\mD\in \eQ^{m\times n}$ with $m\le n$, $h\in \eQ^m$, let $\cal Q = \{x\in\R^n \mid \mD x = h, x \ge \0_n\} \subset \R^n$ be a polytope, let $J\subseteq [n]$ be an index set, and let $y\in{\cal Q}\cap\eQ^n$ with $y_j>0$ for all $j\in J$. Let $R,K\in\Zplusplus$ be
such that $R\ge \max_{x \in {\cal Q}} \|x_J/y_J\|_\infty$. There is an algorithm that computes a vector $z \in {\cal Q}\cap \eQ^n$ such that $\|(z_J - y_J)/y_J\|_2 \le 1/K$ with $\size(z) \le \poly(n)\size(\mD,h)+ \tilde O(n^2\log(K+R))$. The running time is polynomial in $n$ and $\log(K+R)$ and the space requirement is $\poly(n)\size(\mD,h,y)+ \tilde O(n^2\log(K+R))$.
\end{lemma}
\begin{proof}
First, we apply \Cref{lem:minkowski-weyl} to obtain a Minkowski--Weyl decomposition $y=\sum_{i=1}^k \lambda_i v^{(i)}$, $\sum_{i=1}^k \lambda_i=1$, $\lambda\ge 0$, such that $k\le n+1$ and each $v^{(i)}$ is a vertex of $\cal Q$. Note that the decomposition does not include extreme rays, since the polytope is bounded. By \Cref{lem:minkowski-weyl}, this single decomposition costs $O(n^7)$ arithmetic operations and comparisons and runs in space $\poly(n)\bigl(\size(\mD,h,y)\bigr)$, with each $v^{(i)}$ of size $\poly(n)\size(\mD,h)$.
Let $\delta\coloneqq 1/(4n^2 RK)$, and let us round the coefficients $\lambda_i$ to $\tilde \lambda_i$ such that each $\tilde\lambda_i$ is an integer multiple of $\delta$, $\sum_{i=1}^k \tilde \lambda_i=1$, $\tilde\lambda\ge 0$, and $\|\tilde \lambda-\lambda\|_\infty\le \delta$. To achieve this, we first round down each $\lambda_i$ to the nearest integer multiple of $\delta$. This can be done by $O(n\log (1/\delta))=O(n\log(n+K+R))$ comparisons in total by doing a binary search on each coordinate. Let $\lambda'_i$ denote these coordinates. At this point, $1-\sum_{i=1}^k\lambda_i'=\ell \delta$ for an integer $\ell\le n$. We now pick arbitrary $\ell$ coordinates and set $\tilde \lambda_i=\lambda_i'+\delta$ for them, and set $\tilde \lambda_i=\lambda_i'$ for all other coordinates.

We claim that $z\coloneqq\sum_{i=1}^k \tilde\lambda_i v^{(i)}$ satisfies the requirements. First, note that
\[
\left\|\frac{z_J-y_J}{y_J}\right\|_2\le \sum_{i=1}^k |\lambda_i-\tilde\lambda_i|\cdot\left\|\frac{v^{(i)}_J}{y_J}\right\|_2\le 2R n^{1.5}\delta< \frac1K\, ,\]
since each $|\lambda_i-\tilde\lambda_i|\le \delta$ and $\|v^{(i)}_J/y_J\|_2\le \sqrt{|J|}\|v^{(i)}_J/y_J\|_\infty \le \sqrt{n} R$ by the definition of $R$.
 The bound on the encoding length of $z$ follows since each vertex $v^{(i)}$ returned by the decomposition has, by \Cref{lem:minkowski-weyl}, size bounded by $\poly(n)\size(\mD,h)$, and each coefficient $\tilde \lambda_i$ is represented as a ratio of two integers of absolute value at most $1/\delta^2 = (4n^2 RK)^2$, so $\size(\tilde\lambda_i)=O(\log(nRK))$. %
Indeed, we compute each coordinate $z_j$ as the sum of the $k\le n+1$ products $\tilde\lambda_i v^{(i)}_j$. By \Cref{rem:size-bounds-encoded}, $\size(z_j)\le 2\sum_{i=1}^k\bigl(\size(\tilde\lambda_i)+\size(v^{(i)}_j)\bigr)\le 2\sum_{i=1}^k\size(v^{(i)}_j)+O(n\log(nRK))$. Summing over $j\in[n]$ gives $\size(z)\le 2\sum_{i=1}^k\size(v^{(i)})+O(n^2\log(nRK))=\poly(n)\size(\mD,h)+\tilde O(n^2\log(K+R))$. The space bound follows by construction. Beyond the decomposition, the algorithm performs $O(n^2+n\log(K+R))$ operations.
\end{proof}

\begin{proof}[Proof of \Cref{lem:main-round}]
Given $(x,s)\in \cal N^2(\beta)\cap \eQ^{2n}$, we first use \Cref{lem:round-gap} to obtain $(\hat x,\hat s)\in \clN(2\beta)$ and a value $\mu^\star$ with $\hat \mu\le \mu^\star\le(1+\beta/(4n^2))\hat \mu$ for $\hat\mu=\gap(\hat x,\hat s)$, satisfying $\size(\mu^\star)= \poly(n)\size(\mA,b,c)$. Let us also fix any $d\in \eQ^n$ such that $\mA d=b$, computed as follows. We set $d\coloneqq\mA^+b$, computed in $O(n^5)$ arithmetic operations and comparisons by \Cref{lem:linalg}\ref{it:linalg-pinv}; as $\rank(\mA)=m$, we have $\mA\mA^+=\mI_m$, and hence $\mA d=b$. Then we get $\size(d)= \poly(n)\size(\mA,b)$ by \Cref{lem:linalg}\ref{it:linalg-pinv} and \Cref{rem:size-bounds-encoded}. Note that \Cref{lem:size-bounds}\ref{it:inverse} does not apply here, since $d$ is not computed in lowest terms. The vector $d$ is fixed throughout the algorithm.
We use \Cref{lem:approximating-analytic-center}  on the primal-dual polyhedron, which we write in standard form as in the proof of \Cref{lem:round-gap}: we describe $\Dual$ without the variable $y$ using the matrix $\mB\in\eQ^{n\times n}$ defined there, which satisfies $\ker(\mB)=\im(\mA^\top)$, and we add a slack variable $\sigma$ for the cost inequality. Namely,
let us define
\[
{\cal Q}\coloneqq \left\{(x',s',\sigma)\in\R^{n + n + 1}\mid
\mA x' = b\, , \ \mB s' =\mB c\, ,\ \pr{c}{x'}+\pr{d}{s'}+\sigma= \pr{c}{d}+n\mu^\star\, ,\ (x',s',\sigma) \ge \0_{2n+1}
\right\}\, ,
\]
that is, the set of primal and dual solutions with gap at most $\mu^\star$, together with the slack $\sigma$; note that $\pr{x'}{s'}\le n\mu^\star$ follows by \Cref{prop:gap-formula}, which applies since $\mB(c-s')=\0_n$ means $c-s'\in\im(\mA^\top)$, and gives $\pr{c}{x'}+\pr{d}{s'}=\pr{c}{d}+\pr{x'}{s'}$.
We note that $\cal Q$ is in fact a polytope. For the $x'$ and $s'$ variables, we have both lower and upper bounds by \Cref{prop:gap-formula}, using that $\pr{c}{x'} + \pr{d}{s'} = \pr{c}{d} + \pr{x'}{s'} + \pr{x}{s} - \pr{x}{s} = \pr{c}{d} + \pr{x}{s'} + \pr{x'}{s} - \pr{x}{s} $ and that $x, s > \0_n$.
Finally, $0\le\sigma\le\pr{c}{d}+n\mu^\star-\pr{c}{x'}-\pr{d}{s'}$, which is bounded since $x'$ and $s'$ are. Hence, $\cal Q$ is bounded.

We will apply \Cref{lem:approximating-analytic-center} to the polytope $\cal Q$, with $K$ the smallest power of two with $K\ge 8n/\beta$, and with the input point $(\hat x,\hat s,\hat\sigma)\in{\cal Q}\cap \eQ^{2n+1}$, where $\hat\sigma\coloneqq\pr{c}{d}+n\mu^\star-\pr{c}{\hat x}-\pr{d}{\hat s}=n(\mu^\star-\hat\mu)\ge 0$ by \Cref{prop:gap-formula}. As the index set we take
\[
J\coloneqq\{j\in[n]\mid \hat x_j>0\}\cup\{n+j\mid j\in[n],\ \hat s_j>0\}\subseteq[2n+1]\, ,
\]
that is, the $x'$- and $s'$-coordinates on which $(\hat x,\hat s)$ is supported. The bound $R$ is chosen separately in each case below. The data $\mD,h$ describing $\cal Q$ are formed from $\mA,\mB,b,c,d$ and $\mu^\star$. By the proof of \Cref{lem:round-gap}, $\mB$ is computed in $O(n^4)$ arithmetic operations, with $\size(\mB)=\poly(n)\size(\mA)$ and $\size(\mB c)=\poly(n)\size(\mA,c)$. Hence, the size of the description of $\cal Q$ is $O(\size(\mA,\mB,b,\mB c,c,d, \mu^\star))=\poly(n)\size(\mA,b,c)$.

We distinguish the cases $\mu^\star>0$ and $\mu^\star=0$, corresponding to the outcomes \ref{it:round-gap-pos} and \ref{it:round-gap-zero} of \Cref{lem:round-gap}.

\paragraph{Case $\mu^\star>0$.} Then $(\hat x,\hat s)\in\cN(2\beta)$ by \Cref{lem:round-gap}\ref{it:round-gap-pos}; in particular $\hat x,\hat s>\0_n$, so $J$ consists of all $x'$- and $s'$-coordinates.
 For every $(x,s,\sigma)\in{\cal Q}$ we have $\gap(x,s)\le\mu^\star\le(1+\beta/(4n^2))\hat\mu$, so \Cref{lem:near-monotonicity} with $\mu=\hat\mu$ and \Cref{prop:near-central} with $2\beta$ in place of $\beta$ give
\[
(x,s)\le n\Bigl(2+\frac{\beta}{4n^2}\Bigr)\bigl(x\cp(\hat\mu),s\cp(\hat\mu)\bigr)\le \frac{n}{1-2\beta}\Bigl(2+\frac{\beta}{4n^2}\Bigr)(\hat x,\hat s)\le 3n\,(\hat x,\hat s)\, ,
\]
using $\beta\le 1/16$. Hence, $R\coloneqq 3n$ is a valid choice.
 Applying \Cref{lem:approximating-analytic-center} with this $R$, we obtain a point $(x',s')$  with $\size(x', s') = \poly(n)\size(\mA, b, c)$ and
 \begin{equation}\label{eq:x-hat-x}
 \left\|\frac{x'}{\hat x}-\1_n\right\|_2+ \left\|\frac{s'}{\hat s}-\1_n\right\|_2
\le \sqrt{2}\left(\left\|\frac{x'}{\hat x}-\1_n\right\|_2^2+\left\|\frac{s'}{\hat s}-\1_n\right\|_2^2\right)^{1/2}
  \le \frac{\sqrt{2}}K \le \frac{\beta}{4n}\, .
 \end{equation}
For $\mu' \coloneqq \gap(x',s')$ we can bound the ratio ${\mu'}/\hat\mu$ by %
\begin{equation}\label{eq:mup-mu}
\left(1+\frac{\beta}{4n^2}\right)\hat \mu\ge \mu^\star \ge \mu'=\frac{1}{n}\sum_{i=1}^n x'_i s'_i\ge \frac{(1-1/K)^2}{n}\sum_{i=1}^n \hat x_i \hat s_i \ge {\left(1-\frac{\beta}{4n}\right)^2}\hat\mu\, .
\end{equation}
 We are ready to show that $(x',s')\in{\cal N}^2(3\beta)$.
\begin{equation}
    \begin{aligned}
    \left\|\frac{x'\circ s'}{\mu'} - \1_n\right\|_2 &= \left\|
        \left(\frac{x'}{\hat x} - \1_n + \1_n\right)\circ \left(\frac{s'}{\hat s} - \1_n + \1_n\right)\circ \frac{\hat x\circ \hat s}{\mu'} - \1_n
    \right\|_2 \\
    & \le \left(\left\|\frac{x'}{\hat x} - \1_n\right\|_2 + \left\|\frac{s'}{\hat s} - \1_n\right\|_2 + \left\|\frac{x'}{\hat x} - \1_n\right\|_2 \left\|\frac{s'}{\hat s} - \1_n\right\|_2\right)  \left\|\frac{\hat x\circ \hat s}{\mu'}\right\|_\infty + \left\|\frac{\hat x\circ \hat s}{\mu'} - \1_n\right\|_2 \\
    \end{aligned}
  \end{equation}
Using \eqref{eq:x-hat-x}, and bringing out a factor $\hat \mu/\mu'$ in both terms, we get
  \begin{equation}
    \begin{aligned}
     \left\|\frac{x'\circ s'}{\mu'} - \1_n\right\|_2&\le \frac{\hat\mu}{\mu'} \left( \frac{\beta}{4n} + \left(\frac{\beta}{4n}\right)^2\right)\left\|\frac{\hat x\circ \hat s}{\hat \mu}\right\|_\infty  + \frac{\hat\mu}{\mu'}\left\|\frac{\hat x\circ \hat s}{\hat\mu} - \frac{\mu'}{\hat \mu}\1_n\right\|_2  \\
     &\le \frac{\hat \mu}{\mu'}\left(\frac{\beta}{4n}\left(1+\frac{\beta}{4n}\right) (1+2\beta)+ \left\|\frac{\hat x\circ \hat s}{\hat \mu} - \1_n\right\|_2  + \left\|\frac{\mu'}{\hat\mu}\1_n-\1_n\right\|_2 \right)\\
     &\leq\left(1-\frac{\beta}{4n}\right)^{-2}\left(\frac{\beta}{4n}\left(1+\frac{\beta}{4n}\right)(1+2\beta)+2\beta+\sqrt{n}\frac{\beta}{2n}\right) \\
     &= \left(1-\frac{\beta}{4n}\right)^{-2}\left(\frac{\left(1+\frac{\beta}{4n}\right)(1+2\beta)}{4n}+2+\frac{1}{2\sqrt{n}}\right)\beta\\
     &\le \left(1-\frac{1}{64}\right)^{-2}\left(\frac{1}{4}\cdot\frac{65}{64}\cdot\frac98+2+\frac{1}{2}\right)\beta \\
     &\le 3\beta\, .
    \end{aligned}
\end{equation}
Here, the first inequality used \eqref{eq:x-hat-x}, the second used that $(\hat x,\hat s)\in{\cal N}^2(2\beta)$ and the triangle inequality, the third used again $(\hat x,\hat s)\in{\cal N}^2(2\beta)$ and the two-sided bound \eqref{eq:mup-mu} on $\mu'/\hat\mu$.
Therefore, $(x',s') \in \cal N^2(3\beta)$. We set $(\tilde x,\tilde s)\coloneqq(x',s')\in\eQ^{2n}$, which has size $\poly(n)\size(\mA,b,c)$ as shown above, and, by \eqref{eq:mup-mu} and $\gap(\hat x,\hat s)\le\gap(x,s)$ from \Cref{lem:round-gap},
\[
\gap(\tilde x,\tilde s)=\mu'\le\mu^\star\le\Bigl(1+\frac{\beta}{4n^2}\Bigr)\hat\mu\le\Bigl(1+\frac{\beta}{4n^2}\Bigr)\gap(x,s)\, .
\]
This proves the theorem in this case.
\paragraph{Case $\mu^\star=0$.} Then $\gap(\hat x,\hat s)=0$, so $(\hat x,\hat s)$ is a primal--dual optimal pair, and $\cal Q$ consists of the points $(x',s',0)$ for which $(x',s')$ is a primal--dual optimal pair, as $\sigma=-\pr{x'}{s'}$ on $\cal Q$ by \Cref{prop:gap-formula}. Moreover, by \Cref{lem:round-gap}\ref{it:round-gap-zero}, the algorithm also returns $(\hat v,\hat w)$ such that $B\coloneqq\supp(\hat x)$ and $N\coloneqq\supp(\hat s)$ partition $[n]$, $\hat v\in\im(\mA^\top)$, $\hat v_B>\0_B$, $\hat w\in\ker(\mA)$, $\hat w_N>\0_N$, and $\hat q\coloneqq(\hat x_B\circ\hat v_B,\,\hat s_N\circ\hat w_N)$ satisfies $\|\hat q-\1_n\|_2\le 2\beta$. In particular, $J$ consists of the $x'$-coordinates indexed by $B$ and the $s'$-coordinates indexed by $N$.

We first show, using the certificates $(\hat v,\hat w)$, that every primal--dual optimal pair $(x',s')$ satisfies $(x',s')\le 2n\,(\hat x,\hat s)$. By \Cref{prop:gap-formula}, $\pr{x'}{\hat s}+\pr{\hat x}{s'}=\pr{x'}{s'}+\pr{\hat x}{\hat s}=0$, so $x'_N=\0_N$ and $s'_B=\0_B$ by nonnegativity. As $x'-\hat x\in\ker(\mA)$ is orthogonal to $\hat v\in\im(\mA^\top)$, and $x'$ and $\hat x$ vanish on $N$, we get for every $j\in B$
\[
x'_j\hat v_j\le\pr{x'_B}{\hat v_B}=\pr{\hat x_B}{\hat v_B}\le\|\hat q\|_1\le n+2\beta\sqrt n\, .
\]
As $\hat x_j\hat v_j=\hat q_j\ge 1-2\beta$, this gives $x'_j\le\frac{n+2\beta\sqrt n}{1-2\beta}\,\hat x_j\le 2n\,\hat x_j$, using $\beta\le 1/16$. The bound $s'_N\le 2n\,\hat s_N$ follows in the same way from $\pr{\hat w}{s'-\hat s}=0$.

We apply \Cref{lem:approximating-analytic-center} with $R=2n$, which is valid by the bound just shown. It returns $(x',s',\sigma')\in{\cal Q}\cap\eQ^{2n+1}$ with $\size(x',s')=\poly(n)\size(\mA,b,c)$ and $\|r-\1_n\|_2\le\beta/(8n)$, where $r\coloneqq(x'_B/\hat x_B,\,s'_N/\hat s_N)$; in particular $x'_B>\0_B$ and $s'_N>\0_N$. Next, we compute new certificates. Let $v\in\im(\mA^\top)$ minimize $\|x'_B\circ v_B-\1_B\|_2$, and let $w\in\ker(\mA)$ minimize $\|s'_N\circ w_N-\1_N\|_2$. These least-squares problems are solved by $v=\mA^\top\mM^+\1_B$ for $\mM\coloneqq\mathrm{Diag}(x'_B)\,\mA_B^\top$, and $w=\mPi\,(\mM')^+\1_N$ for the orthogonal projection $\mPi\coloneqq\mI_n-\mA^+\mA$ onto $\ker(\mA)$ and $\mM'\coloneqq\mathrm{Diag}(s'_N)\,\mPi_{N,\bullet}$. By \Cref{lem:linalg}\ref{it:linalg-pinv}, applied to these matrices, $v$ and $w$ are computed in strongly polynomial time and have size $\poly(n)\size(\mA,b,c)$. Since $\hat v$ and $\hat w$ are feasible for these least-squares problems and $(x'_B\circ\hat v_B,\,s'_N\circ\hat w_N)=r\circ\hat q$, optimality of $v$ and $w$ gives, for $q\coloneqq(x'_B\circ v_B,\,s'_N\circ w_N)$,
\[
\|q-\1_n\|_2\le\|r\circ\hat q-\1_n\|_2\le\|\hat q\|_\infty\,\|r-\1_n\|_2+\|\hat q-\1_n\|_2\le(1+2\beta)\,\frac{\beta}{8n}+2\beta\le\frac{137}{64}\beta\eqqcolon\theta\, .
\]
As $\theta<1$, this implies $q>\0_n$, and hence $v_B>\0_B$ and $w_N>\0_N$.

We set $(\tilde x,\tilde s)\coloneqq(x',s')$, so that $\gap(\tilde x,\tilde s)=0\le(1+\beta/(4n^2))\gap(x,s)$, and show that $(x',s')\in\clN(3\beta)$. For $t>0$ let $z(t)\coloneqq(x'+tw,\ s'+tv)$. Since $w\in\ker(\mA)$ and $v\in\im(\mA^\top)$, $z(t)$ satisfies the equality constraints of $\Primal\times\Dual$, and since $x'_N=\0_N$, $s'_B=\0_B$ and $x'_B,v_B,s'_N,w_N>0$, we have $z(t) > \0_{2n}$ for all sufficiently small $t > 0$. Moreover, $(x'+tw)\circ(s'+tv)=t\,(q+t\,w\circ v)$, so \Cref{lem:choice-of-mu}, applied with $\mu'=t$ and the parameter $\theta_t\coloneqq\theta+t\,\|w\circ v\|_2\le 1/4$, gives $z(t)\in\cN\bigl(\theta_t/(1-\theta_t)\bigr)$ for all sufficiently small $t>0$. As $\theta/(1-\theta)\le\frac{137}{64}\beta/(1-\frac{137}{1024})<3\beta$, we get $z(t)\in\cN(3\beta)$ for all sufficiently small $t>0$, and letting $t\downarrow 0$ yields $(\tilde x,\tilde s)\in\clN(3\beta)$. Note that $(x',s',v,w)$ moreover satisfies the output conditions of \Cref{thm:SLC-bound} with $\theta$ in place of $\beta$, and has size $\poly(n)\size(\mA,b,c)$. Since $\theta=\frac{137}{64}\beta\le 3\beta$, returning $(\tilde v,\tilde w)\coloneqq(v,w)$ proves the additional claim of the theorem for the case $\gap(\tilde x,\tilde s)=0$.
\end{proof}

Finally, we show that using the rounding step in \Cref{lem:main-round} in combination with \emph{any} path-following IPM already guarantees termination in an exponential number of steps. We state this formally for methods in the $\ell_2$-neighborhood. In the statement, the subroutine $F(x,s)$ represents a possibly longer sequence of IPM iterates. By the algorithm in \Cref{lem:main-round}, we refer to the particular implementation using \Cref{lem:round-gap} and  \Cref{lem:approximating-analytic-center}.

\begin{proposition}\label{prop:simplex}
Let $\beta\in (1/\mathrm{poly}(n),1/16]$, and assume we are given a subroutine that, for any $(x,s)\in{\cal N}^2(3\beta)$, returns an iterate $F(x,s)=(x',s')\in{\cal N}^2(\beta)$ such that $\gap(x',s')\le (1-2\beta/n^2)\gap(x,s)$.
Assume further we are given a starting point $(x^0,s^0)\in{\cal N}^2(\beta)$. Then, the algorithm alternating between calls to the rounding algorithm in \Cref{lem:main-round} and calls to $F$ finds an exact optimal solution after $2^{2n}\poly(n)$ calls to $F$ and to the rounding algorithm.
\end{proposition}
\begin{proof}
Let $H\subseteq \Rplus\cup\{\infty\}$ be the set of all values $\gap(v,u)$ such that $v$ is a vertex of $\Primal$ and $u$ is a vertex of $\Dual$; we also add $\infty$ to the set.  Thus, $|H|\le 2^{2n}+1$; let $H=\{0=h_0<h_1<h_2<\ldots<h_T<h_{T+1}=\infty\}$.
An iteration consists of a call to the rounding algorithm followed by a call to $F$, and we stop once the rounding algorithm returns a point with zero gap, which is optimal. Otherwise, the rounded point has strictly positive gap and lies in $\cN(3\beta)$. Each iteration multiplies the gap by at most $(1+\beta/(4n^2))(1-2\beta/n^2)\le 1-\beta/n^2$; as $\beta>1/\poly(n)$, decreasing the gap by a $\poly(n)$ factor takes $\poly(n)$ iterations.

We claim that if the current point $(x,s)$ has $\gap(x,s)\le h_t$ for some $t\in [1,T+1]$, then within $\poly(n)$ iterations we either stop or reach a point $(\bar x,\bar s)$ with $\gap(\bar x,\bar s)\le h_{t-1}$. Since $h_0=0$ and iterates have positive gap, applying the claim for $t=T+1,T,\ldots,1$ proves the statement. In particular, if each call to $F$ consists of $\poly(n)$ IPM iterations, the algorithm performs $2^{2n}\poly(n)$ IPM iterations in total.

Let $\delta=\beta/(16n^2)$ as in \eqref{eq:delta-def}. Within $\poly(n)$ iterations we reach an iterate $(x',s')$ with $\gap(x',s')<h_t\delta^2/(4n)$. If $\gap(x',s')\le h_{t-1}\cdot 4n/\delta^2$, then another $\poly(n)$ iterations bring the gap to at most $h_{t-1}$; for $t=1$, this cannot happen as $h_0=0$.

Otherwise, $\mu\coloneqq\gap(x',s')$ lies strictly between $h_{t-1}\cdot 4n/\delta^2$ and $h_t\delta^2/(4n)$, so no element of $H$ lies in $[\mu\delta^2/(4n),4n\mu/\delta^2]$. The vertex pairs in the decomposition in the proof of \Cref{lem:round-gap} have gaps in $H$, so the next rounding call is in the case \eqref{eq:far-away}, and all pairs in $I^-$ have gap at most $h_{t-1}$. Let $g\le h_{t-1}$ be the largest of these gaps. If $g=0$, as is always the case for $t=1$, then $\mu^\star=0$, the rounding algorithm returns an optimal pair by the case $\mu^\star=0$ in the proof of \Cref{lem:main-round}, and we stop. If $g>0$, it returns a point with gap at most $\mu^\star=\frac{9n^2}{\beta}g\le\frac{9n^2}{\beta}h_{t-1}$, and another $\poly(n)$ iterations bring the gap to at most $h_{t-1}$.
\end{proof}

\subsection*{Acknowledgments}
This project has received funding from the European Research Council (ERC) under the European Union's Horizon 2020 research and innovation programme (grant agreement nos. 757481-ScaleOpt and 805241-QIP), and from the Dutch Science Foundation (NWO) under the NWO Talent Programme (grant 016.Vidi.189.087).
Part of this work was done during the ICERM programme ''Combinatorics and Optimization'' and the Simons Institute programme ''Data Structures and Optimization for Fast Algorithms''.

\subsection*{AI Disclosure}
The mathematical ideas of the paper were obtained without using AI models. OpenAI's GPT-6 Astra and Anthropic's Claude Opus 5.5 were used in editing and  proofreading the manuscript, leading to several  corrections and clarifications.
\bibliographystyle{myalpha}
\bibliography{genflow}

\appendix

\section{Minkowski--Weyl decomposition in standard form}\label{sec:mw-explicit}

In this appendix we prove \Cref{lem:linalg,lem:minkowski-weyl}, including the running time and space bounds used by the rounding procedure of \Cref{sec:bit-complexity}. Results of this kind appear throughout the literature; we nonetheless carry the computations out in full, for two reasons. First, the standard references work in a different computational model: the algorithmic Carath\'eodory theorem \cite[Theorem 6.5.11]{gls} for polytopes, and its Minkowski--Weyl extension to pointed polyhedra \cite[Corollary 6.5.13]{gls}, are stated for a polyhedron given by a strong optimization oracle and are oracle-polynomial, whereas our polyhedron is given explicitly as $\mA x=b$, $x\ge\0_n$. Second, we bound the space used by the algorithm in the most restrictive of the strongly polynomial models discussed in \cite[Section 1.3]{gls}, which is more restrictive than the model used in \cite{gls} itself.

We make no attempt to optimize the running time in what follows. In particular, we abstract the basic linear algebra away into the cost of implementing a handful of subroutines, and we make no use of dynamic data structures that could amortize the computations. This is partly for clarity of presentation and partly because many of the implementations of these subroutines in the literature do not work in our strongly polynomial model. The algorithm is otherwise the classical one, and the encoding-length argument is that of \cite[Theorem 6.5.11]{gls}, carried out explicitly in standard form.

\subsection{The cost of strongly polynomial linear algebra}\label{sec:mw-cost}
\begin{definition}[Linear-algebra tasks and their cost]\label{def:tasks}
Each of the quantities below denotes the number of arithmetic operations and comparisons on rationals performed by an algorithm for the stated task that is \emph{strongly polynomial in the Turing model} in the sense of \Cref{sec:comp-mod}. The input matrix is $\mM\in\eQ^{m\times n}$ with $m\le n$, or, in \ref{it:task-det}--\ref{it:task-inv}, a square matrix $\mM\in\eQ^{r\times r}$ with $r\le n$.
\begin{enumerate}[label=(\roman*), ref=(\roman*)]\setlength{\itemsep}{0pt}
\item\label{it:task-mul} $\Tmul(n)$: computing a product $\mM\mN$ of two conformable matrices $\mM\in\eQ^{m\times n}$ and $\mN\in\eQ^{n\times r}$ with $m,r\le n$.
\item\label{it:task-det} $\Tdet(n)$: computing the determinant $\det\mM$ of a square $\mM\in\eQ^{r\times r}$ with $r\le n$.
\item\label{it:task-adjt}$\Tadj(n)$: computing the adjugate $\adjt(\mM)$ of a square $\mM\in\eQ^{r\times r}$ with $r\le n$.
\item\label{it:task-inv} $\Tinv(n)$: computing the inverse $\mM^{-1}\in\eQ^{r\times r}$ of a nonsingular $\mM\in\eQ^{r\times r}$ with $r\le n$.
\item\label{it:task-basis} $\Tbasis(n)$: computing a \emph{basis} of $\mM$, that is, an index set $B\subseteq[n]$ such that $\{\mM_j\}_{j\in B}$ is a maximal linearly independent set of columns of $\mM$.
\item\label{it:task-circuit} $\Tcirc(n)$: returning a \emph{circuit} of $\mM$, that is, an elementary vector $g\in\eQ^n\setminus\{\0_n\}$ with $\mM g=\0_m$ and $\supp(g)$ minimal among the supports of the nonzero kernel vectors of $\mM$, or reporting that the columns of $\mM$ are linearly independent.
\end{enumerate}
\end{definition}

We will give explicit upper bounds for \ref{it:task-mul}, \ref{it:task-det} and \ref{it:task-adjt}, and reduce the remaining tasks \ref{it:task-inv}--\ref{it:task-circuit} to these three. Two further tasks, basis crashing and the Minkowski--Weyl decomposition itself, are introduced with the same counting convention in \Cref{def:task-crash,def:task-mw}.

Before doing so, we recall that all numbers are encoded rationals as in \Cref{def:size}, never reduced to lowest terms. Every size bound in this appendix refers to the computed representation.

We begin with the standard cubic bound for matrix multiplication, from which all other running time bounds below are derived. We work with this obvious bound and  do not invoke fast matrix multiplication results, as the strongly polynomial implementability of some of these algorithms would require more involved technical arguments.

\begin{proposition}[Matrix multiplication]\label{prop:tmul}
$\Tmul(n)=O(n^3)$.
\end{proposition}

The determinant and the adjugate below are computed by the method of Le Verrier \cite{leverrier1840}, also known as the Faddeev--Le Verrier method (see also \cite{hou98}). The method presented is again not as fast as possible, but it is sufficient for our purposes.

\begin{proposition}[Le Verrier's identities]\label{prop:leverrier}
Let $\mM\in\Z^{n\times n}$, let $\chi_{\mM}(\lambda) = \det(\lambda\mI_n-\mM) = \sum_{i = 0}^n c_i \lambda^i$ be the characteristic polynomial of $\mM$, and let $s_j\coloneqq\operatorname{tr}(\mM^j)$. Then, $c_n = 1$, and
\begin{align}
\label{eq:newton}(n-i)\,c_i&=-\sum_{j=1}^{n-i}c_{i+j}\,s_j\qquad(0\le i\le n-1)\, ,\\
\label{eq:cayley}\adjt(\mM)&=(-1)^{n+1}\sum_{i=1}^{n}c_i\,\mM^{\,i-1}\, ,
\end{align}
and $\det\mM=(-1)^nc_0$.
\end{proposition}

Identity \eqref{eq:newton} is Newton's identity in the Faddeev--Le Verrier form \cite[Algorithm 2.2.7]{cohen93}, \cite{hou98}; \eqref{eq:cayley} follows from the  Cayley--Hamilton theorem.

To compute the determinant and the adjugate in our computational model we require rescalings of the $c_i$.

\begin{corollary}\label{cor:leverrier-integral}
\begin{equation}\label{eq:tdet}
\Tdet(n),\ \Tadj(n)=O\bigl(n\,\Tmul(n)\bigr).
\end{equation}
\end{corollary}
\begin{proof}
Let $\mM\in\eQ^{r\times r}$ with $r\le n$. We first make $\mM$ integral for the following reason. Le Verrier's method computes the powers $\mM^j$, and on encoded rationals, which are never reduced, their denominators grow too fast for our required space.

We multiply $\mM$ with a common denominator, using $O(r^2)$ multiplications and no divisions. Write $\mM_{ij}=(p_{ij},q_{ij})$ with $p_{ij}\in\Z$ and $q_{ij}\in\Z_{>0}$, and let $\Delta\coloneqq\prod_{k,l}q_{kl}$. We cannot compute $(\Delta\mM)_{ij}$ as $p_{ij}\cdot\Delta/q_{ij}$, since the exact division $\Delta/q_{ij}$ is not allowed in the model of \Cref{sec:comp-mod}. Instead, we set $(\Delta\mM)_{ij}\coloneqq p_{ij}\Delta_{-ij}$, where $\Delta_{-ij}\coloneqq\prod_{(k,l)\ne(i,j)}q_{kl}$. Computing each $\Delta_{-ij}$ separately would take $O(r^4)$ multiplications. Instead, we list the $r^2$ entries in a fixed order and compute all prefix and all suffix products of the $q_{kl}$ in this order, using $O(r^2)$ multiplications. Then, $\Delta_{-ij}$ is the product of the prefix ending before $(i,j)$ and the suffix starting after it. Both $\Delta$ and every entry of $\Delta\mM$ have size $O(\size(\mM))$, so $\size(\Delta\mM)=O(r^2\size(\mM))$.

The outputs for $\mM$ are then $\det\mM=\det(\Delta\mM)/\Delta^r$ and $\adjt(\mM)=\adjt(\Delta\mM)/\Delta^{r-1}$, obtained with $O(r)$ multiplications for the powers of $\Delta$. By the size bounds above, bounds of the form $\poly(n)\size(\Delta\mM)$ for the computation on $\Delta\mM$ translate into $\poly(n)\size(\mM)$. The reduction thus takes $O(n^2)$ operations, which is dominated by the bound below. For the rest of the proof, we therefore assume $\mM\in\Z^{r\times r}$.

As all costs in \Cref{def:tasks} are monotone in $n$, it suffices to consider $r=n$.
With the $c_i$ of \Cref{prop:leverrier}, set $C_i\coloneqq(n-i)!\cdot c_i$ and $\mD^\star\coloneqq(n-1)!\cdot\adjt(\mM)$. Then $C_n = 1$ and multiplying \eqref{eq:newton} by $(n-i-1)!$ and \eqref{eq:cayley} by $(n-1)!$, and substituting $c_k=C_k/(n-k)!$ in both, puts the two identities in the division-free form
\begin{align}
\label{eq:newton-integral}C_i&=-\sum_{j=1}^{n-i}\frac{(n-i-1)!}{(n-i-j)!}\,C_{i+j}\,s_j\qquad(0\le i\le n-1)\, ,\\
\label{eq:adj-integral}\mD^\star&=(-1)^{n+1}\sum_{i=1}^{n}\frac{(n-1)!}{(n-i)!}\,C_i\,\mM^{\,i-1}\, ,
\end{align}
in which every coefficient is a falling factorial, hence an integer of size polynomial in $n$. These coefficients are all of the form $a!/(a-j)!=a(a-1)\cdots(a-j+1)$ with $0\le j\le a\le n-1$, and we can compute all of them in $O(n^2)$ multiplications.

The powers $\mM^2,\ldots,\mM^n$ and with them the traces $s_j$ cost $O(n)$ matrix multiplications and $O(n^2)$ additions. Since $\|\mM^n\|_\infty\le n^n\|\mM\|_\infty^n$ we have $\size(\mM^n)=\poly(n,\size(\mM))$, and hence $\size(s_j)=\poly(n,\size(\mM))$ as well.

Evaluating \eqref{eq:newton-integral} for $i=n-1,\ldots,0$ in this order therefore takes $O(n^2)$ multiplications and additions of integers of strongly polynomial size, and \eqref{eq:adj-integral} a further $O(n^3)$. Finally, $\det\mM$ is returned as the pair $\bigl((-1)^nC_0,\,n!\bigr)$ and $\adjt(\mM)$ as $\bigl(\mD^\star,(n-1)!\bigr)$.

As $\Tmul(n)\ge n^2$, the matrix multiplications dominate, which gives the result.
\end{proof}
Note that for integral $\mM$, $\det\mM$ is returned as the pair $\bigl((-1)^nC_0,\,n!\bigr)$, although its value is an integer. Computing this integer would require an exact division by $n!$, which is not allowed in our model.

\medskip

We now reduce $\Tinv$, $\Tbasis$ and $\Tcirc$ to the primitives $\Tmul$, $\Tdet$ and $\Tadj$. These constructions are classical; we include them to show that they can be carried out in the model of \Cref{sec:comp-mod}.

\begin{corollary}[Matrix inversion]\label{cor:tinv}
$\Tinv(n)=O\bigl(\Tadj(n)+\Tdet(n)\bigr)$.
\end{corollary}
\begin{proof}
For invertible $\mM \in \eQ^{r \times r}$ with $r\le n$ this is immediate from Cramer's rule $\mM^{-1}=\adjt(\mM)/\det\mM$.
\end{proof}

A system $\mM z= h$ with $\mM$ invertible can therefore be solved as $z=\mM^{-1}h$ in $\Tinv(n)+O(n^2)$ operations.

\begin{lemma}[Basis computation]\label{lem:basis}
 $\Tbasis(n)=O\bigl(n(\Tmul(n)+\Tdet(n))\bigr)$.
\end{lemma}
\begin{proof}
Let $\mM \in \eQ^{m\times n}$ with $m\le n$.
Note that for $\mN\in\R^{m\times r}$ the matrix $\mN^\top\mN$ is nonsingular if and only if $\mN$ has full column rank. Hence, a single pass over the columns suffices: start with $B\coloneqq\emptyset$ and for $j=1,\ldots,n$ adjoin $j$ to $B$ whenever
\[
\det\bigl(\mM_{B\cup\{j\}}^\top\mM_{B\cup\{j\}}\bigr)\ne0\, ,
\]
which by the above happens exactly when $\mM_j$ is linearly independent of $\mM_B$. Each column costs one matrix product of matrices of dimension at most $n$ and one determinant of order $|B|+1\le m+1$. Clearly, the size of every number occurring remains bounded by $\poly(n)\size(\mM)$.
\end{proof}

Let us now consider the computation of circuits.

\begin{lemma}[Circuits]\label{lem:circuit}
\begin{equation}\label{eq:tc}
\Tcirc(n)= O\left(\Tbasis(n)+\Tmul(n)+\Tinv(n)\right).
\end{equation}
\end{lemma}
\begin{proof}
One call to \Cref{lem:basis} returns a basis $B\subseteq[n]$ of $\mM$; if $B=[n]$ the columns of $\mM$ are linearly independent and we report this. Otherwise, pick any nonbasic column $t\in[n]\setminus B$. If $B=\emptyset$, then $\mM=\0$, so $\{t\}$ is a circuit and we return $g=-e_t$. Otherwise, the vector $g_t = -1$, $g_B = (\mM_B^\top \mM_B)^{-1} \mM_B^\top \mM_t$, $g_{[n] \setminus (B \cup \{t\})} = \0$ is the elementary vector corresponding to the fundamental circuit of $t$ with respect to $B$. Using the expression above, it can be computed with a single call to $\Tinv(n)$ and $\Tmul(n)$ with an overhead of $O(n^2)$ for the additional matrix-vector multiplications. The overhead is absorbed by the other operations, as $\Tmul(n)\ge n^2$.
\end{proof}

\paragraph{Pseudoinverses.}

We can apply these primitives to compute the Moore--Penrose pseudoinverse. As in \Cref{def:tasks}, we assume $m\le n$. If $m>n$, we apply the lemma to $\mM^\top$ and use $(\mM^\top)^+=(\mM^+)^\top$; the cost is the same with the roles of $m$ and $n$ exchanged.

\begin{lemma}[Pseudoinverse computation]\label{lem:pseudoinverse-compute}
Let $\mM\in\eQ^{m\times n}$ with $m\le n$. One can compute the Moore--Penrose pseudoinverse $\mM^+\in\eQ^{n\times m}$ in
\begin{equation}\label{eq:tpinv}
O\bigl(\Tbasis(n)+\Tinv(n)+\Tmul(n)\bigr)
\end{equation}
arithmetic operations and comparisons, every number occurring having size $\poly(n)\size(\mM)$.
\end{lemma}
\begin{proof}
For $\rank(\mM) = 0$ the statement is trivial, so assume $\rank(\mM) \ge 1$.
Apply \Cref{lem:basis} to $\mM$ to obtain a basis $B\subseteq[n]$ and put $\mF\coloneqq\mM_{B}\in\eQ^{m\times r}$ with $r=\rank(\mM)\in \{1, \dots, m\}$, so that $\mF$ has full column rank and $\mF^\top\mF$ is nonsingular. Set $\mG\coloneqq(\mF^\top\mF)^{-1}\mF^\top\mM\in\eQ^{r\times n}$. Since $\mF(\mF^\top\mF)^{-1}\mF^\top$ is the orthogonal projection onto $\im(\mF)=\im(\mM)$, it fixes every column of $\mM$, that is, $\mF\mG=\mM$. Further, $\mG_B=(\mF^\top\mF)^{-1}\mF^\top\mF=\mI_r$, so $\mG$ has full row rank and $\mM=\mF\mG$ is a full-rank factorization. For such a factorization $\mM^+=\mG^+\mF^+$, while $\mF^+=(\mF^\top\mF)^{-1}\mF^\top$ and $\mG^+=\mG^\top(\mG\mG^\top)^{-1}$ by full column and full row rank respectively.

Beyond the single basis computation, this is a constant number of matrix products together with the two inversions of $\mF^\top\mF$ and $\mG\mG^\top$, both of order $r\le n$. The size bounds are immediate.
\end{proof}

\subsection{Basis crashing}\label{sec:mw-support}

The first of the two polyhedral tasks is \emph{basis crashing}, a term coined by Beling and Megiddo~\cite{belingmegiddo98}.

\begin{definition}[Basis crashing]\label{def:task-crash}
$\Tcr(n)$ is the cost of the following task: given $\mA\in\eQ^{m\times n}$ with $m\le n$, $b\in\eQ^m$ and a point $y\in\cal P\cap\eQ^n$ where $\cal P\coloneqq\{x\in\R^n\mid\mA x=b,\ x\ge\0_n\}$, return a vertex $v$ of $\cal P$ with $\supp(v)\subseteq\supp(y)$.
\end{definition}
In this section we use the linear algebra subroutines introduced in the previous section to find vertices $v$ whose support is contained in the support of a given point $y$. The standard approach is to move from $y$ along an elementary vector $g$ with $\supp(g)\subseteq\supp(y)$ until a variable becomes zero, and to repeat. If the iterates are computed from one another, their encoding lengths may grow exponentially, as in a naive implementation of Gaussian elimination. Following \cite[Theorem 6.5.11]{gls}, we avoid this by computing every iterate directly from $y$, by solving a single linear system. We use the same technique again in \Cref{lem:mw-polytope}.

\begin{lemma}[Basis crashing]\label{lem:crash}
\begin{equation}\label{eq:tcrash}
\Tcr(n)=O\bigl(n(\Tcirc(n)+\Tinv(n)+\Tmul(n))\bigr)\, ,
\end{equation}
and with $\mA,b,y$ as in \Cref{def:task-crash}, every number occurring has size $\poly(n)\bigl(\size(\mA,b)+\size(y)\bigr)$, while the vertex $v$ returned has $\size(v)=\poly(n)\size(\mA,b)$.
\end{lemma}
\begin{proof}
Set $y^{(1)}\coloneqq y$ and $S_1\coloneqq\supp(y^{(1)})$, and repeat for $i=1,2,\ldots$ the following round. For each $i\ge 2$ it will hold that for $S_i \coloneqq \supp(y^{(i)})$ we have $S_i\subsetneq S_{i-1}$ and $\mA y^{(i)}=b$. Call the circuit routine \ref{it:task-circuit} of \Cref{def:tasks} on the submatrix $\mA_{S_i}$. If it reports the columns of $\mA_{S_i}$ to be linearly independent, stop and return $v\coloneqq y^{(i)}$. To obtain the desired size bound of $\size(\mA,b)$, note that we need to recompute $v$ as the unique feasible solution to $\mA x = b$ supported on $S_i$. That is, we solve the system $\mA_{S_i} v_{S_i} = b$ for $v_{S_i}$ and set $v_{[n]\setminus S_i} = 0$.
Otherwise, the call to \ref{it:task-circuit} returns a circuit of $\mA_{S_i}$, which padded by zeros outside $S_i$ is a vector $g^{(i)}\in\ker(\mA)\setminus\{\0_n\}$ with $\supp(g^{(i)})\subseteq S_i$. Replacing $g^{(i)}$ by $-g^{(i)}$ if necessary, we may assume that it has a negative coordinate. As $y^{(i)}_j>0$ for every $j\in S_i\supseteq\supp(g^{(i)})$, set $\tau_i$ to be the maximal step length such that $y^{(i+1)}\coloneqq y^{(i)}+\tau_i g^{(i)}$ remains feasible. The new support $S_{i+1} \coloneqq \supp(y^{(i+1)})$ satisfies that $S_{i+1}\subsetneq S_i$.

We do not compute $y^{(i+1)}$ from $y^{(i)}$, but directly from $y$, as follows. Let $\mG$ be the matrix with columns $g^{(1)},\ldots,g^{(i)}$. Let $R\coloneqq S_1\setminus S_{i+1}$. By construction $y^{(i+1)}=y+\mG\tau$ for $\tau=(\tau_1,\ldots,\tau_i)$, and $y^{(i+1)}_R=\0$. We claim that $\tau$ is the unique solution of $(\mG\alpha)_R=-y_R$. For $t\le i$ pick $p_t\in S_t\setminus S_{t+1}\subseteq R$. Then $g^{(t)}_{p_t}=-y^{(t)}_{p_t}/\tau_t\ne 0$, whereas $g^{(u)}_{p_t}=0$ for $u>t$ since $\supp(g^{(u)})\subseteq S_u\subseteq S_{t+1}$. Hence the rows $p_1,\ldots,p_i$ of $\mG$ form a triangular matrix with nonzero diagonal, so $\mG_R$ has full column rank and the solution is unique. We compute it as $\alpha=-(\mG_R^\top\mG_R)^{-1}\mG_R^\top y_R$.
We can therefore obtain $\alpha$ with one call each to $\Tinv(n)$ and $\Tmul(n)$ and $O(n^2)$ further operations. The computation gives $\size(\alpha)=\poly(n)\size(\mG, y) = \poly(n) \size(\mA, y)$ and consequently $\size(y^{(i+1)})= \poly(n)(\size(\mA, y))$.

Since $S_{i+1}\subsetneq S_i$, there are at most $n$ rounds. Each round uses one call to $\Tcirc(n)$, $O(n)$ operations for the ratio test, and one call each to $\Tinv(n)$ and $\Tmul(n)$ to recompute $y^{(i+1)}$. In the final round, $\mA_{S_i}$ has full column rank, so $v_{S_i}=(\mA_{S_i}^\top\mA_{S_i})^{-1}\mA_{S_i}^\top b$ is obtained with one more call each to $\Tinv(n)$ and $\Tmul(n)$, and $\size(v)=\poly(n)\size(\mA,b)$. This gives \eqref{eq:tcrash}.
\end{proof}

\subsection{Polytopes}\label{sec:mw-polytope}

We are ready to compute Carath\'eodory decompositions.
\begin{lemma}[Carath\'eodory decomposition in standard form]\label{lem:mw-polytope}
Let $\mA\in\eQ^{m\times n}$, $m \le n$,  and $b\in\eQ^m$ be such that $\cal P\coloneqq\{x\in\R^n\mid \mA x=b,\ x\ge\0_n\}$ is a nonempty polytope, and let $y\in\cal P\cap\eQ^n$. There is a strongly polynomial algorithm returning vertices $v^{(1)},\ldots,v^{(k)}$ of $\cal P$ and multipliers $\lambda_1,\ldots,\lambda_k\in\eQplusplus$ such that
$v^{(1)},\ldots,v^{(k)}$ are affinely independent,
\[\sum_{i=1}^k\lambda_i=1\, ,\qquad y=\sum_{i=1}^k\lambda_i v^{(i)}\, ;
\]
in particular $k\le n-\rank(\mA)+1$. It performs $O\bigl(n(\Tcr(n)+\Tinv(n))\bigr)$ arithmetic operations and comparisons. The sizes of all numbers occurring during the algorithm are bounded by $\poly(n)\size(\mA, b, y)$.
\end{lemma}
\begin{proof}
Set $y^{(1)}\coloneqq y$ and $S_i\coloneqq\supp(y^{(i)})$, and repeat for $i=1,2,\ldots$ the following. Apply \Cref{lem:crash} to $\cal P$ and $y^{(i)}$ to obtain a vertex $v^{(i)}$ of $\cal P$ with $\supp(v^{(i)})\subseteq S_i$. If $v^{(i)}=y^{(i)}$, we stop and set $k\coloneqq i$. Otherwise, we move away from $v^{(i)}$ along the direction $h^{(i)}\coloneqq y^{(i)}-v^{(i)}\ne\0_n$, which satisfies $\mA h^{(i)}=\0_m$ and $\supp(h^{(i)})\subseteq S_i$. Let $\tau_i$ be the maximal step length in this direction, i.e., $y^{(i)} + \tau_i h^{(i)} \in \cal P$, but $y^{(i)} + (\tau_i+\epsilon) h^{(i)} \notin \cal P$ for any $\epsilon > 0$. By the assumption that $\cal P$ is bounded we have that $\tau_i$ is finite. Let $y^{(i+1)} \coloneqq  y^{(i)} + \tau_i h^{(i)}$.
Exactly as in the proof of \Cref{lem:crash}, we do not form $y^{(i+1)}$ recursively, but re-derive it from $y$ as $y^{(i+1)} = y + \mH^{(i)} \alpha$, for $\mH^{(i)} = [\,h^{(1)}\ \cdots\ h^{(i)}\,]$ where $\alpha$ is the unique solution to
 \begin{equation}
 [y + \mH^{(i)} \alpha]_{S_1 \setminus S_{i+1}} = \0_{S_1 \setminus S_{i+1}}\, .
 \end{equation}

 Uniqueness of $\alpha$ follows as in the proof of \Cref{lem:crash}: for $p_t\in S_t\setminus S_{t+1}$ we have $h^{(t)}_{p_t}=-y^{(t)}_{p_t}/\tau_t\ne0$ and $h^{(u)}_{p_t}=0$ for $u>t$, so $\mH^{(i)}_{S_1\setminus S_{i+1}}$ has full column rank.
 As $\size(\mH^{(i)}) = \poly(n) \size(\mA, b,y)$, we obtain that $\size(\alpha) = \poly(n) \size(\mA, b,y)$ and consequently $\size(y^{(i+1)}) = \poly(n) \size(\mA, b,y)$.

The same indices $p_t$ show that $v^{(1)},\ldots,v^{(k)}$ are affinely independent. For $t<k$ we have that $v^{(t)}_{p_t} > 0$, whereas $v^{(u)}_{p_t}=0$ for $u>t$ since $\supp(v^{(u)})\subseteq S_u\subseteq S_{t+1}$. Hence, $v^{(t)}$ is not contained in the affine hull of $v^{(t+1)}, \dots, v^{(k)}$, implying that $v^{(1)},\ldots,v^{(k)}$ are affinely independent.

 Finally, compute $\lambda$ as the unique solution to the augmented linear system
 \begin{equation}\label{eq:mw-barycentric}
 \begin{pmatrix}v^{(1)}&\cdots&v^{(k)}\\1&\cdots&1\end{pmatrix}\lambda=\binom{y}{1}\, .
 \end{equation}
 Existence follows by construction of the $v^{(i)}$: since $y^{(i+1)}=y^{(i)}+\tau_i(y^{(i)}-v^{(i)})$ with $\tau_i>0$, we have $y^{(i)}=\frac{\tau_i}{1+\tau_i}v^{(i)}+\frac{1}{1+\tau_i}y^{(i+1)}$ for $i<k$, and $y^{(k)}=v^{(k)}$. Unrolling these identities from $y=y^{(1)}$ expresses $y$ as a convex combination of $v^{(1)},\ldots,v^{(k)}$ with all coefficients $\lambda_i>0$ and $\sum_{i=1}^k\lambda_i=1$. Uniqueness follows by the affine independence of $v^{(1)},\ldots,v^{(k)}$, which is equivalent to the coefficient matrix of \eqref{eq:mw-barycentric} having full column rank. As $\lambda$ is the unique solution to a linear system with data $y$ and $v^{(1)}, \ldots, v^{(k)}$ we have that $\size(\lambda) = \poly(n) \size(v^{(1)},\ldots,v^{(k)},y) = \poly(n) \size(\mA, b,y)$.

Both systems can be solved by inverting a square subsystem that is triangular with nonzero diagonal by the arguments above: for $\alpha$ the rows $p_1,\ldots,p_i$, and for $\lambda$ the rows $p_1,\ldots,p_{k-1}$ together with the last row.
Each of the at most $n$ rounds costs $\Tcr(n)$ for the basis crash, $O(n)$ for the ratio test to compute the $\tau_i$ and one call to $\Tinv(n)$. The final solve for $\lambda$ requires one more linear system solve, which is a total of $O\bigl(n(\Tcr(n)+\Tinv(n))\bigr)$.
\end{proof}
\subsection{Cones and general polyhedra}\label{sec:mw-cone}
The conic case reduces to the bounded one by intersecting with a hyperplane through the given point.

\begin{corollary}[Conic decomposition in standard form]\label{cor:mw-cone}
Let $\mA\in\eQ^{m\times n}$, $m\le n$, $\cal C\coloneqq\{x\in\R^n\mid \mA x=\0_m,\ x\ge \0_n\}$ and $y\in\cal C\cap\eQ^n$ with $y\ne\0_n$. There is an algorithm returning linearly independent vectors $r^{(1)},\ldots,r^{(\ell)}\in\eQ^n$, each a nonzero generator of an extreme ray of $\cal C$, with $y=\sum_{j=1}^\ell r^{(j)}$, $\ell\le n$. It performs $O\bigl(n(\Tcr(n+1)+\Tinv(n+1))\bigr)$ operations and comparisons, and the size of all numbers occurring is bounded by $\poly(n)\size(\mA, y)$.
\end{corollary}
\begin{proof}
Let $\sigma\coloneqq\1_n^\top y>0$. The set
\[
\cal Q\coloneqq\{x\in\R^n\mid \mA x=\0_m,\ \1_n^\top x=\sigma,\ x\ge\0_n\}\ni y
\]
is in standard form with $m+1\le n+1$ constraints, with constraint matrix $\binom{\mA}{\1_n^\top}$ and right-hand side $\binom{\0_m}{\sigma}$ of size $O(\size(\mA,y))$, and it is bounded, since $x\ge\0_n$ and $\1_n^\top x=\sigma$ force $0\le x_i\le\sigma$.
Its vertices are precisely the extreme rays of $\cal C$ normalized to $\1_n^\top x=\sigma$.

\Cref{lem:mw-polytope} applied to $\cal Q$ and $y$ yields affinely independent vertices $w^{(1)},\ldots,w^{(\ell)}$ of $\cal Q$ and $\lambda\in\eQplusplus^\ell$ with $\sum_j\lambda_j=1$ and $y=\sum_j\lambda_jw^{(j)}$. The $w^{(j)}$ are linearly independent: if $\sum_j a_jw^{(j)}=\0_n$, then applying $\1_n^\top$ gives $\sigma\sum_j a_j=0$, so $\sum_j a_j=0$, and affine independence yields $a_1 = \dots = a_\ell = 0$. Setting $r^{(j)}\coloneqq\lambda_jw^{(j)}$ gives the desired vectors.
\end{proof}

We are now ready to prove \Cref{lem:minkowski-weyl}.

\begin{definition}[Minkowski--Weyl decomposition]\label{def:task-mw}
$\TMW(n)$ is the cost, in the sense of \Cref{def:task-crash}, of the following task: given $\mA\in\eQ^{m\times n}$ with $m\le n$, $b\in\eQ^m$ and a point $x\in\cal P\cap\eQ^n$ where $\cal P\coloneqq\{x\in\R^n\mid\mA x=b,\ x\ge\0_n\}$, compute a decomposition of $x$ as in \Cref{lem:minkowski-weyl}.
\end{definition}

\begin{theorem}\label{thm:mw-general}
\Cref{lem:minkowski-weyl} holds with $\TMW(n)=O\bigl(n(\Tcr(n+1)+\Tinv(n+1))\bigr)$ arithmetic operations and comparisons. Every vertex returned has size $\poly(n)\size(\mA,b)$, and every extreme ray returned has size $\poly(n)\size(\mA,b,x)$. Every number occurring during the algorithm has size $\poly(n) \size(\mA,b,x)$.
\end{theorem}

\begin{proof}
Write $y\coloneqq x$ for the input point and run the algorithm of \Cref{lem:mw-polytope} on $\cal P$. Since $\mA h^{(i)}=\0_m$, the point $y^{(i)}+\tau h^{(i)}$ lies in $\cal P$ for every $\tau\ge0$ if and only if $h^{(i)}\ge\0_n$; that is, the ratio test in round $i$ fails to find a finite $\tau_i$ exactly when $h^{(i)}$ is nonnegative. If no direction $h^{(i)}$ is nonnegative, the algorithm runs unchanged and returns the decomposition of \Cref{lem:mw-polytope}, with $\ell=0$ and $k\le n+1$. Otherwise, let $k$ be the first round with $h^{(k)}\ge\0_n$; we stop in this round. Now $h^{(k)}\ne\0_n$ lies in the recession cone $\cal C\coloneqq\{x\in\R^n\mid\mA x=\0_m,\ x\ge\0_n\}$ of $\cal P$, whose extreme rays are exactly those of $\cal P$, and one call to \Cref{cor:mw-cone} returns linearly independent extreme rays $r^{(1)},\ldots,r^{(\ell)}$ of $\cal P$ with $h^{(k)}=\sum_{j=1}^\ell r^{(j)}$.

As in the proof of \Cref{lem:mw-polytope}, $y$ is a convex combination of $v^{(1)},\ldots,v^{(k-1)},y^{(k)}$ with positive coefficients $\lambda\in\eQplusplus^k$, which we compute as the unique solution of the augmented system \eqref{eq:mw-barycentric} with $y^{(k)}$ in place of $v^{(k)}$. Substituting $y^{(k)}=v^{(k)}+\sum_{j}r^{(j)}$ gives the asserted decomposition with $\eta_j\coloneqq\lambda_k>0$ for all $j\in[\ell]$. The linear independence of $(v^{(1)},1),\ldots,(v^{(k)},1),(r^{(1)},0),\ldots,(r^{(\ell)},0)$ in $\R^{n+1}$, and thus $k+\ell\le n+1$, follows analogously to \Cref{lem:mw-polytope} and \Cref{cor:mw-cone}.

For the sizes, $\size(v^{(i)})=\poly(n)\size(\mA,b)$ by \Cref{lem:crash}, and $\size(h^{(k)})=\poly(n)\size(\mA,b,y)$ by \Cref{lem:mw-polytope}, so every number in the call to \Cref{cor:mw-cone}, and in particular every $r^{(j)}$, has size $\poly(n)\size(\mA,h^{(k)})=\poly(n)\size(\mA,b,y)$. For the operation count, the at most $n$ rounds cost $O\bigl(n(\Tcr(n)+\Tinv(n))\bigr)$ in total, as in \Cref{lem:mw-polytope}. Solving for $\lambda$ costs one call to $\Tinv(n+1)$ and $O(n^2)$ further operations; and the call to \Cref{cor:mw-cone} costs $O(n(\Tcr(n+1)+\Tinv(n+1)))$. This gives the claimed bound.
\end{proof}
Combining the reductions of \Cref{sec:mw-cost} turns \Cref{thm:mw-general} into an explicit bound using the naive cubic bound $\Tmul(n)=O(n^3)$ on matrix multiplication of \Cref{prop:tmul}.
\begin{corollary}[Explicit running time and space]\label{cor:mw-explicit-bounds}
Assuming $\Tmul(n)=O(n^3)$, the reductions of \Cref{cor:leverrier-integral,cor:tinv,lem:basis,lem:circuit,lem:crash} and \Cref{thm:mw-general} give
\begin{equation}\label{eq:chain}
\begin{aligned}
\Tdet(n),\ \Tadj(n),\ \Tinv(n)&=O(n^4)\, ,&\qquad \Tbasis(n),\ \Tcirc(n)&=O(n^5)\, ,\\
\Tcr(n)&=O(n^6)\, ,&\qquad \TMW(n)&=O(n^7)\, .
\end{aligned}
\end{equation}
Consequently, for $\mA\in\eQ^{m\times n}$ with $m\le n$, $b\in\eQ^m$ and $x\in\cal P\cap\eQ^n$, the decomposition of \Cref{lem:minkowski-weyl} is computed in $O(n^7)$ arithmetic operations and comparisons on rationals. Every vertex returned has size $\poly(n)\size(\mA,b)$, every extreme ray and every multiplier returned has size $\poly(n)\size(\mA,b,x)$, and the total space used by the algorithm is $\poly(n)\size(\mA,b,x)$.
\end{corollary}
\begin{proof}
For the running time, substitute $\Tmul(n)=O(n^3)$ into the chain of reductions. \Cref{cor:leverrier-integral} gives $\Tdet(n),\Tadj(n)=O(n\,\Tmul(n))=O(n^4)$, and \Cref{cor:tinv} then gives $\Tinv(n)=O(\Tadj(n)+\Tdet(n))=O(n^4)$. \Cref{lem:basis} gives $\Tbasis(n)=O\bigl(n(\Tmul(n)+\Tdet(n))\bigr)=O(n^5)$, and \Cref{lem:circuit} $\Tcirc(n)=O\bigl(\Tbasis(n)+\Tmul(n)+\Tinv(n)\bigr)=O(n^5)$. \Cref{lem:crash} gives $\Tcr(n)=O\bigl(n(\Tcirc(n)+\Tinv(n)+\Tmul(n))\bigr)=O(n^6)$, and finally \Cref{thm:mw-general} gives $\TMW(n)=O\bigl(n(\Tcr(n+1)+\Tinv(n+1))\bigr)=O(n^7)$, as claimed in \eqref{eq:chain}.

For the space bound, the algorithm performs $O(n^7)$ operations, each producing one number, and every number occurring has size $\poly(n)\size(\mA,b,x)$ by \Cref{thm:mw-general}. Hence at most $O(n^7)$ numbers of size $\poly(n)\size(\mA,b,x)$ are stored at any time, which gives the claim.
\end{proof}

In particular, \Cref{thm:mw-general,cor:mw-explicit-bounds} prove \Cref{lem:minkowski-weyl}. We also collect the bounds stated in \Cref{lem:linalg}.

\begin{proof}[Proof of \Cref{lem:linalg}]
Part \ref{it:linalg-basis} is \Cref{lem:basis}, and part \ref{it:linalg-inv} is \Cref{cor:leverrier-integral,cor:tinv}, with the bounds $\Tbasis(n)=O(n^5)$ and $\Tdet(n),\Tadj(n),\Tinv(n)=O(n^4)$ of \eqref{eq:chain}. For part \ref{it:linalg-pinv}, if $m\le n$, then \Cref{lem:pseudoinverse-compute} and \eqref{eq:chain} give $O(n^5)$ operations; if $m>n$, we apply it to $\mM^\top$ and use $(\mM^\top)^+=(\mM^+)^\top$. The size bounds are part of these statements and their proofs.
\end{proof}

\end{document}